\documentclass[11pt]{amsart}

\usepackage[inner=2.4cm,outer=2.4cm,top=2.4cm,bottom=2.4cm]{geometry}
\usepackage[english]{babel}
\usepackage{amsmath} 
\usepackage{amsthm,todonotes}
\usepackage{amsfonts}
\usepackage{amssymb}
\usepackage{graphicx}
\usepackage{mathrsfs}

\usepackage{color,enumerate,mathtools}
\usepackage[dvipsnames, svgnames]{xcolor}
\usepackage[normalem]{ulem}
\usepackage{amsmath,epsfig}
\usepackage{pdfsync}
\usepackage{tikz-cd}
\numberwithin{equation}{section}
\usepackage[colorlinks,citecolor=green,linkcolor=red]{hyperref}

\newcommand{\N}{\mathbb{N}}

\newcommand{\R}{\mathbb{R}}

\newcommand{\PP}{\mathscr{P}}

\newcommand{\cF}{{\ensuremath{\mathcal F}}}

\newcommand{\cI}{{\ensuremath{\mathcal I}}}

\newcommand{\mm}{{\mbox{\boldmath$m$}}}

\newcommand{\sfd}{{\sf d}}
\newcommand{\sfe}{{\sf e}}
\newcommand{\restr}[1]{\lower3pt\hbox{$|_{#1}$}}

\newcommand{\eps}{\varepsilon}  
\newcommand{\nchi}{{\raise.3ex\hbox{$\chi$}}}
\newcommand{\weakto}{\rightharpoonup}

\newtheorem{theorem}{Theorem}[section]

\newtheorem{corollary}[theorem]{Corollary}
\newtheorem{lemma}[theorem]{Lemma}
\newtheorem{proposition}[theorem]{Proposition}
\newtheorem{definition}[theorem]{Definition}

\newtheorem{remark}[theorem]{Remark}

\newcommand{\Lip}{\mathrm{Lip}}
\newcommand{\lip}{\mathrm{lip}}

\newcommand{\diam}{\mathrm{diam}}

\newcommand{\fr}{\hfill$\blacksquare$}   
\newcommand{\res}{\mathop{\hbox{\vrule height 7pt width .5pt depth 0pt
\vrule height .5pt width 6pt depth 0pt}}\nolimits} 
\newcommand{\rcd}{{\sf RCD}}
\newcommand{\CD}{{\sf CD}}

\renewcommand{\mm}{\mathfrak m}

\renewcommand{\limsup}{\varlimsup}
\renewcommand{\liminf}{\varliminf}
\renewcommand{\d}{{\rm d}}

\newcommand{\Geo}{{\sf Geo}}

\newcommand{\X}{{\rm X}}

\newcommand{\Z}{{\rm Z}}

\newcommand{\e}{\sfe}

\newcommand{\rmKe}{{\rm Ke}}
\newcommand{\rmCh}{{\rm Ch}}
\newcommand{\OptGeo}{{\rm OptGeo}}
\newcommand{\Opt}{{\rm Opt}}

\newcommand{\Xdm}{(\X,\sfd,\mm)}
\newcommand{\Xdmx}{(\X,\sfd,\mm,x)}
\newcommand{\Xdmxinf}{(\X_\infty,\sfd_\infty,\mm_\infty,x_\infty)}
\newcommand{\Xdmxn}{(\X_n,\sfd_n,\mm_n,x_n)}
\newcommand{\Comp}{{\rm Comp}}

\newcommand{\supp}{{\rm supp}}
\newcommand{\Ent}{{\rm Ent}}

\usepackage{ esint }
\usepackage{tcolorbox}
\title[]{Stability of local Riemannian Ricci curvature lower bounds}

\author[]{Luigi Ambrosio} 
\address{Scuola Normale Superiore, Piazza dei Cavalieri 7, 56126 Pisa, Italy}
\email{\url{luigi.ambrosio@sns.it}}

\author[]{Francesco Nobili} 
\address{Università di Napoli Federico II, Dipartimento di Matematica e Applicazioni, Via Cintia, Monte S. Angelo, 80126 Napoli (NA), Italy}
\email{\url{francesco.nobili@unina.it}}

\author[]{Federico Renzi} 
\address{Scuola Normale Superiore, Piazza dei Cavalieri 7, 56126 Pisa, Italy}
\email{\url{federico.renzi@sns.it}}

\author[]{Federico Vitillaro} 
\address{Scuola Normale Superiore, Piazza dei Cavalieri 7, 56126 Pisa, Italy}
\email{\url{federico.vitillaro@sns.it}}

\begin{document}
\begin{abstract}
    We establish the stability of local Riemannian Ricci curvature lower bounds along Gromov-Hausdorff convergence. A central part of our analysis is devoted to showing the stability of the parallelogram identity for weak gradients, obtained implementing the Lagrangian approach developed in \cite{NobiliRenziVitillaro25} in the local setting. As an application, we deduce the almost everywhere existence of Euclidean weak tangents. An important ingredient, of independent interest, is an effective local Evolution Variational Inequality along the heat flow on sufficiently small balls, with a remainder term depending on the rate of decay of the flow. This has applications to strong displacement convexity of the Entropy functional along local Wasserstein interpolations and to local essential nonbranching properties.
\end{abstract}
\maketitle

\setcounter{tocdepth}{1}
\tableofcontents
\section{Introduction}
A central theme in analysis and geometry is the study of shapes and their degeneration under curvature bounds. A fundamental observation of Gromov (\cite{Gromov07}) is that the manifolds $(M^d,g)$ with uniform Ricci curvature lower bounds form a pre-compact class with respect to Gromov-Hausdorff (GH) convergence. Understanding the resulting limit structures (Ricci limits) is therefore a natural problem with important geometric and analytic consequences.

A systematic study of Ricci limits was initiated by Cheeger and Colding in \cite{Cheeger-Colding96,Cheeger-Colding97I,Cheeger-Colding97II,Cheeger-Colding97III}. Among their contributions are the almost splitting theorem, extending the classical splitting theorem of Cheeger--Gromoll (\cite{CheegerGromoll71}), and the continuity of Neumann eigenvalues along GH convergence, confirming a conjecture of Fukaya (\cite{Fukaya87}). The latter result reveals in particular the stability of the first-order differential structure of manifolds under uniform Ricci curvature bounds.

\medskip 

Motivated by the Cheeger-Colding theory, and by simultaneous developments in optimal transport \cite{JKO98,McCann97,CorderoMcCannSc01,SturmVonRenesse09}, Lott-Villani and Sturm introduced independently a synthetic notion of Ricci curvature lower bounds for metric measure spaces in \cite{Lott-Villani09,Sturm06I,Sturm06II}. A metric measure space is a triple $\Xdm$, where $(\X,\sfd)$ is a metric space endowed with a Borel measure $\mm$. The curvature-dimension condition ${\sf CD}(K,N)$ synthetically encodes Ricci curvature bounded below by $K$ and dimension bounded above by $N$ through convexity properties of entropy functionals along Wasserstein geodesics (see \cite{Villani2016,Sturm24_Survey}).

A crucial feature of this notion is that, thanks to its optimal transport-based definition, it is stable under appropriate notions of GH convergence. As a byproduct, this nonsmooth class contains that of Ricci limit spaces, but this inclusion is strict, as it contains many more non-Riemannian geometries. This led the first-named author, together with Gigli and Savar\'e, to formulate in \cite{AmbrosioGigliSavare11-2} a Riemannian curvature-dimension condition ${\sf RCD}(K,\infty)$ enforcing a Riemannian-like Ricci curvature lower bound (and no dimension upper bound) by looking at contractivity properties of the heat flow. Later, the finite-dimensional class of ${\sf RCD}(K,N)$ spaces was proposed in \cite{Gigli12} by coupling the ${\sf CD}(K,N)$ condition with infinitesimal Hilbertianity. The latter is defined by looking at the Sobolev calculus $H^{1,2}(\X)$ over a metric measure space (\cite{Shan00,Hajlasz96,Cheeger00,AmbrosioGigliSavare11-3}, see also \cite{AmbrosioIkonenLucicPasqualetto24} for a detailed discussion) and requiring $H^{1,2}(\X)$ to be Hilbert, equivalently that weak parallelogram identity holds
\begin{equation}
   |\nabla (f+g)|^2 + |\nabla (f-g)|^2 = 2|\nabla f|^2 +2|\nabla g|^2,\qquad \mm\text{-a.e.},\label{parallelogram intro}
\end{equation}
for all $f,g \in H^{1,2}(\X)$, where $|\nabla f|$ is the weak gradient. We also recall the works \cite{BacherSturm10,AmbrosioGigliSavare12,AGMR15,AmbrosioMondinoSavare13-2,EKS15,CM16} establishing crucial results in the ${\sf RCD}$-theory (see \cite{AmbICM} for a more complete introduction to the topic). 

\medskip 

Even though the ${\sf CD}(K,N)$ condition is stable under GH convergence, proving the stability of the ${\sf RCD}$ class required showing that infinitesimal Hilbertianity is stable as well within the class (and in general, this might be false \cite{LucicPasqualettoRajala23}). 
This was deduced in \cite{AmbrosioGigliSavare11-2} with an approach based on the stability of the Evolution Variational Inequalities (EVI) for the heat flow, and in \cite{GMS15} by looking alternatively at the parallelogram identity, together with a comprehensive study of the various (mostly equivalent) notions of GH convergence appearing in the aforementioned literature. We also recall the work \cite{Gigli10} where a useful set-up was originally clarified (see also the recent \cite{GigliVincini24}). Notice that these stability results are highly nontrivial, as \eqref{parallelogram intro} involves the first-order differential structure of a metric measure space through the concept of weak gradient. This breakthrough stimulated further stability results for central objects in geometric analysis, including heat kernels (\cite{AmbrosioHondaTewodrose18}), regular Lagrangian flows (\cite{AmbrosioStraTrevisan17}), sharp geometric and functional constants (see, e.g.\ \cite{Honda15,Honda17,AmbrosioHonda17,AmbrosioHondaPortegies18,HondaKettererMondelloPeralesRigoni24,NobiliViolo24,NobiliViolo24_PZ,NobiliViolo26,Nobili24_overview,AntonelliBrueFogagnoloPozzetta22,AntonelliPasqualettoPozzettaSemola25,Pozzetta23}), and structural results in codimension one (\cite{AmbrosioBrueSemola19,BruPasSem19,BruPasSem21-constantcodimension}). We refer to \cite{Gigli23_working} for a comprehensive overview.

\medskip 
Turning to the goals of this work, we note that, unlike the local and extrinsic Cheeger-Colding theory, the synthetic conditions are intrinsically global. Even though local Riemannian curvature-dimension conditions can be easily formulated (see Section \ref{sec:CD loc defs}), establishing a corresponding stability theory is highly nontrivial.
 In this direction, \cite{HondaSun26} studied almost smooth ${\sf RCD}$ and locally ${\sf RCD}$ spaces, raising in \cite[Question 7.5]{HondaSun26} the question of whether stability results analogous to \cite{GMS15} should hold under suitable local assumptions along the sequences. A major obstacle is that much of the modern ${\sf RCD}$ theory relies on \cite{GMS15}, whose arguments crucially exploit the stability of the heat flow under GH convergence, a tool that does not readily extend to the local setting as it is global in nature.

Nevertheless, infinitesimal Hilbertianity is inherently local, being characterized by the identity \eqref{parallelogram intro}. This suggests that its stability should be accessible under local Ricci lower bounds. In this work, we rigorously verify this heuristic. Our main objectives are:
\begin{itemize}
    \item[a)] to introduce a local ${\sf CD}$ condition on open sets and establish the fundamental local Doubling, Poincaré, and GH-stability properties, in the spirit of \cite{Sturm06I,Sturm06II,Lott-Villani09};

    \item[b)] to introduce a local ${\sf RCD}$ condition and investigate local contraction properties of the heat flow, in the spirit of \cite{AmbrosioGigliSavare11,AmbrosioGigliSavare11-2};

    \item[c)] to prove the stability of the local ${\sf RCD}$ condition under GH convergence through a Lagrangian approach, substantially different from \cite{GMS15}.
\end{itemize}
\medskip 
\noindent\textbf{Statements of the main results}.
Given a metric measure space $\Xdm$ and an open set $\varnothing~\neq~\Omega\subset~\X$, we say that $\Omega$ satisfies ${\sf CD}_{loc}(K(\cdot),N(\cdot))$ if for all $x \in \Omega$ there is a radius $r(x)\in(0,+\infty]$, called \emph{local radius} of $\Omega$, so that the usual convexity of entropies as originally formulated in \cite{Lott-Villani09,Sturm06I,Sturm06II} holds, with parameter $K(x),N(x)$ and for every pair of probabilities $\mu_0,\mu_1 \in \PP(\X)$ with supports in $B_{r(x)/2}(x)$. If $(\X,\sfd)$ is locally complete, we shall further decrease $r(x)$ so that $(\overline{B}_r(x),\sfd)$ is a complete metric space. We again stress that $K(x)\in\R$, $N(x)\in [1,\infty)$ depend on the point $x$, as opposed to the local notion already considered in \cite{Sturm06II,BacherSturm10}. Additionally, we say that $\Omega$ satisfies ${\sf RCD}_{loc}(K(\cdot),N(\cdot))$ if it is also infinitesimally Hilbertian (equivalently, \eqref{parallelogram intro} holds for every $f,\,g$ with support well contained in $\Omega$, see Theorem \ref{thm:equivalent hilbert}). Similar definitions are given for open sets satisfying the infinite dimensional analogues ${\sf CD}_{loc}(K(\cdot),\infty)$ and ${\sf RCD}_{loc}(K(\cdot),\infty)$. See Section \ref{sec:CD loc defs} for the details.

Our first main result is that the ${\sf RCD}_{loc}(K(\cdot),N(\cdot))$ condition is stable for the pointed measured Gromov convergence (pmG for short, see Definition \ref{def:pmG} adopting the extrinsic approach of \cite{GMS15}), provided the defining curvature dimension parameters $K(\cdot),N(\cdot)$ and the local radius $r(\cdot)$ are uniformly bounded along the sequence.
\begin{theorem}\label{Main intro}
Let $(\X_n, \sfd_n, \mm_n, x_n)$ be a sequence of pointed complete metric measure spaces that are pmG-converging to a pointed complete limit space $(\X_\infty, \sfd_\infty, \mm_\infty, x_\infty)$, and let $(\Z,\sfd)$ be a realization of the convergence. Let $\Omega_n\subseteq \X_n,\Omega_\infty \subseteq \X_\infty$ be non-empty open sets with $\supp(\mm_n \res {\Omega_n})=\overline{\Omega}_n$. Suppose that for all $1 \leq n < \infty$ it holds that $\Omega_n$ satisfies ${\sf RCD}_{loc}(K_n(\cdot),N_n(\cdot))$ and that for all $x \in \Omega_\infty$ there are $y_n \in \Omega_n$ with $\sfd(y_n,x)\to 0$ such that, denoting $r_n(\cdot)$ the local radius of $\Omega_n$, it holds
\[
    \liminf_{n\to \infty} r_n(y_n) >0,\qquad K_\infty(x)\coloneqq \liminf_{n\to \infty} K_n(y_n)>-\infty,\qquad N_\infty(x)\coloneqq \limsup_{n\to \infty} N_n(y_n)<+\infty.
\]
Then $\Omega_\infty$ satisfies ${\sf RCD}_{loc}(K_\infty(\cdot),N_\infty(\cdot))$ with local radius
$r_\infty(x)=c \cdot \liminf_n r_n(y_n)$ for some universal constant $c>0$. Finally, if $r_\infty (x)=+\infty$ for some $x\in\Omega_\infty$, then $\X_\infty$ satisfies ${\sf RCD}(K,N)$ globally for $K=K_{\infty}(x)$, $N=N_{\infty}(x)$.
\end{theorem}
This analysis, and in particular the last conclusion, has important application in the to study the regularity of weak tangents under local curvature bounds. Indeed, combining our methods (especially the content of Section \ref{sec:nonnegative} that is tailored for nonnegative curvature) with the work \cite{GigliMondinoRajala15}, we can deduce the following result, referring to Definition \ref{Def:weakTan} for the relevant notation.
\begin{theorem}\label{thm:weak tangents intro}
    Let \((\X,\sfd,\mm)\) be a complete metric measure space, and let $\varnothing \neq \Omega\subseteq  \X$ be an open set satisfying ${\sf RCD}_{loc}(K(\cdot),N(\cdot))$. Then, for all $x\in\Omega \cap \supp(\mm)$ and all $(Y,\rho,\mu,y)\in {\rm Tan}(\X,\sfd,\mm,x)$, we have that $Y$ satisfies ${\sf RCD}(0,N(x))$. Furthermore, at $\mm$-almost every $x \in \Omega$ there exists $n \in \mathbb{N}$, $n \le N(x)$, such that
    \[
        (\mathbb{R}^n, \sfd_{\sf eu}, \mathcal{L}^n, 0) \in \mathrm{Tan}(\X, \sfd, \mm, x),
    \]
    where $\sfd_{\sf eu}$ is the Euclidean distance, and $\mathcal{L}^n$ is the Lebesgue measure normalized so that $\int_{B_1(0)} (1 - |x|) \, \d\mathcal{L}^n(x) = 1.$
\end{theorem} 
In the rest of this Introduction, we shall focus on Theorem \ref{Main intro}, that is obtained by combination of two steps:
\begin{itemize}
    \item we generalize the celebrated stability results of \cite{Lott-Villani09,Sturm95II,Sturm06II} for the convexity of entropies to the local setting (cf.\ Theorem \ref{thm:stability dimensional CDloc});
    \item we generalize the stability of the  Riemannian curvature dimension condition of \cite{AmbrosioGigliSavare11-2,GMS15}, and especially of the infinitesimal Hilbertianity, to the local setting.
\end{itemize}
We shall give emphasis on the second step that is the main challenge of our work and contains several ingredients of independent interest. First, to better highlight the main novelties, we briefly review the previously available strategy (see \cite{AmbICM,Gigli23_working} for a detailed discussion). One of the main results of \cite{GMS15} is the so-called Mosco convergence of the Cheeger energy functional
\[
    \rmCh(f) \coloneqq  \inf \Big\{ \liminf_{n\to\infty} \int \left(\lip  (f_n)\right)^2\, \d \mm \colon f_n\in L^2(\X)\cap \Lip(\X), f_n\to f \text{ in } L^2(\X) \Big\},
\]
along Gromov converging spaces. This directly yields the stability of infinitesimal Hilbertianity by standard reasoning (see \cite[Theorem 7.2]{GMS15}). As already remarked, the main ingredients in \cite{GMS15} seem difficult to generalize to local settings, as they rely on Eulerian and nonlocal arguments based on the stability of the heat flow.

In the present work, we instead rely on a Lagrangian argument along converging spaces recently developed by the second-to-last named authors in \cite{NobiliRenziVitillaro25}. This is effective in showing the Mosco convergence of Cheeger energies while avoiding, in its core argument, any Eulerian methods. If, on the one hand, this local result follows rather directly from \cite{NobiliRenziVitillaro25} when one assumes nonnegative curvature (see Theorem \ref{MainT K>=0}), most of our efforts are devoted to handling the case of locally bounded and possibly negative curvature.

Our second main result deal with the local Mosco convergence of the Cheeger energy under local Ricci lower bounds (see Definition \ref{def:L2 convergence} for the notions of $L^2$ strong/weak convergences).
\begin{theorem}\label{thm:mosco local intro}
    Let $(\X_n, \sfd_n, \mm_n, x_n)$ be a sequence of pointed complete metric measure spaces that are pmG-converging to a pointed complete limit space $(\X_\infty, \sfd_\infty, \mm_\infty, x_\infty)$, and let $(\Z,\sfd)$ be a realization of the convergence.
 Let $\Omega_n\subseteq \X_n,\Omega_\infty \subseteq \X_\infty$ be non-empty open sets with $\supp(\mm_n \res {\Omega_n})=\overline{\Omega}_n$.  Suppose that $\Omega_n$ satisfies ${\sf RCD}_{loc}(K_n(\cdot),N_n(\cdot))$ for all $1 \leq n < \infty$. Furthermore, suppose that there is $x \in \Omega_\infty$ and $y_n \in \Omega_n$ with $\sfd(y_n,x) \to 0$ such that 
    \[
        N_n(y_n)<+\infty,\quad \forall n \in\N,\qquad r(x)\coloneqq \liminf_{n\to\infty} r_n(y_n) >0,\qquad \liminf_{n\to\infty} K_n(y_n) > -\infty.
    \]
    Let $f_\infty \in L^2(\X_\infty)$ with $\supp(f_\infty)\Subset B\coloneqq B_{r}(x)$ for some $r<c \cdot r(x)$ where $c>0$ is a universal constant. 
    Then, it holds:
    \begin{enumerate}[(i)]
        \item there exists $f_n \in L^2(\X_n) \cap \Lip(\X_n)$ with $\supp(f_n) \subset B \cap \X_n$ that converges $L^2$-strong to $f_\infty$ such that
        \[
             \limsup_{n\to \infty} {\rm Ch}(f_n) \leq {\rm Ch}(f_\infty) .
        \] 
        \item for every sequence $f_n \in L^2(\mm_n)$ converging $L^2$-weak to $f_\infty \in L^2(\mm_\infty)$, it holds
        \[
            {\rm Ch}(f_\infty) \leq \liminf_{n\to\infty} {\rm Ch}(f_n).
        \]
    \end{enumerate}
    Finally, if $r(x)=+\infty$, then same conclusions hold for all $f_\infty \in L^2(\X_\infty)$ with no restriction on the support.
\end{theorem}
This result is implied by the combination of Proposition \ref{Gammalimsup}, dealing with (i) under no curvature assumptions (evident from \cite{GMS15,AmbrosioHonda17}, see also \cite[Theorem 4.4]{Gigli23_working}), and Theorem \ref{PolyProp} which settles property (ii). Notice that, differently from Theorem \ref{Main intro}, no asymptotic upper bounds on $N_n(y_n)$ are imposed. This is because Theorem \ref{PolyProp} works in the infinite-dimensional setting, but assumes a locally (possibly non uniform) doubling property that is automatically satisfied under the finite dimensionality assumption (see Corollary \ref{cor:dim_CDloc}). Hence $\limsup_{n\to \infty}N_n(y_n)=\infty$ is covered as well.

Let us discuss an important feature of the proof of Theorem \ref{thm:mosco local intro}, which is based on a Wasserstein polygonal approximation argument performed in \cite{NobiliRenziVitillaro25} (in the spirit of \cite{Lisini07}, see also \cite{GigliNobili22,NobiliPasqualettoSchultz22}). For this strategy to work under possible negative curvature, a key obstacle is to deduce the uniqueness of Wasserstein interpolations between absolutely continuous probability measures, which is typically guaranteed by a nonbranching property \cite{RajalaSturm12} (see also \cite{Deng25}). If, on the one hand, the work \cite{RajalaSturm12} is based on local arguments, its applicability is, on the other hand, possible only after establishing that the Shannon entropy 
\[
    {\rm Ent}_\mm(\mu)\coloneqq \int \rho \log (\rho)\,\d \mm,\qquad \text{if }\mu=\rho\mm,
\]
is \emph{strongly} semiconvex in the Wasserstein space, namely semiconvex along \emph{all} Wasserstein interpolations (enforcing this gives rise to the so-called \emph{strong} curvature dimension bounds). Unfortunately, the latter property is based on nonlocal arguments and leverages infinitesimal Hilbertianity exploiting the EVI for the heat flow \cite{AmbrosioGigliSavare11-2,AmbrosioGigliSavare11-3}, together with the abstract characterization results of strong geodesic convexity in \cite{DaneriSavare08} in terms of the EVI property. See \cite{AmbrosioGigliSavare08,MuratoriSavare20} for the gradient flow theory.

A third main result of our work is, therefore, the validity of the following version of the EVI property for the localized heat flow on sufficiently small neighborhoods (see Section \ref{sec:prelim} for the relevant notation and concepts in optimal transport).%
\begin{theorem}\label{th:EVI Intro}
Let $\Xdm$ be a complete metric measure space, let $\varnothing \neq \Omega\subseteq  \X$ be open, and assume that $\Omega$ satisfies ${\sf RCD}_{loc}(K(\cdot),\infty)$. Let $B\coloneqq B_{r}(x)$ be a ball such that $r < \frac{r(x)}6$ and $\frac{\d}{\d s} \mm(B_s(x))$ exists at $s =r$. Let $\eta=f\mm \in \PP(\overline{B})$, with $f \in L^2 \cap L^\infty(\overline{B})$, and let $(f_t)$ denote the $L^2(\overline{B})$-gradient flow of the Cheeger energy on $(\overline B,\sfd,\mm\res{\overline B})$ starting from $f$. Then, setting $\mu_t \coloneqq f_t \mm$, we have for all $\sigma \in \PP(\overline{B}_r(x))$ that
\begin{equation}\label{eq:intro EVI err}
    \Ent_\mm (\sigma)-\Ent_\mm(\mu_t)-\frac{K(x)}{2}W_2^2(\mu_t ,\sigma) \ge \frac 12 \frac\d{\d t}W_2^2(\mu_t, \sigma)-C\lim_{s\uparrow r}\|f_t\|_{L^\infty(B_r(x) \setminus B_{s}(x))},
\end{equation}
for a.e.\ $t>0$, where $C=C(r)=2r\frac{\d}{\d r} \mm(B_r(x))$.
\end{theorem}

In the above result, the heat flow $(f_t)$ is the standard Hilbertian gradient flow of a convex and lower semicontinuous functional on the Hilbert space $L^2(\overline B)$. However, it is also quadratic under the ${\sf RCD}_{loc}(K(\cdot),\infty)$ assumption, and this guarantees that this localized heat flow is a linear semigroup. Notice that the radius $r>0$ is sufficiently small depending on the local radius $r(x)$ arising from the curvature-dimension condition. The proof is then obtained by carefully revisiting arguments from \cite{AmbrosioGigliSavare11-2,AmbrosioGigliSavare11-3, AGMR15}, and our localization produces an error term $\|f_t\|_{L^\infty(B_r(x) \setminus B_{s}(x))}$ that depends on the rate of decay of the flow.

This novel EVI, which to the best of our knowledge is the first attempt to localize Hilbertian properties of the heat flow under curvature-dimension bounds, has two essential applications:
\begin{itemize}
    \item we deduce in Theorem \ref{th:RCD_implies_strongCD} the strong displacement convexity of the Shannon entropy for local Wasserstein interpolations;
    \item we obtain in Corollary \ref{final_cor} the local essential nonbranching property of the ${\sf RCD}_{loc}$ class.
\end{itemize}
As previously discussed, these ingredients are crucial to implement the Lagrangian strategy of \cite{NobiliRenziVitillaro25} and make it possible to achieve (ii) in Theorem \ref{thm:mosco local intro}. As a byproduct, we can thus show that $\Omega_\infty$ is infinitesimally Hilbertian in Theorem \ref{Main intro}.

\begin{remark}\label{rmk:locally complete intro}
\rm
  We can also consider locally complete metric measure spaces, and ask that the local radius $r(x)>0$ arising from the ${\sf CD}_{loc}$ condition is required to guarantee that $\overline{B}_{r(x)}(x)$ is complete. We will comment in Remark \ref{rmk:locally complete results} that both our main results Theorem \ref{Main intro} and Theorem \ref{thm:mosco local intro} extend to this generality. 
    \fr
\end{remark}

We conclude this Introduction by remarking that a further crucial challenge in obtaining these localized results is to show that the term $\|f_t\|_{L^\infty(B_r(x) \setminus B_{s}(x))}$ in \eqref{eq:intro EVI err} vanishes as $t \to 0$ if the starting function $f$ has support well contained in $B_s(x)$ (cf.\ Proposition \ref{prop:flux anulus vanish}). For this result to hold, we need to further assume that $\mm$ is locally doubling on $B_r(x)$ (recall Corollary \ref{cor:dim_CDloc} for our scope). This is the main reason why Theorem \ref{Main intro}  and Theorem \ref{thm:mosco local intro} were presented in a finite dimensional setting, even though they are directly implied by Theorem \ref{MainT} and Theorem \ref{PolyProp}  stated in an infinitesimal Hilbertian  \emph{strong} ${\sf CD}_{loc}(K(\cdot),\infty)$ setting. The doubling assumptions precisely guarantee that the same holds on an infinitesimal Hilbertian ${\sf CD}_{loc}(K(\cdot),\infty)$ setting, using that small enough balls are John domains so to show that the error term appearing in \eqref{eq:intro EVI err} vanishes asymptotically. This last step uses Gaussian upper bounds (cf.\ Theorem \ref{th:heatonOmega}) that have been systematically studied in connection with doubling and Poincar\'e assumptions in the abstract setting of Dirichlet forms starting from \cite{Sturm95II}, and more recently in \cite{Chen_et_al.21} under the validity of Nash-type Sobolev inequalities. The latter can be crucially proved to hold on the relevant sufficiently small balls in our proofs, as they are John domains supporting local Poincar\'e inequalities.
\section{Preliminaries}\label{sec:prelim}
\subsection{General notation}
 We use $\subset$ for the strict inclusion and $\subseteq$ for the weak inclusion. In addition, given a metric space $(\X,\sfd)$ and an open set $\varnothing \neq \Omega\subseteq \X$, we write $K\Subset \Omega$ provided ${\rm dist}(K,\Omega^c)>0$. When $\Omega=\X$, we canonically set ${\rm dist}(K,\Omega^c)=+\infty$ so that $K\Subset \X$ reduces to the standard weak inclusion. For each $r>0$ and $x\in \X$, we shall denote
\[
    B_r(x)\coloneqq \{y \in \X \colon \sfd(x,y)<r \},\qquad \overline{B}_r(x) \coloneqq \overline{B_r(x)},
\]
where $\overline{E}$ is the topological closure of a set $E\subseteq \X$. A metric space $(\X,\sfd)$ is locally complete if for all $x\in\X$ there exists $r_C(x) \in (0,+\infty]$ such that $\overline{B}_{r_C(x)}(x)$ is complete (the case $r_C(x)=+\infty$ corresponds to $(\X,\sfd)$ being complete).
A metric measure space is a triple \((\X,\sfd,\mm)\), where $(\X,\sfd)$ is a locally complete, separable metric space and $\mm\geq 0$ is a nonnegative Borel measure. We will always assume the following growth condition, which is meaningful when $(\X,\sfd)$ is unbounded, to handle infinite dimensional settings
\begin{equation}\label{eq:growth_cond}
 \int e^{-C\sfd(\cdot,x_0)^2} \d\mm < \infty \qquad \mbox{for some }C>0, x_0 \in \X.
\end{equation}
Notice that this automatically guarantees that $\mm$ is boundedly finite. Actually, in several parts of this note we shall deal with bounded metric spaces where the above is trivially satisfied: given a closed and bounded set $C\subseteq \X$, we consider the (restricted) metric measure space $(C,\sfd,\mm\res C)$. Here $\sfd$ is the distance induced on a subset, and $\mm\res C$ is the reference measure defined on the subspace $\sigma$-algebra.

For all $p\in [1,\infty)$, we denote as customary $L^p(\X),L^p_{loc}(\X)$ respectively the space of $p$-integrable functions and $p$-integrable functions on a neighbourhood of every point. Sometimes, to stress the measure we are considering, we shall write $L^p(\X,\mm)$ to avoid confusions when multiple measures are considered. If $E\subseteq \X$ is Borel with $\mm(E)<\infty$ and $u \in L^1_{loc}(\X)$, we denote 
\[
    (u)_E := \fint_E u\,\d\mm.
\]
We denote by \(\mathscr P(\X)\) the family of Borel probability measures equipped with the weak topology in duality with continuous and bounded functions \(C_b(\X)\). Given $\varphi \colon \X \to {\rm Y}$ where ${\rm Y}$ is a metric space, we define $\varphi_\sharp \mu (B) \coloneqq \mu(\varphi^{-1}(B))$, for all $B\subseteq {\rm Y}$ Borel and $\mu \in \PP(\X)$. We denote, respectively, the local Lipschitz constant and the asymptotic Lipschitz constant of $f\colon \X\to\R$ as
\[ 
\lip ( f)(x) \coloneqq \limsup_{y\to x}\frac{|f(y)-f(x)|}{\sfd(y,x)},\qquad \lip_a \, f(x) \coloneqq \limsup_{y,z\to x}\frac{|f(y)-f(z)|}{\sfd(y,z)},
\]
set to $0$ if $x$ is isolated. We shall also consider the ascending and descending slopes, defined respectively as follows:
$$
    |D^+ f|(x)= \limsup_{y\to x}\frac{(f(y)-f(x))^+}{\sfd(y,x)},\qquad|D^- f|(x)= \limsup_{y\to x}\frac{(f(y)-f(x))^-}{\sfd(y,x)}. 
$$
Given an open set $\Omega$, we shall consider the sets $\Lip(\Omega)$, $\Lip_{bs}(\Omega)$ of Lipschitz functions and Lipschitz functions with bounded support on $\Omega$, where we recall that a function $f$ is said to have bounded support in $\Omega$ provided $\supp(f)\Subset \Omega$.
\subsection{Test plans}
Denote by $C([0,1],\X)$ the set of continuous curves valued in $\X$ endowed with the sup norm, and denote by
\(AC^2([0,1],\X)\) the subset of curves
\(\gamma\in C([0,1],\X)\) so that there is \(g\in L^2(0,1)\) non-negative satisfying
\begin{equation}\label{eq:def_AC_curve}
\sfd(\gamma_t,\gamma_s)\leq\int_s^t g(r)\,\d r,
\qquad \forall s,t\in[0,1]\text{ with }s\leq t.
\end{equation}
We recall that, given \( \gamma\in AC^2([0,1],\X)\) there exists
\(|\dot\gamma_t|\coloneqq\lim_{h\to 0}\sfd(\gamma_{t+h},\gamma_t)/|h|\) for a.e.\ \(t\in (0,1)\) (see \cite[Theorem 1.1.2]{AmbrosioGigliSavare08}). We now turn our attention to the notion of test plan \cite{AmbrosioGigliSavare11,AmbrosioGigliSavare11-3}.
\begin{definition}\label{def:test_plan}
Let \((\X,\sfd,\mm)\) be a metric measure space. A probability measure \(\pi\in\mathscr P\big(C([0,1],\X)\big)\) is a test plan, provided it is concentrated on \(AC^2([0,1],\X)\) and
\begin{subequations}\begin{align*}
&\exists\,{ C}>0:\quad\forall t\in[0,1],\quad(\e_t)_\sharp\pi\leq{ C}\mm,\\
& \rmKe(\pi) \coloneqq\iint_0^1|\dot \gamma_t|^2\,\d t\d \pi <\infty.
\end{align*}\end{subequations}
The minimal constant $C\ge 0$ {is called the compression constant} of $\pi$ and is denoted by \({\rm Comp}(\pi)\).
\end{definition}
The notation $\rmKe(\pi)$ stands for the kinetic energy of a plan, and it is canonically set to $+\infty$ if $\pi$ is not concentrated on \(AC^2([0,1],\X)\).  Furthermore, we define $$\Lip(\pi):=\||\dot\gamma_t|\|_{L^\infty(\pi\otimes\mathcal{L}^1)}.$$

\subsection{Optimal transport on metric spaces}
We recall some useful facts on optimal transport on a metric space $(\X,\sfd)$ (we refer to \cite{Villani09,AmbrosioBrueSemola24_Book} for a detail treatment). An admissible plan between two probability measures \(\mu_0,\,\mu_1 \in \PP(\X)\) is a probability measure  $\alpha \in \PP(\X\times\X)$ satisfying 
\(P^1_\sharp\alpha=\mu_0\) and \(P^2_\sharp\alpha=\mu_1\), 
\(P^1,\,P^2\colon\X\times\X\to\X\) being the canonical projection maps. We denote by \({\rm Adm}(\mu_0,\mu_1)\) the family of admissible plans between $\mu_0$ and $\mu_1$. We denote by $\PP_2(\X) \subseteq\PP(\X)$ the collection of probability measures with finite second moment, and we define the Wasserstein distance by 
\begin{equation}\label{eq:def_W2}
W_2^2(\mu_0,\mu_1)\coloneqq\inf_{\alpha\in{\rm Adm}(\mu_0,\mu_1)}
\int\sfd^2(x,y)\,\d\alpha(x,y)
,\qquad \forall
\mu_0,\mu_1\in\mathscr P_2(\X).
\end{equation}
We denote by \({\rm Opt}(\mu_0,\mu_1)\) the set of all optimal
plans between \(\mu_0\) and \(\mu_1\), i.e.\ of all minimisers
of \eqref{eq:def_W2}. We next introduce the so-called optimal dynamical plans. A geodesic in \((\X,\sfd)\) is a curve \(\gamma\in C([0,1],\X)\) satisfying \(\sfd(\gamma_t,\gamma_s)=|t-s|\,\sfd(\gamma_0,\gamma_1)\) for every \(t,s\in[0,1]\). We denote by $\Geo(\X)$ the set of all
geodesics in \(C([0,1],\X)\). Given \(\mu_0,\mu_1\in\mathscr P_2(\X)\), we set
\[
\OptGeo(\mu_0,\mu_1)\coloneqq\Big\{
\pi\in\mathscr P\big({\rm Geo}(\X)\big)\,:\,
(\e_i)_\sharp\pi=\mu_i\;\forall i=0,1,\,
(\e_0,\e_1)_\sharp\pi\in{\rm Opt}(\mu_0,\mu_1)\Big\}.
\]
We next discuss the dual formulation of the optimal transport problem, and recall basic properties around the Kantorovich potentials for the cost $c(x,y)=\sfd^2(x,y)$. The $c$-transform of $\varphi\colon\X\to\R\cup\{-\infty\}$ is defined by
\begin{equation}\label{Eq:defconiugato} 
\varphi^c(y) \coloneqq \inf_{x \in \X} \left\{ \frac{\sfd^2(x,y)}{2} - \varphi(x) \right\}.
\end{equation}
A function $\varphi$ is said to be $c$-concave if $\varphi=\psi^c$ for some $\psi$, in this case, it holds $\varphi=(\varphi^c)^c$.
\begin{definition}
    Let $(\X,\sfd)$ be a Polish metric space, let $\mu,\nu\in\PP_2(\X) $, let $\varphi\in L^{1}(\X,\mu)$ be a $c$-concave function. We say that $\varphi$ is a Kantorovich potential for $(\mu,\nu)$ if $(\varphi,\varphi^c)$ is a maximizer of
    \[
        \sup \left\{ \int\phi \, \d\mu + \int \psi \, \d\nu \colon \phi\in\ L^{1}(\X,\mu),\,\psi\in L^{1}(\X,\nu)\,\,\textit{and}\,\, \phi(x) + \psi(y) \le \frac{1}{2} \sfd^2(x,y) \right\}.
    \]
\end{definition}
We recall that if $\varphi$ is a Kantorovich potential for $(\mu,\nu)$, then the Kantorovich duality holds: 
$$\frac{1}{2}W^2_2(\mu,\nu)=\int \varphi \, \d\mu + \int \varphi^c \, \d\nu.$$
In particular, for any $\alpha\in\Opt(\mu,\nu)$, it holds $\varphi(x)+\varphi^c(y)=\frac{\sfd^2(x,y)}{2}$ for $\alpha\textit{-a.e.\ }(x,y)$. We next recall some results on Kantorovich potentials.

\begin{proposition}\label{prop:goodKant}
Let $(\X,\sfd)$ be a Polish metric space and let $\mu,\,\nu\in\PP_2(\X)$. Then, fixed $x_0 \in \X$, there exists $\varphi=\psi^c$ Kantorovich potential relative to $(\mu,\nu)$ s.t. $\psi \le \frac 12 \sfd^2(\cdot, x_0)$ on $\X$ and $\psi \equiv -\infty$ on $\X \setminus \supp(\nu)$. In particular,
\begin{equation}\label{eq:choice_kant}
\varphi= \inf_{y \in \supp(\nu)} \left\{ \frac{\sfd^2(\cdot,y)}{2} - \psi(y) \right\}.
\end{equation}
Furthermore, if $(\X,\sfd)$ is a bounded metric space, then $\varphi\in \Lip(\X)$ with $ \displaystyle\Lip(\varphi) \le \sup_{(x,y) \in \X \times \supp(\nu)} \sfd(x,y)$.
\end{proposition}
\begin{proof}
The existence part of the result is well-known, we refer, for example, to  \cite[Corollary 3.18]{AmbrosioBrueSemola24_Book}. Up to respectively adding and subtracting the same constant to the two potentials in a maximizing pair $(\phi, \psi)$, we can assume that $\phi(x_0)=0$, hence $\psi \le \frac 12 \sfd^2(\cdot, x_0)$. Furthermore, we can modify $\psi$ so that $\psi \equiv -\infty$ on $\X \setminus \supp(\nu)$, and then take $\varphi=\psi^c$. In the bounded case, the lipschitzianity is implied by Equation \eqref{eq:choice_kant} observing that $\frac{1}{2}\sfd^2(\cdot,y)$ is $L$-Lipschitz with $L:=\sup_x \sfd(x,y)$. 
\end{proof}

\begin{remark}  \rm
Notice that every $\varphi$ of the form \eqref{eq:choice_kant} enjoys the following property: if $\supp(\nu) \subseteq C \subseteq \X$, then the restriction $\varphi \mid_C$ is still $c$-concave in $(C, \sfd)$. In particular, it is still a Kantorovich potential for $(\mu, \nu)$ in $C$ if also $\supp(\mu) \subseteq C$.
\end{remark}

\begin{proposition}\label{Prop:slopeKantoPotentials} Let $(\X,\sfd)$ be a Polish metric space, let $\mu,\nu\in\PP(\X)$ such that $W_2(\mu,\nu)<\infty $, and let $\alpha \in \Opt(\mu,\nu)$. Then, for all Kantorovich potentials $\varphi:\X\to\R\cup\{-\infty\}$ relative to $(\mu,\nu)$, it holds 
\begin{equation}\label{eq:itforza}
    |D^+ \varphi|(x)\leq \sfd(x,y)\qquad \text{$\alpha$-a.e.\ }(x,y).
\end{equation}
In particular, $ |D^+ \varphi|\in L^2 (\X,\mu)$.
\end{proposition}
We refer to \cite[Proposition 3.9]{AmbrosioGigliSavare11} for a proof of the above result. We shall also need the following result for the compactness of Kantorovich potentials (see \cite[Lemma 2.3]{AGMR15}).
\begin{lemma}\label{lemma:Kantorovichpotentialscompactness}
    Let $(\X,\sfd)$ be a Polish metric space, and let $\sigma$, $\eta = f \mm$, $\eta_n = f_n \mm \in \PP_2(\X)$ be probability measures satisfying
\begin{itemize}
    \item $\sigma$ has compact support;
    \item$f_n \to f$ $\mm$-a.e.\ and $\sup_n f_n(x)(1 + \sfd^2(x, x_0)) \in L^1(\X)$ for some $x_0 \in \X$.
\end{itemize}
Suppose that there exist $C > 0$ and Kantorovich potentials $\varphi_n = \psi_n^c$ relative to $(\eta_n, \sigma)$ satisfying
\begin{equation} \label{eq:2.15}
    |\varphi_n(x)| \leq C(1 + \sfd^2(x, x_0)) \quad \forall x \in \X 
\end{equation}
and
\begin{equation} \label{eq:2.16}
    \psi_n \equiv -\infty \quad \text{on } \X \setminus \supp(\sigma) \quad \text{and} \quad \psi_n(x) \leq C \quad \forall x \in \X. 
\end{equation}
Then there exist a subsequence $n(k)$ and a Kantorovich potential $\varphi = \psi^c$ for $(\eta, \sigma)$ such that $\varphi_{n(k)} \to \varphi$ pointwise. In addition $\varphi$ satisfies \eqref{eq:2.15} and $\psi$ satisfies \eqref{eq:2.16}.
\end{lemma}
\subsection{Cheeger energy, heat flow and Entropy}
We recall the definition of Sobolev space via relaxation, following \cite{AmbrosioGigliSavare11-3} (we also refer to \cite{Cheeger00,Shanmugalingam00} and to \cite{AmbrosioIkonenLucicPasqualetto24} for a complete discussion). For possibly non complete spaces, the Sobolev space theory was investigated in \cite{SylvesterPoggi24,CaputoCavallucci24}.
\begin{definition}[Sobolev space]\label{def:Chp}
Let $\Xdm$ be a metric measure space. We define the Cheeger energy as the functional ${\rm Ch} \colon L^2(\X)\to [0,\infty]$ defined by 
\[ {\rm Ch}(f) \coloneqq  \inf \Big\{ \liminf_{n\to\infty} \int \left(\lip  (f_n)\right)^2\, \d \mm \colon f_n\in L^2(\X)\cap \Lip(\X), f_n\to f \emph{ in } L^2(\X) \Big\}.\]
The Sobolev space is defined as $H^{1,2}(\X,\sfd,\mm)\coloneqq \{f \in L^2(\X) \colon {\rm Ch}(f) <\infty\}$. When no confusion will arise on the reference measure, we shall shortly write $H^{1,2}(\X)$.
\end{definition}

Sometimes, we shall consider the Sobolev space and the Cheeger energy on the metric measure space $(C,\sfd,\mm\res C)$ for some $C\subseteq \X$ closed. 
In this case, we shall write $\rmCh^C(f)$ for the Cheeger energy of a Sobolev function $f \in H^{1,2}(C,\sfd,\mm\res C)$. It follows that $H^{1,2}(\X)$ is a Banach space equipped with the norm $\|f\|_{H^{1,2}(\X)}^2 = \|f\|_{L^2(\X)}^2 + \rmCh(f)$. 
We next report a key result on the Cheeger energy showing that it always admits an integral representation with a weak gradient satisfying the usual calculus rules of Sobolev functions (we refer to \cite{AmbrosioIkonenLucicPasqualetto24} and references therein).
\begin{theorem}\label{Calculus-weakg}
Let $\Xdm$ be a complete metric measure space, and let $f\in H^{1,2}(\X)$. There exists $|\nabla f| \in~L^2(\X)$, called \emph{weak gradient}, such that
\[
    {\rm Ch}(f) = \int |\nabla f|^2\, \d \mm.
\]
Moreover, if $g \in H^{1,2}(\X)$, the weak gradient satisfies the following properties:
\begin{enumerate}[i)]
    \item{(Convexity)} for all $\alpha,\,\beta\in\R$
    \[
        |\nabla(\alpha f + \beta g)|\le |\alpha||\nabla f|+|\beta||\nabla g|,\qquad  \mm\text{-a.e.;}
    \] 
    \item\textit{(Locality)} for any $\mathcal{L}^1$-negligible Borel set $N \subset \mathbb{R}$ it holds
    $$|\nabla f|(x) = 0 \quad \mm\text{-a.e. } x\in f^{-1}(N),$$
    so that, in particular, it holds 
    $$|\nabla f| = |\nabla g| \qquad \mm\text{-a.e. on } \{f = g\}\,;$$
  
    \item (Chain rule) for any $\phi \in \Lip(\mathbb{R})$ with $\phi(0) = 0$, it holds that $\phi \circ f \in H^{1,2}(\X)$ and 
     $$|\nabla(\phi \circ f)| = (\lip(\phi) \circ f)|\nabla f| \qquad \mm\text{-a.e.}; $$
     
    \item (Leibniz rule) if $f, g \in H^{1,2}(\X) \cap L^\infty(\X)$, then  $fg \in H^{1,2}(\X)$ and 
   $$ |\nabla (fg)| \leq |f||\nabla g| + |g||\nabla f| \qquad \mm\text{-a.e..} $$
\end{enumerate}
\end{theorem}
In light of the equivalence results for Sobolev spaces \cite{AmbrosioGigliSavare11-3}, a Sobolev function $f \in H^{1,2}(\X)$ is so that, for all test plans $\pi$ and for $\pi$-a.e.\ $\gamma \in C([0,1],\X)$, it holds that
\begin{equation}
    f \circ \gamma \in W^{1,1}(0,1),\qquad (f\circ \gamma_t)'\le |\nabla f|(\gamma_t)|\dot \gamma_t|,\qquad \text{a.e.\ }t\in(0,1).
\label{eq:weak gradient}
\end{equation}
We refer to \cite{AmbrosioIkonenLucicPasqualetto24} for a detailed discussions on the equivalent approaches to the metric Sobolev space theory. We shall also need an ``integrated'' version of \eqref{eq:weak gradient}, that essentially was observed to yield and equivalent notion of metric BV space in \cite{AmbrosioDiMarino14,DiMarinoPhD} (see also \cite{NobiliPasqualettoSchultz22} for further equivalent formulations). The following result, tailored for $p=2$, is taken from \cite[Theorem 3.1]{NobiliRenziVitillaro25}.
\begin{theorem}\label{thm:Sobolevintegrated}
Let $\Xdm$ be a complete metric measure space and let $f \in L^2(\X)$. The following are equivalent:
\begin{enumerate}[(i)]
\item $f \in H^{1,2}(\X)$;
\item there exists $C\ge 0$ so that
\[ \int f(\gamma_1) - f(\gamma_0)\,\d \pi \le {\rm Comp}(\pi)^{1/2} \rmKe^{1/2}(\pi) C,\]
for all test plans $\pi$.
\end{enumerate}
Moreover, denoting $C_f$ the minimal constant $C\ge0$ for (ii) to hold, we have 
$C_f^2 = {\rm Ch}(f)$. 
\end{theorem}
We recall the definition of the Shannon entropy, that is well defined by \eqref{eq:growth_cond} and \cite[Eq. (2.5)]{AGMR15}.
\begin{definition}
Let $\Xdm$ be a metric measure space. We define the Shannon entropy functional ${\rm Ent}_\mm \colon \PP_2(\X) \to \R \cup\{+\infty\}$  by
\begin{equation} {\rm Ent}_\mm(\mu) := \int \rho\log (\rho) \,\d \mm, \qquad \text{if } \mu = \rho\mm, \quad +\infty\text{ otherwise}. \label{eq:Entm}
\end{equation}
\end{definition}
We shall keep the same notation for the natural extension of the entropy to any finite non-negative measure, defined by the same integral formula.

We shall consider gradient flows of convex and lower semicontinuous functionals on Hilbert spaces (see \cite{AmbrosioGigliSavare08} for a complete discussion). An example is the heat flow generated as the $L^2$-gradient flow of the Cheeger energy. Later on, in Section \ref{sec:heat} we will consider such flow as a semigroup generated by a Dirichlet form. We report from \cite{AmbrosioGigliSavare11} a result for comple metric measure spaces.
\begin{proposition}\label{prop:dis_flussi}
Let $\Xdm$ be a complete metric measure space satisfying \eqref{eq:growth_cond}, and let $f_0 \in L^2(\X)$ such that $f_0\mm\in\mathscr{P}_2(\X)$. Let $(f_t)$ be the $L^2$-gradient flow of the Cheeger energy starting from $f_0$, and denote $\mu_t:=f_t\mm$ for any $t\ge 0$. Assume also that $\mu_0 \in D(\Ent_\mm)$. Then, for all $T>0$ we have $(\mu_t) \in {AC}^2([0,T], \mathscr{P}_2(\X))$ and
\begin{equation}\label{eq_dis_flussi1}
\Ent_\mm(\mu_0) \geq \Ent_\mm(\mu_T)+2\int_0^T \rmCh(\sqrt{f_t}) \,\d t+\frac 12\int_0^T |\dot{\mu}_t|^2 \, \d t.
\end{equation}
\end{proposition}
\begin{proof}
The proof follows immediately by combining \cite[Proposition 4.22 and Lemma 6.1]{AmbrosioGigliSavare11}. 
\end{proof}
We next recall the notion of an infinitesimal Hilbertian space from \cite{Gigli12}.
\begin{definition}
    Let $\Xdm$ be a metric measure space, and let $\varnothing \neq \Omega\subseteq  \X$ be open. We say that $\Omega$ is \emph{infinitesimally Hilbertian} provided for every $f,g \in H^{1,2}(\X)$ with $\supp(f),\supp(g)\Subset \Omega$ it holds
    \begin{equation}\label{eq:parall_id}
        \rmCh(f+g)+\rmCh(f-g)=2\rmCh(f)+2\rmCh(g).
    \end{equation}
    If $\Omega=\X$, we say that $\Xdm$ is an infinitesimally Hilbertian metric measure space.
\end{definition}
\begin{remark}\label{rem:restr_infHilb} \rm
If $\Omega$ is infinitesimally Hilbertian, then, for any open set $\Omega'\subset\X$ with $\overline{\Omega'} \Subset \Omega$ and $\mm(\partial \Omega')=0$, it holds that $H^{1,2}(\overline{\Omega'})$ is a Hilbert space equipped with the norm $\| f \|_{H^{1,2}(\overline{\Omega'})}^2=\|f\|_{L^2(\overline{\Omega'})}^2 +\rmCh^{\overline{\Omega'}} (f)$. This follows from the same arguments used in \cite[Proposition 4.22]{Gigli12}.\fr
\end{remark}
We recall a characterization of infinitesimally Hilbertian spaces, referring to \cite[Theorem 4.3.3]{GigliPasqualetto20}.
\begin{theorem}\label{thm:equivalent hilbert}
    Let $\Xdm$ be a complete metric measure space, and let $\varnothing \neq \Omega\subseteq  \X$ be open. Then $\Omega$ is \emph{infinitesimally Hilbertian} if and only if for all $f,g \in H^{1,2}(\X)$ with $\supp(f),\supp(g)\Subset \Omega$ it holds
    \[ 
        |\nabla(f+g)|^2+|\nabla(f-g)|^2 = 2|\nabla f|^2+2|\nabla g|^2,\qquad\mm\text{-a.e.\ on }\Omega.
    \]
\end{theorem} 
On an infinitesimally Hilbertian space, the parallelogram identity guarantees that the following is well-defined and bilinear in its entries
\[
    \mathcal{E}_\mm(f,g) \coloneqq \lim_{\eps\to 0}\int \frac{|\nabla(f+\varepsilon g)|^2-|\nabla f|^2}{\varepsilon}\,\d\mm,\qquad\forall f,g \in H^{1,2}(\X),
\]
hence (see \cite{Gigli12,AGMR15}) it holds that $\mathcal E_\mm(f,f)$ is the canonical regular and Markovian Dirichlet form associated to the Cheeger energy, and we shall denote by $\mathbf{h}_t(f)$ the associated linear heat semigroup for $f \in L^2(\X)$. When no confusion will arise on the chosen reference measure, we shall shortly write $\mathcal E(f,f)$. 
Notice that such linear semigroup is precisely the $L^2$-gradient flow of the Cheeger energy, as the linearity is guaranteed by the fact that the Cheeger energy is assumed quadratic. Thanks to Proposition \ref{prop:dis_flussi}, we can show the following useful property.
\begin{proposition}\label{prop:der_Wass}
Let $\Xdm$ be an infinitesimally Hilbertian complete metric measure space 
satisfying \eqref{eq:growth_cond}. Let $f_0 \in L^\infty \cap L^2(\X)$ such that $f_0 \mm \in \PP_2(\X)$ and let $\mu_t:=f_t\mm$ with $f_t:=\mathbf{h}_tf_0$. Assume also that $\mu_0 \in D(\Ent_\mm)$. Let $\sigma \in \PP(\X)$. Then for a.e. $t>0$ the following property holds: for any Kantorovich potential $\varphi_t \in H^{1,2}(\X)$ relative to $(\mu_t, \sigma)$, one has
\begin{equation}
    \frac 12 \frac\d{\d t}W_2^2(\mu_t, \sigma)=-\mathcal{E}(f_t, \varphi_t).\label{eq:hilb dot product} 
\end{equation}
\end{proposition}
\begin{proof}
We follow the proof of \cite[Theorem 6.3]{AGMR15}. By Proposition \ref{prop:dis_flussi}, for a.e. $t>0$ the derivative of the left-hand side in \eqref{eq:hilb dot product}  exists. Also, for a.e. $t>0$ the derivative of $s \mapsto f_s$ at $s=t$ exists in $L^2(\X)$ and coincides with the infinitesimal generator of the semigroup, namely the distributional Laplacian $\Delta f_t$ (cf.\ \cite{Gigli12}). 
Then, for all $g \in H^{1,2}(\X)$ we have
\begin{equation}\label{eq:limiteinL2}
\lim_{h \to 0} \int g \frac{f_{t+h}-f_t}{h} \,\d\mm=\int g\Delta f_t \,\d\mm=-\mathcal{E}(f_t, g).\end{equation}
Now, since $\varphi_t$ is a Kantorovich potential for $(\mu_t, \sigma)$, we get
\begin{align*}
& \frac 12 W_2^2(\mu_t, \sigma)=\int \varphi_t \,\d\mu_t+\int \varphi_t^c \, \d\sigma, \\
&\frac 12 W_2^2(\mu_{t+h}, \sigma) \ge \int \varphi_t \,\d\mu_{t+h}+\int \varphi_t^c \,\d\sigma \qquad \forall h>-t.
\end{align*}
Taking the difference between the two lines and using \eqref{eq:limiteinL2} gives
\[\frac 12 W_2^2(\mu_{t+h}, \sigma)-\frac 12 W_2^2(\mu_t, \sigma) \ge -h\mathcal{E}(f_t, \varphi_t)+o(h).\]
As the derivative of $s \mapsto W_2^2(\mu_s, \sigma)$ exists at $s=t$, we conclude by sending $h \to 0$.
\end{proof}
We point out that, for our scope, the above results will be invoked in metric measure spaces where \eqref{eq:growth_cond} is trivially satisfied.

We also need to change the reference measure through certain weight functions. In \cite[Theorem 3.6]{AGMR15} it was shown that if $\nu=g\mm$ with $g\in L^\infty(\mm),g\ge 0$ and $\mathcal{E}(\sqrt g,\sqrt g)<\infty$, then every $f\in H^{1,2}(\X,\sfd,\mm)$ satisfying $|\nabla f|\in L^2(\nu)$ belongs to $H^{1,2}(\X,\sfd,\nu)$ and, denoting by $|Df|_g$ the weak gradient in the metric measure space $(\X,\sfd,\nu)$, it holds
\begin{equation}
    |\nabla f| = |\nabla f|_g,\qquad \nu\text{-a.e. in $\X$}. \label{eq:weak gradi pesati}
\end{equation} 
In addition, we have the following useful chain rule result from \cite[Theorem 3.9]{AGMR15}.
\begin{theorem}\label{Teo:Dirichletchain}
    Let $\Xdm$ be an infinitesimally Hilbertian complete metric measure space and let $\nu = g\mm \in \mathscr P_2(\X)$ with
$g \in L^\infty(\X)$ and $\mathcal{E}(\sqrt g,\sqrt{g})<\infty$.
Then also $(\X,\sfd,\nu)$  is infinitesimally Hilbertian and the associated Dirichlet form $\mathcal{E}_{\nu}$ satisfies
\[
\mathcal E_\nu(\log(g),\varphi)=\mathcal E_\mm(g,\varphi),\qquad \forall \varphi \in \Lip_{bs}(\X).
\]
\end{theorem}
\subsection{Doubling \& Poincar\'e properties}
We give next the definition of doubling and locally doubling measures on open sets. 
\begin{definition}\label{def:loc_doubling}
Let \((\X,\sfd,\mm)\) be a metric measure space and let $\varnothing \neq \Omega\subseteq \X$ be an open set. We say that $\mm$ is doubling on $\Omega$ if there exists a constant $C_D>0$ such that
\begin{equation}\label{eq:locally_doub}
\mm(B_{2r}(x_0)) \le C_D \mm(B_r(x_0)),
\end{equation}
for all $x_0 \in \Omega$ and $r>0$ such that $\overline{B}_{2r}(x_0) \subseteq \Omega$. Furthermore, we say that $\mm$ is \emph{locally doubling} on $\Omega$ if for all $x \in \Omega$ there exists a radius $r_D(x) \in (0,+\infty]$ such that $\mm$ is doubling on $B_{r_D(x)}(x)$, in which case we denote by $C_D(x)>0$ the constant appearing in \eqref{eq:locally_doub}.
\end{definition}
A well known consequence of the doubling assumption is the following result.
\begin{lemma}\label{lem:loc_compact}
Let \((\X,\sfd,\mm)\) be a metric measure space, and let $\varnothing \neq \Omega\subseteq  \X$ be an open set. Assume that $\mm$ is locally doubling on $\Omega$. Then for every $x \in \Omega  \cap  \supp(\mm)$ and $r <\frac{r_D(x)}{2}$ the ball $\overline{B}_r(x)$ is compact. In particular, if $\supp(\mm \res \Omega)=\overline{\Omega}$ then $\Omega$ is locally compact.
\end{lemma}
\begin{proof}
Fix $x \in \Omega$ and choose $c>0$ such that $r< \frac{r_D(x)}{2+c}$. Fix also $\eps=\frac{r(1+c)}{2^N}$ for some $N \in \N$. For any $x_1, \ldots, x_k \in \overline{B}_r(x)$ such that $\sfd(x_i, x_j) \ge \eps$ for all $i \neq j$, we have that the balls $B_{\eps/2}(x_i)$ are disjoint and contained in $\overline{B}_{r+\frac{\eps}{2}}(x)$. On the other hand, since $\overline{B}_{cr}(x) \subset \overline{B}_{r(1+c)}(x_i) \subset \overline{B}_{r(2+c)}(x) \subset B_{r_D(x)}(x)$, the locally doubling condition gives that
\[\mm(B_{\eps/2}(x_i)) \geq C_D(x)^{-(N+1)} \mm(B_{cr}(x))>0,\]
since $x \in \supp(\mm)$. Therefore
\[k \le C_D(x)^{N+1}\frac{\mm(\overline{B}_{r+\frac{\eps}{2}}(x))}{\mm(B_{cr}(x))}< \infty,\]
which implies that $\overline{B}_r(x)$ is totally bounded, as $\eps$ is arbitrarily small.
\end{proof}
We shall also need the concept of a Poincar\'e inequality on open sets that we recall next. We refer to \cite{Caputo26} and references therein for a general treatment. 
\begin{definition}
    Let $p,q \in [1,\infty)$, let $(\X, \sfd, \mm)$ be a metric measure space, and let $u \in L^q(\X)$ and $g \in L^p(\X)$ with $g$ non-negative. We say that the pair $(u,g)$ satisfies the weak $(q,p)$-Poincaré inequality in $\varnothing \neq \Omega\subseteq  \X$ open if there exists a finite constant $C_P>0$ such that
    \begin{equation}\label{eq:weak_P}
        \|u-(u)_{B_r(x_0)}\|_{L^q (B_r(x_0))} \leq C_P r\|g\|_{L^p (B_{2r}(x_0))},
    \end{equation}
    for any $x_0\in \Omega$ and $r>0$ such that $\overline{B}_{2r}(x_0)\subseteq \Omega$. Furthermore, we say that $\Omega$ satisfies the weak $(q,p)$-Poincaré inequality if \eqref{eq:weak_P} holds for all $u \in \Lip_{bs}(\X)$ with $g=\lip(u)$.
\end{definition}
Note that if $(\X,\sfd)$ is complete and $\Omega$ satisfies the weak $(q,2)$-Poincaré inequality, then \eqref{eq:weak_P} holds for all $u \in H^{1,2}(\overline{\Omega})$ with $g=|\nabla u|$. This can be shown by appealing the energy-density of Lipschitz functions \cite{AmbrosioGigliSavare11-3} and a truncation and monotone approximation argument (recall also \cite{Keith03}). 
\subsection{pmG-topology and convergence of functions}
We recall the definition of Gromov convergence, originally due to \cite{Gromov07}. We shall follow the extrinsic approach of \cite{GMS15} that, even if different from the original definition, will be better suited for our goals (we refer to \cite{GMS15} for the known equivalences results). To allow unbounded metric spaces, we shall consider pointed metric measure space, namely a quadruple $\Xdmx$, where $\Xdm$ is a metric measure space and $x  \in \X$.
\begin{definition}[pmG-convergence]\label{def:pmG}
Let $\Xdmxn$ be a sequence of pointed metric measure spaces. We say that that $\Xdmxn$ \emph{pointed measured Gromov-converges} (pmG-converges) to some limit pointed metric measure space $\Xdmxinf$, provided there exist a complete and separable metric space $(\Z,\sfd)$ and isometric embeddings
\[
\begin{split}
\iota_n &\colon (\X_n,\sfd_n) \to (\Z,\sfd), \\
\iota_\infty &\colon (\X_\infty,\sfd_\infty) \to (\Z,\sfd),
\end{split}
\]
such that $(\iota_n)(x_n) \to \iota_\infty(x_\infty)$ and $ (\iota_n)_\sharp \mm_n \to (\iota_\infty)_\sharp \mm_\infty$ in duality with $C_{bs}(\Z)$.
\end{definition}
In what follows, we shall shortly write $\X_n \overset{pmG}{\rightarrow} \X_\infty$ and identify the spaces $\X_n$ with their isomorphic images in $\Z$ following the extrinsic approach. We also recall the definition of convergence of functions along pmG-converging spaces (\cite{GMS15,AmbrosioHonda17}).
\begin{definition}\label{def:L2 convergence}
Let $\Xdmxn$ be a sequence of pointed metric measure spaces satisfying $\X_n \overset{pmG}{\rightarrow} \X_\infty$ for some pointed limit space, and let us fix a realization $(\Z,\sfd)$. We say that
\begin{enumerate}[(i)]
\item $f_n\in L^2(\mm_n)$ {converges $L^2$-weak} to $f_\infty\in L^2(\mm_\infty)$, if $\sup_{n}\|f_n\|_{L^2(\mm_n)}<\infty$ and $f_n\mm_n \weakto ~f_\infty\mm_\infty$ in duality with $C_{bs}(\Z)$,
\item $f_n\in L^2(\mm_n)$ {converges $L^2$-strong} to $f_\infty\in L^2(\mm_\infty)$, if it converges $L^2$-weak and \\$\limsup_n \|f_n\|_{L^2(\mm_n)} \leq~ \|f_\infty\|_{L^2(\mm_\infty)}$.
\end{enumerate}
\end{definition}
We recall, see e.g.\ \cite[Proposition 6.2]{NobiliViolo22}, that if $f_n$ converges $L^2$-strong to $f_\infty$ and $g_n$ converges $L^2$-weak to $g_\infty$, then 
\begin{equation}
\lim_{n\to\infty}\int f_n g_n\,\d\mm_n = \int f_\infty g_\infty\,\d\mm_\infty.\label{eq:coupling}
\end{equation}
\section{Local curvature dimension conditions: definitions, properties and stability}\label{sec:CD loc defs}
\subsection{Infinite dimensional setting}

We will now introduce a definition of local curvature dimension condition in an open set $\Omega\subseteq  \X$ based on the displacement convexity of the Shannon entropy, thus localizing the celebrated notions of \cite{Sturm06I,Sturm06II} and \cite{Lott-Villani09}. 

\begin{definition}\label{def:CDloc open set}
Let \((\X,\sfd,\mm)\) be a metric measure space, and let $\varnothing \neq \Omega\subseteq \X$ be an open set. We say that $\Omega$ satisfies the local infinite dimensional curvature dimension condition if for all $x\in \Omega$ there exist $K(x) \in \R$ and $r(x) \in (0,+\infty]$ with $r(x)\le r_C(x)$ such that, for all $\mu_0,\mu_1 \in \PP_2(\X)$  with $\supp(\mu_i)\Subset B_{\frac{r(x)}{2}} (x)$ for $i=0,1$,  there exists \(\pi\in{\rm OptGeo}(\mu_0,\mu_1)\) such that, denoting
\(\mu_t\coloneqq(\e_t)_\sharp\pi\), it holds
\begin{equation}\label{eq:convexity Ent}
    {\rm Ent}_\mm(\mu_t)\leq(1-t){\rm Ent}_\mm(\mu_0)+t{\rm Ent}_\mm(\mu_1)-\frac{K(x)}{2}t(1-t)W_2^2(\mu_0,\mu_1),
\end{equation}
for every $ t\in[0,1]$. In this case, we will say that $\Omega$ satisfies ${\sf CD}_{loc} (K(\cdot),\infty)$. Finally, we say that $\Omega$ satisfies strong \({\sf CD}_{loc}(K(\cdot),\infty)\) if (\ref{eq:convexity Ent}) holds for all $\pi\in~{\rm OptGeo}(\mu_0,\mu_1)$.
\end{definition}
The definition recovers the global ${\sf CD}(K,\infty)$ condition for complete metric measure spaces, if there is $x$ such that $r(x) = +\infty$ and $K(x) = K$.

We next show some useful results implied by the definition.
\begin{lemma}\label{lem:restr under strong convexity}
Let \((\X,\sfd,\mm)\) be a metric measure space, and let $\varnothing \neq \Omega\subseteq  \X$ be an open set. Suppose that $\Omega$ satisfies \({\sf CD}_{loc}(K(\cdot),\infty)\). Fix $x \in \Omega, r\le \frac{r(x)}{4}$ and let $\mu_0,\mu_1 \in \PP_2(\X)$ with $\supp(\mu_i) \Subset B_{r}(x)$ for $i=0,1$ and $\pi \in \OptGeo(\mu_0,\mu_1)$ be so that (\ref{eq:convexity Ent}) holds. Then, for every $t,s \in [0,1], s<t$, there exists $\eta \in \OptGeo((\e_s)_\sharp \pi,(\e_t)_\sharp \pi)$ so that   
\[
    {\rm Ent}_\mm((\e_{h})_\sharp \eta)\leq(1-h){\rm Ent}_\mm((\e_s)_\sharp \pi)+h{\rm Ent}_\mm((\e_t)_\sharp \pi)-\frac{K(x)}{2} h(1-h)W_2^2((\e_s)_\sharp \pi,(\e_t)_\sharp \pi),
\]
for all $h\in [0,1]$. Moreover, if $\Omega$ satisfies strong \({\sf CD}_{loc}(K(\cdot),\infty)\), then it is possible to take $\eta = \left({\sf rest}_s^t\right)_\sharp \pi$, where ${\sf rest}_s^t(\gamma)(h) \coloneqq \gamma_{(1-h)s + h t}$.
\end{lemma}
\begin{proof}
    Since  $\supp(\mu_i) \subset B_{r}(x)$ for $i=0,1$ and $r\le \frac{r(x)}{4}$ we have that $\supp\left( (\e_t)_\sharp\pi\right) \subset B_{2r} (x) $ for all $t\in[0,1]$. Thus, the existence of $\eta$ is granted again by Definition \ref{def:CDloc open set}. The last conclusion is also straightforward.
\end{proof}
\begin{definition}
   Let $(\X,\sfd)$ be a metric space, let $t\in(0,1)$ and let $\mu_0,\mu_1 \in \PP_2(\X)$. The class of $t$-midpoints between $\mu_0$ and $\mu_1$ is defined by
    $$ \mathcal{I}_t (\mu_0,\mu_1):=\bigl\{\mu\in\PP_2(\X)\text{  s.t. }  
    W_2(\mu_0 ,\mu)=tW_2(\mu_0 ,\mu_1) \text{ and }W_2(\mu ,\mu_1)=(1-t)W_2 (\mu_0 ,\mu_1)\,\bigr\}.$$
\end{definition}

\begin{lemma}\label{lem:supportintermmeasures}
    Let $(\X,\sfd)$ be a metric space, let $t\in(0,1)$ and let $\mu_0,\mu_1 \in \PP_2(\X)$. Consider $\mu_t\in~ \mathcal{I}_t (\mu_0,\mu_1)$, let $[0,1]\ni s\mapsto \tilde{\mu}^1_s,\tilde{\mu}^2_s$  be two $W_2$-geodesics respectively from $\mu_0$ to $\mu_t$ and from $\mu_t$ to $\mu_1$. Then, the curve
    $$
    s\mapsto \mu_s:=
    \begin{cases}
     \,\tilde{\mu}^1 _{\frac{s}{t}}          &\text{if $0\leq s\leq t$,}\\
     \,\tilde{\mu}^2 _{\frac{s-t}{1-t}}           &\text{if $t< s\leq 1$.}
    \end{cases}
    $$
    is a $W_2$-geodesic from $\mu_0$ to $\mu_1$. Furthermore, if $D:=\sup \{\Lip(\pi)\,\mid \pi\in\OptGeo(\mu_0,\mu_1)\}$ then, for any $0\leq s<r\leq 1$ and $\Tilde{\pi}\in\OptGeo (\mu_s ,\mu_r)$ it holds $\Lip(\Tilde{\pi})\leq (r-s)D$.
\end{lemma}
\begin{proof}
To show that $\mu_s$ is a geodesic, we prove that
$$W_2(\mu_s, \mu_r) \leq (r-s)W_2(\mu_0, \mu_1),$$
for any $0\leq s<r\leq 1$. By triangular inequality, it is enough to check the cases $0\leq s<r\leq t$ and $t\leq s<r\leq 1$. In the first case, since $s \mapsto \tilde{\mu}_s^1$ is a $W_2$-geodesic and $\mu_t\in \mathcal{I}_t (\mu_0,\mu_1)$, we get
$$W_2(\mu_s, \mu_r)=W_2(\tilde{\mu}_{s/t}^1,\tilde{\mu}_{r/t}^1)=\left( \frac rt- \frac st \right) W_2(\mu_0, \mu_t)=(r-s)W_2(\mu_0, \mu_1).$$
The case $t\leq s<r\leq 1$ is analogous. Finally, the last conclusion is straightforward.
\end{proof}
\begin{proposition}\label{Prop:existgeodesicsCDdoubling}
    Let $\Xdm$ be a complete metric measure space, and let $\varnothing \neq \Omega\subseteq  \X$ be an open set satisfying \({\sf CD}_{loc}(K(\cdot),\infty)\). Suppose that $\mm$ is locally doubling on $\Omega$. Then, for all $x\in\Omega$ and $z,y\in ~\overline{B}_{\min\left\{\frac{r(x)}{2},\frac{r_D (x)}{5}\right\}}(x)\cap~\supp(\mm)$ there exists $\gamma\in\Geo(\X)$ from $y$ to $z$.
\end{proposition}
\begin{proof}
    Let $R:= \max\{\sfd(x,y),\sfd(x,z)\}$. By Lemma \ref{lem:loc_compact} $\overline{B}_{2R}({x})$ is compact.
    By Definition \ref{def:CDloc open set}, for all $n\in~\N$ sufficiently large, there exist
    $\pi_n\in\OptGeo\left(\frac{\mm\res {B_{\frac{1}{n}}(y)}}{\mm(B_{\frac{1}{n}}(y))},\frac{\mm\res B_{\frac{1}{n}}(z)}{\mm(B_{\frac{1}{n}}(z))}\right)$. Let $\gamma_n \in\supp (\pi_n)$ and observe that $(\gamma_n )$ is a  $2R$-Lipschitz collection of geodesics supported in $\overline{B}_{2R}({x})$ and satisfying $\lim_{n\to\infty}\gamma_n(0)=y$ and $\lim_{n\to\infty}\gamma_n (1)=z$.
    By applying the Arzelà-Ascoli theorem, it is possible to extract a uniform-limit $\gamma \in C([0,1],\X)$. It is then straightforward to verify that $\gamma$ is a geodesic joining $y$ and $z$.
\end{proof}
\begin{remark}\rm
   Let us comment about the optimal choice for the function $K(\cdot)$ given by Definition \ref{def:CDloc open set}. Consider the function $\tilde{K}$ defined by
   \begin{equation}\label{Def:exactcurvature}
        \tilde{K}(x):=\sup\{K\in \R\,\mid \,\exists\,r_K(x)>0  \text{ s.t. }\eqref{eq:convexity Ent}\text{ holds  }\forall\,\mu_0,\mu_1 \in \PP(\X)\text{ supported in }B_{\frac{r_K (x)}2}(x)  \},
   \end{equation}
    for all $x \in \Omega$. Even in the smooth case the supremum will not be in general a maximum, hence $\Omega$ will not always satisfy ${\sf CD}_{loc} (\tilde{K}(\cdot),\infty)$, although it will satisfy ${\sf CD}_{loc} (K(\cdot),\infty)$ for any $K<\tilde{K}$. Furthermore, $\tilde{K}$ is admissible in the sense of \cite{Ketterer17}, namely it is lower semicontinuous and locally bounded below. Indeed, by \eqref{Def:exactcurvature} the function $\tilde{K}$ is locally bounded below . Furthermore, it is simple to infer its lower semicontinuity by noting that, for any $K< \tilde{K}(x) $ and any $y\in B_{\frac{r_K(x)}2}(x)$, it holds that $\tilde{K}(y)\geq K$, since \eqref{eq:convexity Ent} holds for all $\mu_0,\mu_1 \in \PP(\X)$ supported on $B_{{\frac 12r_K (x)}-\sfd(x,y)}(y)$.
    Finally, it is worth noting that $\tilde{r}_K(\cdot)$ can also be shown to be lower semicontinuous, where $\tilde{r}_K(x)$ denotes the supremum among all admissible values of $r_K(x)$ in \eqref{Def:exactcurvature}. \fr
\end{remark}
We end this part by showing that the local curvature dimension bounds are stable in the pmG-topology. We follow closely the proof of \cite[Theorem 4.9]{GMS15} and, to this aim, we first show a local Gamma-convergence result for the Entropy functionals in the spirit of \cite[Proposition 4.7]{GMS15}. 
\begin{proposition}\label{prop:Gamma entropy}
     Let $(\X_n, \sfd_n, \mm_n, x_n)$ be a sequence of pointed complete metric measure spaces that are pmG-converging to a pointed complete limit space $(\X_\infty, \sfd_\infty, \mm_\infty, x_\infty)$. Then, it holds
     \begin{enumerate}[(i)]
         \item for every $\mu_\infty \in \PP_2(\X_\infty)$ and $\mu_n \in \PP_2(\X_n)$ with $W_2(\mu_n,\mu_\infty)\to 0$, it holds
         \[
            {\rm Ent}_{\mm_\infty}(\mu_\infty) \le \liminf_{n\to\infty}{\rm Ent}_{\mm_n}(\mu_n);
         \]
         \item  for every $\eps>0$ and $\mu_\infty \in \PP(\X_\infty)$ absolutely continuous with $\supp(\mu_\infty)\subseteq B_r(x_\infty)$, $r>0$, there exists $\mu_n\in \PP(\X_n)$ with $\mu_n = \rho_n \mm_n,\rho_n \in L^\infty(\X_n),\supp(\rho_n)\subseteq B_{r +\eps}(x_n)$ such that it holds that $W_2(\mu_n,\mu_\infty)\to 0$ and
         \[
             \limsup_{n\to\infty}{\rm Ent}_{\mm_n}(\mu_n) \le {\rm Ent}_{\mm_\infty}(\mu_\infty).
         \]
     \end{enumerate}
\end{proposition}
\begin{proof}
    The first conclusion is directly implied by the first conclusion in \cite[Proposition 4.7]{GMS15}.

    We next show the second conclusion. Let $k \in \N$ and let us define the family of truncated probability measures $\mu_\infty^k \coloneqq \rho^k\mm_\infty$ where $\rho^k \coloneqq \min\{ \rho,k\} / \|\min\{ \rho,k\}\|_{L^1(\X_\infty)} $ where we write $\mu_\infty= \rho \mm_\infty$. Observe that $ {\rm Ent}_{\mm_\infty}(\mu^k_\infty) \to {\rm Ent}_{\mm_\infty}(\mu_\infty)$ as $k\to \infty$ and that  $\supp(\mu^k_\infty)\subset B_r(x_\infty)$ for each $k$. By \cite[Lemma 4.1]{NobiliRenziVitillaro25}, we infer that there exists a sequence $\mu_n^k \coloneqq\rho_n^k\mm_n$ of probability measures such that  $\rho_n^k \in L^\infty(\X_n)$, $\supp(\mu^k_n)\subset B_{r+\eps}(x_n)$, $\mu_n^k \to \mu_\infty^k$ weakly as $n \to \infty$ and 
    \[
        \int \varphi(\rho_n^k)\,\d \mm_n \to \int \varphi(\rho_\infty^k)\,\d\mm_\infty,\qquad \forall \varphi \in C(\R),\varphi(0)=0.    
    \]
    Actually, the latter is stated for $\varphi(t) = |t|^q$, but the proof is achieved for any such $\varphi$ exploiting the uniformly bounded supports. In particular, we can choose $\varphi(t) = t\log(t)$. Similarly, the weak convergence stated in \cite[Lemma 4.1]{NobiliRenziVitillaro25} turns into $W_2$-convergence for the same reason. The proof is concluded as, by a diagonalization argument, we can find a sequence  $k_n \to \infty$ such that $\mu_n \coloneqq \rho_n^{k_n}\mm_n$ satisfies all the listed conclusions.
\end{proof}
\begin{theorem}\label{thm:stability CDloc}
    Let $(\X_n, \sfd_n, \mm_n, x_n)$ be a sequence of pointed complete metric measure spaces that are pmG-converging to a pointed complete limit space $(\X_\infty, \sfd_\infty, \mm_\infty, x_\infty)$, and let $(\Z,\sfd)$ be a realization of the convergence. Let $\Omega_n\subseteq \X_n,\Omega_\infty \subseteq \X_\infty$ be non-empty open sets. Assume that $\Omega_n$ satisfies ${\sf CD}_{loc}(K_n(\cdot),\infty)$ for all $1 \leq n < \infty$. Furthermore, denoting by $r_n(\cdot)$ the local radius of $\Omega_n$, suppose the following: for all $x \in \Omega_\infty$ there are $y_n \in \Omega_n$ with $\sfd(y_n,x)\to 0$ satisfying $r_\infty(x)\coloneqq \liminf_n r_n(y_n) >0$ and $K_\infty(x)\coloneqq \liminf_n K_n(y_n) > -\infty$. Then $\Omega_\infty$ satisfies ${\sf CD}_{loc}(K_\infty(\cdot),\infty)$ with local radius $r_\infty(x)$. Finally, if $r_\infty(x)=+\infty$ for some $x\in\Omega_{\infty}$, then $\X_\infty$ satisfies ${\sf CD}(K,\infty)$ globally with $K=K_\infty(x)$.
\end{theorem}
\begin{proof}    
    Fix $x \in \Omega_\infty$ (without loss of generality, we also require $x \in \supp(\mm_\infty)$) and $\mu_0,\,\mu_1 \in \PP_2(\X)$ absolutely continuous such that $\supp(\mu_i) \Subset B_{\delta/2}(x)$ with $\delta > 0$ to be specified depending on $x$, for $i=0,1$. Consider $y_n \to x$ in the realization of the convergence satisfying the assumptions.

    Consider, by Proposition \ref{prop:Gamma entropy} and for all $\eps>0,i=0,1$, the probability measures $\mu_{i,n}\in \PP(\X_n)$ with $\mu_{i,n} = \rho_{i,n} \mm_n ,\supp(\rho_{i,n})\subseteq B_{\delta/2 +\eps}(x_n)$ satisfying the second conclusion. Let $\pi_n \in \OptGeo(\mu_{0,n},\mu_{1,n})$ be such that \eqref{eq:convexity Ent} holds. By arguing as in \cite[Theorem 4.9]{GMS15}, it is possible to show that there exists a $W_2$-geodesic $\mu_{t}$ such that for all $t\in [0,1]$
    \[
        W_2((\e_t)_\sharp \pi_n , \mu_{t})\to 0,\qquad \text{as }n\to \infty, 
    \]
    so that by the first conclusion in Proposition \ref{prop:Gamma entropy}, we deduce
    \[
        {\rm Ent}_{\mm_\infty}( \mu_{t}) \le \liminf_{n\to\infty}{\rm Ent}_{\mm_n}((\e_t)_\sharp \pi_n),\qquad\forall t \in [0,1].
    \]
    Choose $\delta <\liminf_{n\to \infty}r_n(y_n)$. By recalling the properties of $\mu_{i,n}$, we deduce (for $\eps$ small enough)
    \begin{align*}
          {\rm Ent}_{\mm_\infty}( \mu_{t})& \le \liminf_{n\to\infty}\left( (1-t) {\rm Ent}_{\mm_n}( \mu_{0,n}) + t{\rm Ent}_{\mm_n}( \mu_{1,n}) -\frac{K_n(y_n)}{2}t(1-t)W_2^2(\mu_{0,n},\mu_{1,n})  \right) \\
          &  \le (1-t) {\rm Ent}_{\mm_\infty}( \mu_{0}) + t{\rm Ent}_{\mm_\infty}( \mu_{1}) -\frac{K_\infty(x)}{2}t(1-t)W_2^2(\mu_{0},\mu_{1}),
    \end{align*}
    where $K:= \liminf_n K_n(y_n)$, and the desired conclusion is obtained by lifting the $W_2$-geodesic $(\mu_{t})$ with \cite{Lisini07} to obtain the desired plan $\pi \in \OptGeo(\mu_{0},\mu_{1})$.  The last conclusions also follows by noticing that, when $r_\infty(x)=+\infty$, then we can choose $\delta$ arbitrarily large.
\end{proof}

\subsection{Finite dimensional setting}
We next consider a local notion of dimensional curvature dimension condition, and we recall before some basic ingredients from \cite{Lott-Villani09,Sturm06I,Sturm06II}. The \(\sigma_{K,N}^{(t)}\) coefficient for \(K\in\R\),
\(N\ge 1\), \(t\in[0,1]\), and \(\theta\in[0,\infty)\), are defined as
\[
\sigma^{(t)}_{K,N}(\theta) := \begin{cases} 
\ \infty, &\text{if } K\theta^2 \ge N\pi^2,\\
\ \frac{\sin(t\theta\sqrt{K/N})}{\sin(\theta\sqrt{K/N)}}, &\text{if } 0<K\theta^2<N\pi^2,\\
\ t, &\text{if } K\theta^2<0\text{ and } N=1 \text{ or if } K\theta^2=0, \\
\ \frac{\sinh(t\theta\sqrt{-K/N})}{\sinh(\theta\sqrt{-K/N)}}, &\text{if } K\theta^2< 0 \text{ and }N>1.\\
\end{cases}
\]
We then set $\tau^{(t)}_{K,N}(\theta) := t^{\frac{1}{N}}\sigma^{(t)}_{K,N-1}(\theta)^{1-\frac{1}{N}}$ while $\tau^{(t)}_{K,1}(\theta) =t$ if $K\le 0$ and $\tau^{(t)}_{K,1}(\theta)=\infty$ if $K>0$.

Given a metric measure space \((\X,\sfd,\mm)\) and
\(N> 1\), we define the \(N\)-R\'{e}nyi relative entropy functional
\(\mathcal U_N\colon\mathscr P(\X)\to[0,\infty]\) as
\[
\mathcal U_N(\mu)\coloneqq\int\rho^{1-\frac{1}{N}}\,\d\mm,
\qquad \forall \mu\in\mathscr P(\X),\;\mu=\rho\mm+\mu^s
\text{ with }\mu^s\perp\mm.
\]
In the borderline case $N=1$,  this reduces by convention to $\mathcal U_1(\mu) =~\mm(\{\rho>~0\})$.
\begin{definition}\label{def:CDloc dimensional}
Let \((\X,\sfd,\mm)\) be a metric measure space, and let $\varnothing \neq \Omega\subseteq  \X$ be an open set. We say that $\Omega$ satisfies the local finite dimensional curvature dimension condition, if for all $x\in \Omega$ there exist $K(x) \in \R,N(x)\ge 1$ and  $r(x)\in (0,+\infty]$ with $r(x)\leq r_C (x)$ such that, for all $\mu_0,\mu_1 \in \PP_2(\X)$  with $\supp(\mu_i)\Subset B_{\frac{r(x)}{2}} (x)$ for $i=0,1$,  there exists \(\pi\in{\rm OptGeo}(\mu_0,\mu_1)\) such that, denoting
\(\mu_t\coloneqq(\e_t)_\sharp\pi\), it holds
\begin{equation}
\mathcal U_{N'}(\mu_t)\geq\int\rho_0(\gamma_0)^{-\frac{1}{N'}}
\tau_{K(x),N'}^{(1-t)}\big(\sfd(\gamma_0,\gamma_1)\big)+
\rho_1(\gamma_1)^{-\frac{1}{N'}}
\tau_{K(x),N'}^{(t)}\big(\sfd(\gamma_0,\gamma_1)\big)\,\d\pi(\gamma),
\label{eq:convexity Reny}
\end{equation}
for every \(N'\geq N(x)\) and \(t\in[0,1]\). In this case, we say that $\Omega$ satisfies ${\sf CD}_{loc}(K(\cdot),N(\cdot))$. Finally, we say that $\Omega$ satisfies strong ${\sf CD}_{loc}(K(\cdot),N(\cdot))$ if (\ref{eq:convexity Reny}) holds for all $\pi\in{\rm OptGeo}(\mu_0,\mu_1)$.
\end{definition}
As previously discussed, the above definition recovers the original global notion of ${\sf CD}(K,N)$ for complete metric measure spaces. Furthermore, if $\Omega$ satisfies the above, then it also satisfies Definition \ref{def:CDloc open set} by sending $N' \uparrow \infty$ by similar arguments as in \cite{Lott-Villani09,Sturm06I,Sturm06II}. 
\begin{remark}\label{rmk:kett implies us}
    \rm
    We note that the notion of curvature dimension condition with variable Ricci curvature lower bound formulated by \cite{Ketterer17} implies Definition \ref{def:CDloc dimensional} in the case $\Omega = \X$. Indeed, in \cite{Ketterer17} the variable lower bound on the Ricci curvature is a lower semicontinuous and \emph{locally bounded below} function $k\colon \X\to \R$.  This guarantees that for every $x$, there is $K(x)\in \R$ such that $k(x)\ge K(x)$ on $B_1(x)$. Therefore, exploiting the monotonicity of the $\tau$-coefficient (cf.\ \cite[Theorem 1.1]{Ketterer17}), it is straightforward to see that Ketterer's notion implies ${\sf CD}_{loc}(K(\cdot),N)$ in Definition \ref{def:CDloc dimensional} with the choice of local radius $r(x)=1$. 
\end{remark}

We now discuss some basic properties implied by the above definition. First, we investigate the validity of a local Bishop-Gromov monotonicity property. To this aim, we introduce the following notation
\[
    \left[0, \pi \sqrt{\frac{N-1}{K^+}}\right] \ni t\mapsto  \sin_{K/(N-1)}(t) = \begin{cases}
    \, \sin\left(t\sqrt{\frac{K}{N-1}}\right) &\text{if }K>0 ,\\
    \, t &\text{if } K=0,\\
    \,  \sinh\left(t\sqrt{\frac{-K}{N-1}}\right)&\text{if } K<0,
    \end{cases}
\]
so that the volume of a ball of radius $r \in \left[0, \pi \sqrt{\frac{N-1}{K^+}} \, \right]$ in the model space of dimension $N>1$ and constant curvature $K\in\R$ is given by $v_{K,N}(r) = \int_0^r \sin_{K/(N-1)}^{N-1}(t)\,\d t.$
\begin{lemma}   
    Let \((\X,\sfd,\mm)\) be a metric measure space, and let $\varnothing \neq \Omega\subseteq \X$ be an open set satisfying ${\sf CD}_{loc}(K(\cdot),N(\cdot))$. Then, for every $x \in \Omega, x_0\in B_{r(x)/2}(x) \,
     \cap \,\supp(\mm)$ and $0<r\le R< \sfd\left(x_0,\left( B_{r(x)/2}(x)\right)^c\right)$, it holds that $\mm(\partial B_r(x_0))=0$. Furthermore, for all
    $0<r\le R< \min\left\{  \sfd\left(x_0,\left(B_{r(x)/2}(x)\right)^c\right) , \pi \sqrt{\frac{N(x)-1}{K^+(x)}} \right\}$, if $N(x)>1$ we have
    \begin{equation}
        \frac{\mm(B_r(x_0))}{\mm(B_R(x_0))} \ge \frac{v_{K(x),N(x)}(r)}{v_{K(x),N(x)}(R)},
    \label{eq:BG loc}
    \end{equation}
    while if $N(x)=1,K(x)\le 0$ we have
    \[
         \frac{\mm(B_r(x_0))}{\mm(B_R(x_0))} \ge \frac{r}{R}.
    \]
\end{lemma}
\begin{proof}
   Let $x$, $x_0$ as in the statement and $0<r\le R< \sfd(x_0,B_{r(x)/2}(x)^c)  =: ~L$. Define $t=r/R$ and let $\eps,\delta>0$ small enough so that $\max\{\eps, R+\delta R, r+\delta r+\eps t \}<  \sfd(x_0,B_{r(x)/2}(x)^c) $. Consider the sets
    \[
        A_0\coloneqq B_\eps(x_0),\qquad A_1 = \overline{B}_{R+\delta R}(x_0)\setminus B_{R}(x_0),\qquad A_t \coloneqq \{\gamma_t \colon \gamma \in {\rm Geo}(\X), \gamma_0 \in A_0,\gamma_1 \in A_1\},
    \] 
    and define 
    \[
        \Theta \coloneqq \begin{cases}
        \, \inf_{y \in A_0,z \in A_1}\sfd(y,z),&\text{if }K(x)\ge 0,\\
        \, \sup_{y \in A_0,z \in A_1}\sfd(y,z),&\text{if }K(x)<0.
    \end{cases}
    \]
   Assume for the moment that $\mm(A_1)>0$ (note that $\mm(A_0)>0$ is certainly true as $x_0 \in \supp(\mm)$).
    We point out that the generalized Brunn-Minkowski inequality proved in \cite[Proposition 2.1]{Sturm06I} holds true in the current setting, with our choice of sets $A_0,A_1$ and parameters $t,\Theta,K(x),N(x)$, since its proof only relies on the existence of a plan $\pi \in \OptGeo( \mm\res{A_0}/\mm(A_0), \mm\res{A_1}/\mm(A_1))$ satisfying the convexity inequality \eqref{eq:convexity Reny} and that, by construction, $A_0, A_1 \subset B_{\frac{r(x)}2}(x)$). In particular, noticing that $A_t \subset \overline{B}_{r+\delta r+\eps t}(x_0)\setminus B_{r-\eps r}(x_0)$ and that $R-\eps \le \Theta \le R+\delta R+\eps$, we deduce after manipulations and monotonicity of the $\tau$-coefficients that
    \[
        \mm( \overline{B}_{r+\delta r+\eps t}(x_0)\setminus B_{r-\eps r}(x_0) )^{\frac 1N}\ge  \tau_{K(x),N(x)}^{(1-t)}(R\mp \delta R \mp \eps)\mm(A_0)^\frac{1}{N} + \tau_{K(x),N(x)}^{(t)}(R\mp \delta R \mp \eps)\mm(A_1)^\frac{1}{N},
    \]
    where the choice $\mp$ coincides with the sign of $K(x)$. Taking $\eps\downarrow 0$, we get after manipulations
    \begin{equation}\label{eq:BGincrementalinequality}
          \mm (\overline{B}_{(1+\delta)r}(x_0)) - \mm(B_r(x_0)) \ge \left( \tau_{K(x),N(x)}^{(t)}(R\mp \delta R)\right)^N \left( \mm(\overline{B}_{(1+\delta)R}(x_0)) - \mm(B_R(x_0))\right).
    \end{equation}
    Note that \eqref{eq:BGincrementalinequality} is trivially true even in the case $\mm(A_1)=0$.
    Consider the function $s\mapsto v(s)\coloneqq \mm(B_s(x_0))$ that, by definition, is left-continuous and non-decreasing. Since it is possible to choose arbitrarily small radii $r>0$ such that $\mm(\partial B_{r}(x_0))=0$, the estimate \eqref{eq:BGincrementalinequality} implies that $v$ is continuous at each point $s \in (0,L) $. In particular, we get that  $\mm(\partial B_s(x_0))=0$ for all $s \in (0,L)$. 
    From here, the Bishop-Gromov monotonicity \eqref{eq:BG loc} follows by the same arguments as in \cite[Theorem 2.3]
    {Sturm06II}. 
\end{proof}

The next corollary is about a local doubling property of the reference measure and it immediately follows from the previous lemma (cf.\ \cite[Corollary 2.4]{Sturm06II}).
\begin{corollary}\label{cor:dim_CDloc}
     Let \((\X,\sfd,\mm)\) be a metric measure space, and let $\varnothing \neq \Omega\subseteq \X$ be an open set satisfying ${\sf CD}_{loc}(K(\cdot),N(\cdot))$. Then, for every $x \in \Omega$ there exists $C(x)>0$ so that for all $x_0\in B_{r(x)/2}(x) \,\cap~\,\supp(\mm)$ we have
     \[
        \mm(B_{2r}(x_0))\le C(x)\mm(B_r(x_0)),\qquad\forall  r>0\text{ s.t. } 2r<\sfd(x_0,B_{r(x)/2}(x)^c).
     \]
     Moreover, we can estimate
     \[
        C(x)\le \begin{cases}\, 2^{N(x)},&\text{if }K(x)\ge 0\text{ or }N(x)=1, \\
        2^{N(x)} \cosh\left( \frac{r(x)}{2}\sqrt{\frac{-K(x)}{N(x)-1}}\right)^{N(x)-1},&\text{otherwise}.
        \end{cases}
     \]
     Furthermore, if $\supp(\mm \res \Omega)=\overline{\Omega}$, then $\mm$ is locally doubling on $\Omega$ (in the sense of Definition \ref{def:loc_doubling}) with the choice $r_D(x)= r(x)/2$.
\end{corollary}
\begin{remark}\label{rmk:doubling small ball} \rm 
We point out that if $x,r(x),C(x)$ are as in the above corollary, there is $\overline{C}=\overline{C}(x)>0$ (depending  only on $C(x)$ above) such that for all $r\le r(x)/10$, it holds 
\[
    \mm\res{\overline{B}_r(x)} (B_{2s}(y))\le \overline{C}\cdot \mm\res{\overline{B}_r(x)}(B_s(y)),\qquad\forall s >0, \forall y \in \overline{B}_r(x).
\]
In particular, we deduce that
\[
    \left(\overline{B}_r(x),\sfd,\mm\res{\overline{B}_r(x)}\right),\qquad \text{is \emph{uniformly} doubling},
\]
meaning that the doubling condition holds at every point and with every radius with uniform positive constant.
This follows by standard arguments, using that $B_{r(x)/10}(x)$ is geodesic (cf.\ Lemma \ref{Prop:existgeodesicsCDdoubling}), which implies the so-called \emph{corkscrew condition} (see, e.g., \cite[Definition 2.4 and Lemma 2.5]{BjornSh07}).
\fr 
\end{remark}
We next show the stability of local curvature dimension bounds in the pmG-topology in the finite dimensional setting. The proof is similar in spirit to that already carried in Theorem \ref{thm:stability CDloc}. We shall refer to well established arguments contained in the monograph \cite{Villani09}, limiting ourselves to note the main differences in the current local setting. 
\begin{theorem}\label{thm:stability dimensional CDloc}
     Let $(\X_n, \sfd_n, \mm_n, x_n)$ be a sequence of pointed complete metric measure spaces that are pmG-converging to a pointed complete limit space $(\X_\infty, \sfd_\infty, \mm_\infty, x_\infty)$, and let $(\Z,\sfd)$ be a realization of the convergence. Let $\Omega_n\subseteq \X_n,\Omega_\infty \subseteq \X_\infty$ be non-empty open sets. Assume that $\Omega_n$ satisfies ${\sf CD}_{loc}(K_n(\cdot),N_n(\cdot))$ for all $1 \leq n < \infty$. Furthermore, denoting by $r_n(\cdot)$ the local radius of $\Omega_n$, suppose the following: for all $x \in \Omega_\infty$ there are $y_n \in \Omega_n$ with $\sfd(y_n,x)\to 0$ satisfying 
     $$
     \liminf_{n \to \infty} r_n(y_n) >0, \qquad K_\infty(x)\coloneqq\liminf_{n \to \infty} K_n(y_n) > -\infty, \qquad N_\infty (x)\coloneqq \limsup_{n\to\infty}N_n(y_n)< +\infty.
     $$ 
     Then $\Omega_\infty$ satisfies ${\sf CD}_{loc}(K_\infty(\cdot),N_\infty(\cdot))$ with local radius $r_\infty(x)=\frac{1}{10}\cdot \liminf_n r_n(y_n)$. Finally, if $r_\infty(x)=+\infty$ for some $x\in\Omega_{\infty}$, then $\X_\infty$ satisfies ${\sf CD}(K,N)$ globally with $K=K_\infty(x),N=N_\infty(x)$.
\end{theorem}
\begin{proof}
We fix $x \in \Omega_\infty$ as in the assumption (without loss of generality, we also require $x \in \supp(\mm_\infty)$). \\
        The classical argument from \cite[Chapter 29]{Villani09} that shows the stability of the global $\CD(K,N)$ condition  still works in our context (notice that \cite[Theorem 29.25]{Villani09} is stated for fixed $K, N$ along the sequence, but the same result indeed holds for variable sequences $K_n$, $N_n$ for instance taking into account \cite[Theorem 29.20]{Villani09}). Indeed, we observe that for a.e.\ $r < \frac{1}{10}\cdot \liminf_n {r_n(y_n)}$ it holds
        \begin{equation}\label{Eq:pmGconvergencerestrictions}
            (\overline{B}_r(y_n)\cap \supp(\mm_n),\mm_n\res{\overline{B}_r(y_n)},y_n)\xrightarrow{pmG}(\overline{B}_r(x)\cap \supp (\mm_\infty),\mm_\infty\res{\overline{B}_r(x)},x).
        \end{equation}
       For any such $r$, we know that the sequence of spaces is uniformly doubling. Indeed, Proposition \ref{Prop:existgeodesicsCDdoubling} guarantees the existence of geodesics from $y_n$ to any $z \in \overline{B}_r (y_n)$. By Remark \ref{rmk:doubling small ball}, this allows us to localize the doubling property given by Corollary \ref{cor:dim_CDloc} to these spaces, at the cost of a fixed increase in the doubling constant. In light of the latter observation, \cite[Proposition 3.33]{GMS15} and \cite[Remark 3.29]{GMS15} guarantee that the convergence \eqref{Eq:pmGconvergencerestrictions} in the sense of Definition \ref{def:pmG} is equivalent to the one employed in \cite{Villani09}, that is, the so-called \emph{pointed measured Gromov Hausdorff (pmGH) convergence}. The setting is now identical to the one of \cite[Theorem 29.25]{Villani09}: since we performed an analogous localization, the original proof carries over, yielding that \eqref{eq:convexity Reny} holds for all probabilities $\rho_0 ,\rho_1$ supported in $B_{\frac{r}{2}}(x)$. Indeed, the convexity inequality \eqref{eq:convexity Reny} is only required on the balls $B_{\frac r2}(y_n)$ (see also the proof of \cite[Theorem 29.24]{Villani09}), a fact that is guaranteed in our local scenario.
\end{proof}
\subsection{The local Riemannian curvature dimension condition}
We recall that coupling the infinitesimal Hilbertianity with the curvature dimension condition leads to a subclass of metric measure spaces that is tailored to study limits of Riemannian structures. In the following definitions, we adopt the same combination for the local finite or infinite dimensional setting.
\begin{definition}\label{def:RCDloc open set}
    Let \((\X,\sfd,\mm)\) be a metric measure space, and let $\varnothing \neq \Omega\subseteq  \X$ be an open set. We say that $\Omega$ satisfies \({\sf RCD}_{loc}(K(\cdot),\infty)\), if it is infinitesimally Hilbertian and satisfies ${\sf CD}_{loc}(K(\cdot),\infty)$.
\end{definition}
\begin{definition}\label{def:dimensional RCDloc open set}
    Let \((\X,\sfd,\mm)\) be a metric measure space, and let $\varnothing \neq \Omega\subseteq \X$ be an open set. We say that $\Omega$ satisfies \({\sf RCD}_{loc}(K(\cdot),N(\cdot))\), if it is infinitesimally Hilbertian and satisfies ${\sf CD}_{loc}(K(\cdot),N(\cdot))$.
\end{definition}
\section{Local density estimates along transport geodesics}
In the following theorem, we revisit the main result in \cite{Rajala12-2} on open sets satisfying the ${\sf CD}_{loc}$ condition. We show that there is an abundance of test plans concentrated on geodesics between bounded marginals whose supports are sufficiently small. 
\begin{theorem}\label{thm:rajalaCDqKN}
Let \((\X,\sfd,\mm)\) be a complete metric measure space, and let $\varnothing \neq \Omega\subseteq  \X$ be an open set satisfying \({\sf CD}_{loc}(K(\cdot),\infty)\). Let $x \in \Omega$, and let $\rho_0 ,\rho_1 \in L^{\infty}(\mm)$ be boundedly supported probability densities satisfying
\[
 \supp(\rho_0)\cup \supp(\rho_1) \subseteq \overline{B}_{r(x)/6}(x).
\]
Then, there exists $ \pi \in \OptGeo(\rho_0\mm,\rho_1\mm)$ with $(\e_t)_\sharp \pi \ll \mm$ and, writing $(\e_t)_\sharp \pi\coloneqq  \rho_t\mm$, we have
\begin{equation}\label{Eq:rajalacompressionestimate}
    \| \rho_t\|_{L^\infty(\X)} \le e^{\frac{K^-(x)}6D^2}\max\left\{ \|\rho_0\|_{L^\infty(\X)} , \|\rho_1\|_{L^\infty(\X)}\right\},\qquad \forall t \in [0,1],
 \end{equation}
 where $D:=\sup \{\Lip(\pi)\,\mid \pi\in\OptGeo(\rho_0\mm,\rho_1\mm)\}\leq \frac{r(x)}3$.
\end{theorem}
\begin{proof}
    We closely follow the same arguments in \cite{Rajala12-2}. We only highlight the main difference arising from the local framework while reporting the main steps for the sake of clarity. 
    
    Without loss of generality, and upon replacing $K(x)$ with $K^-(x)$, we may consider the case $K(x)\leq 0$ only. Since $\Omega$ satisfies \({\sf CD}_{loc}(K(\cdot),\infty)\), we can find $ \pi \in \OptGeo(\rho_0\mm,\rho_1\mm)$ supported on curves of length at most $D$ satisfying (\ref{eq:convexity Ent}).
    By Jensen's inequality and after manipulations in \ref{eq:convexity Ent} (see \cite{Rajala12-2} for details)  following spreading of mass under the curvature dimension condition holds
    \begin{equation}\label{eq:spreadingmassRajalalowerbound}
        \mm(\{\rho_{\frac{1}{2}}>0\})\geq\left(e^{\frac{K^-(x) D^2}{8}}\max\left\{ \|\rho_0\|_{L^\infty(\X)} , \|\rho_1\|_{L^\infty(\X)}\right\}\right)^{-1},
    \end{equation}
    where $(\e_{\frac{1}{2}})_\sharp\pi=\rho_{\frac{1}{2}}\mm$.
    Hereafter we will set $M:=e^{\frac{K^-(x) D^2}{8}}\max\left\{ \|\rho_0\|_{L^\infty(\X)} , \|\rho_1\|_{L^\infty(\X)}\right\}$ for brevity.
    
    Let us consider the excess functional $\mathcal{F}_C : \PP_2 (\X)\to [0,\infty]$ defined for $C\geq0$ as follows
    \begin{equation}
        \mathcal{F}_C (\mu)=\|(\rho-C)^+ \|_{L^\infty(\X)}+\mu^s (\X),\qquad \text{if }\mu = \rho\mm + \mu^s
    \end{equation}
    Note that $\mathcal{F}_C (\mu)=0$ encodes two pieces of information on $\mu$, namely the fact that it is absolutely continuous, and that its density is bounded from above by $C$.

    Next, we recall from \cite{Rajala12-2} that, given $\mu_0$ and $\mu_1$ with bounded supports, then 
    $\mathcal{I}_t (\mu_0,\mu_1)$ is closed in $(\PP_2(\X), W_2)$, and a minimizer of $\mathcal F_C(\cdot)$ exists in $\mathcal{I}_t (\mu_0,\mu_1)$ by standard lower semicontinuity property of the excess functional. A key point is then to show that for $C>M$ it holds
    \begin{equation}\label{min=0}
        \min_{\nu\in\mathcal{I}_{\frac{1}{2}} (\mu_0,\mu_1)} \mathcal{F}_C (\nu)=0.
    \end{equation}
    The proof of (\ref{min=0}) consists of two parts. The first part aims at  showing that if a minimizer $\nu=\rho_\nu \mm+\nu^s$ satisfies $\rho_\nu \le C$, the second one that $\nu^s=0$. Let us see the proof of the first part. What we do is assuming by contradiction that $\mm(\{\rho_\nu> C\})>0$ and using this information to create a new competitor $\Tilde{\nu} \in \mathcal{I}_{\frac{1}{2}} (\mu_0,\mu_1)$ with $\mathcal{F}_C(\Tilde{\nu}) < \mathcal{F}_C(\nu)$, which cannot be. To construct $\Tilde{\nu}$, one roughly redistributes some of the mass of $\nu$ which is over the threshold $C$, and so is penalized by the functional. To be precise, we choose the minimizer $\nu$ in such a way that
    \begin{equation} \mm(\{\rho_\nu>C\}) \ge \Big(\frac MC\Big)^{\frac14} \sup_{\eta \in \cI_{\min}} \mm(\{\rho_\eta>C\}), \label{eq:Binterim}
     \end{equation}
     where $\cI_{\min}$ is the (non-empty) set of minimizers for $\mathcal{F}_C$ in $\cI_\frac12(\mu_0,\mu_1)$. Subsequently, denoted $A:=\{\rho_\nu>C \}$, we choose a set $A':=\{\rho_{{\nu}}>C+\delta\}$ s.t.
    $$\mm(A') \ge \Big(\frac{M}{C}\Big)^{\frac12}\mm(A).$$
    Let $\alpha \in \Opt(\mu_0,\nu), \beta \in \Opt(\nu,\mu_1)$, and consider 
    \[ \tilde{\pi} \in \OptGeo\Big( (P^0)_\sharp\frac{\alpha\restr{\X\times A'}}{\nu(A')}, (P^1)_\sharp\frac{\beta\restr{A'\times \X}}{\nu(A')}  \Big),\]
    chosen in a way that (\ref{eq:spreadingmassRajalalowerbound}) holds. Denote $\Gamma_t := (\e_t)_\sharp\tilde{\pi}$ the corresponding $W_2$-geodesic and consider its decomposition $\Gamma_\frac12 = \rho_\Gamma\mm +\Gamma^s$. Then it follows  from \eqref{eq:spreadingmassRajalalowerbound} that
    \begin{equation}
     \mm(\{\rho_\Gamma >0\}) \geq \frac{\nu(A')}{M} \ge \frac{C}{M}\mm(A')\ge\Big(\frac{C}{M}\Big)^\frac12\mm(A). \label{eq:Binterim2}
     \end{equation}
    Now we redistribute the mass of the measure $\nu$ via
    \[\Tilde{\nu}:= \nu\restr{X\setminus A'} + \frac{C}{C+\delta}\nu\restr{A'}+\frac\delta{C+\delta}\nu(A')\Gamma_\frac12. \]
    Applying \cite[Lemma 3.5]{Rajala12-2}, one can check that $\Tilde{\nu} \in \cI_\frac12(\mu_0,\mu_1)$ (i.e., $\Tilde{\nu}$ is an admissible candidate). Also, a computation gives that
    \[ \cF_C(\nu)-\cF_C(\Tilde{\nu}) = \int_{\{\rho_\nu <C\}} \min\left\{ C-\rho_\nu, \frac{\delta}{C+\delta}\nu(A')\rho_\Gamma\right\}\, \d \mm \ge 0. \]
     Hence, to avoid a contradiction, it must be $\cF_C(\nu)=\cF_C(\Tilde{\nu})$ (i.e. $\Tilde{\nu}$ is also a minimizer). But then the integrand must vanish $\mm$-a.e., implying that $\mm(\{\rho_\nu<C\}\cap \{\rho_\Gamma>0\})=0$. Moreover, by construction $\rho_{\Tilde{\nu}}>C$ $\mm$-a.e. on the set $\{\rho_\nu \ge C\}\cap \{\rho_\Gamma>0\}$, hence 
    \[\mm (\{ \rho_{\Tilde{\nu}}>C\} ) \ge \mm(\{\rho_\Gamma>0\}) \overset{\eqref{eq:Binterim2}}{\ge} \Big(\frac{C}{M}\Big)^\frac12\mm(\{ \rho_\nu >C\} ) \overset{\eqref{eq:Binterim}}{\ge}  \Big(\frac CM\Big)^{\frac14} \sup_{\eta \in \cI_{\min}} \mm(\{\rho_\eta>C\}), \]
    which is still a contradiction (since $C>M$), concluding the proof that any minimizer of $\cF_C(\cdot)$ has density bounded from above by $C$. The proof that $\nu^s=0$ follows a similar argument, showing finally (\eqref{min=0}). We refer again to \cite{Rajala12-2} for details on the aforementioned computations.
    
    Furthermore, by arbitrariness of $C>M$ and by the definition of the functional, it holds
      $$\min_{\nu\in\mathcal{I}_{\frac{1}{2}} (\mu_0,\mu_1)} \mathcal{F}_M (\nu)=0.$$
      In particular, we can find a midpoint $\mu_{\frac{1}{2}}\in \mathcal{I}_{\frac{1}{2}} (\mu_0,\mu_1)$ such that $\mathcal{F}_M (\mu_{\frac{1}{2}})=0$, which means that $\mu_{\frac{1}{2}}=\rho_{\frac{1}{2}}\mm$ for some $\rho_{\frac{1}{2}}\in L^\infty (\X),\,\rho_\frac{1}{2} \leq M$ at $\mm$-a.e.\ point.
      The proof of the theorem then follows by iterating the construction above on dyadic time scales, obtaining the desired geodesic $s\mapsto \Gamma_s$ by completion.
     More precisely, we set $\Gamma_0 =\mu_0$, $\Gamma_1 =\mu_1$, and at the $n$-th step we produce $\Gamma_{\frac{k}{2^n}}=\rho_{\frac{k}{2^n}}\mm$, $k=1,3,\ldots,2^n -1$ as described, starting (inductively) from $\Gamma_{\frac{k-1}{2^n}}$ and $\Gamma_{\frac{k+1}{2^n}}$.
     Note that by Lemma \ref{lem:supportintermmeasures} all the measures involved in the construction are supported in $\overline{B}_{2r/3}(x)$, which means that every step employing the curvature dimension condition can be carried out in our local setting.
     Furthermore, using again Lemma \ref{lem:supportintermmeasures} to estimate $\Lip (\Tilde{\pi})$ uniformly for all $\Tilde{\pi}\in\OptGeo (\Gamma_{\frac{k-1}{2^n}},\Gamma_{\frac{k+1}{2^n}})$ one can find by induction the following estimate
     for all $n\in\N$, $k=1,\ldots, 2^n$
     \begin{equation}\label{Eq:stimacompressionepuntidiadicirajala}
         \|\rho_\frac{k}{2^n}\|_{L^\infty(\X)} \leq \prod_{i=1}^n e^{\frac{K^-(x) D^2}{2^{2i+1}}} \max\left\{ \|\rho_0\|_{L^\infty(\X)} , \|\rho_1\|_{L^\infty(\X)}\right\}\leq e^{\frac{K^-(x) D^2}{6}} \max\left\{ \|\rho_0\|_{L^\infty(\X)} , \|\rho_1\|_{L^\infty(\X)}\right\}.
     \end{equation}
     Having defined $\Gamma$ on a dense set in $[0,1]$, it can be extended by continuity as a $W_2$-geodesic $[0,1]\ni~ s\mapsto~ \Gamma_s$ from $\mu_0$ to $\mu_1$, also taking into account Lemma \ref{lem:supportintermmeasures}. Finally, the sought geodesic plan $\pi\in\OptGeo (\mu_0,\mu_1)$ is given by the optimal lifting (cf.\ \cite{Lisini07}) and, thanks to the semicontinuity of $\mathcal{F}_C$ and the estimate (\ref{Eq:stimacompressionepuntidiadicirajala}), the desired compression bounds \eqref{Eq:rajalacompressionestimate} hold.
\end{proof}
It is well known that the above compression estimates imply the validity of Poincar\'e inequalities (cf.\ \cite{Rajala12-2}). We next include a short proof of this fact, as it is highly relevant for our goals.
\begin{proposition}\label{prop:poincare}
Let \((\X,\sfd,\mm)\) be a complete metric measure space, and let $\varnothing \neq \Omega\subseteq  \X$ be an open set. Suppose that $\Omega$ satisfies \({\sf CD}_{loc}(K(\cdot),\infty)\). Fix $x_0 \in \Omega, r>0$ and a ball $\overline{B}_{2r}(x)\subset B_0\coloneqq B_{r_0(x)/6}(x_0)$. If $u \in \Lip(B_0)$, then 
for all $p \ge 1$ it holds
\[
  \int_{B_r(x)} |u-(u)_{B_r(x)}|^p \d\mm \leq 2^{2p+1}e^{\frac{K^-(x) 4r^2}{6}}r^p\int_{B_{2r}(x)} \lip(u)^p \,\d\mm.
\]
\end{proposition}
\begin{proof}

We follow the proof of \cite[Theorem 4.1]{Rajala12-2}.
Abbreviate $B = B_r(x)$ and define $M$ to be the median of $u$ in the ball $B$, i.e.
 \[
   M = \min\left\{a \in \R : \mm(\{u > a\}) \le \frac{\mm(B)}2\right\}.
 \]

 Using the median $M$ we cover the ball $B$ with two Borel sets
 \[
  B^+ = \{x \in B ~:~ u(x) \ge M\}\qquad \text{and}\qquad B^- = \{x \in B ~:~ u(x) \le M\}.
 \]
 Notice that $\mm(B^+), \mm(B^-) \ge \mm(B)/2$. Let
 \[
  \pi \in \OptGeo\left(\frac1{\mm(B^+)}\mm|_{B^+},\frac1{\mm(B^-)}\mm|_{B^-}\right)
 \]
 be the geodesic given by Theorem \ref{thm:rajalaCDqKN} and let $\rho_t$ be the density of $(\e_t)_\sharp\pi$ with respect to $\mm$. 
 By \eqref{Eq:rajalacompressionestimate} we have 
 \begin{equation}\label{eq:densitybound_luigi}
  \rho_t(y) \le C(r)\frac{2}{\mm(B)}\qquad\text{$\mm$-a.e.\ $y$}
 \end{equation}
 for all $t\in [0,1]$, with $C(r)=e^{\frac{K^-(x) 4r^2}6}$. Now observe that we have an equality
 \[
   |u(\gamma_0) - u(\gamma_1)| = |u(\gamma_0) - M| + |M - u(\gamma_1)|
 \]
 for $\pi$-almost every $\gamma \in \Geo(\X)$. Therefore
 \begin{align*}
  \int_{\Geo(\X)} & |u(\gamma_0) -  u(\gamma_1)|^p \d\pi(\gamma)\\
   & \ge\int_{\Geo(\X)} |u(\gamma_0) - M|^p\d\pi(\gamma) + \int_{\Geo(\X)}|M - u(\gamma_1)|^p\d\pi(\gamma)\\
   & = \frac{1}{\mm(B^+)}\int_{B^+}|u(y) - M|^p\d\mm(y) + \frac{1}{\mm(B^-)}\int_{B^-}|M-u(y)|^p\d\mm(y)\\
   & \ge \frac{1}{\mm(B)}\int_{B}|u(y) - M|^p\d\mm(y).
 \end{align*}
 Since $\pi$-almost every $\gamma \in \Geo(\X)$ is contained in the ball $B_{2r}(x)$ we have
{\allowdisplaybreaks 
 \begin{align*}
  \int_{B_r(x)}&|u - (u)_{B_r(x)}|^p\d\mm  \le \frac{1}{\mm(B)}\iint_{B \times B} |u(y) - u(z)|^p\d\mm(y)\d\mm(z) \\
   & \le \frac{2^{p-1}}{\mm(B)}\iint_{B \times B} (|u(y) - M|^p + |M - u(z)|^p)\d\mm(y)\d\mm(z) \\
   & = 2^p\int_{B} |u(y) - M|^p\d\mm(y) 
    \le 2^p\mm(B)\int |u(\gamma_0) - u(\gamma_1)|^p \d\pi(\gamma)\\
   & \le 2^p(2r)^p\mm(B) \int \int_0^1 \lip(u)^p(\gamma_t) \d t \d\pi(\gamma)
    = 2^{2p} r^p\mm(B) \int_0^1 \int_X \lip(u)^p(y)\rho_t(y)\d\mm(y)\d t\\
   & \le 2^{2p+1}r^p C(r) \int_0^1 \int_{B_{2r}(x)}\lip(u)^p(y)\d\mm(y)\d t 
    = 2^{2p+1} r^p C(r) \int_{B_{2r}(x)}\lip(u)^p\d\mm,
 \end{align*}}
 having used \eqref{eq:densitybound_luigi} in the last inequality.\qedhere
\end{proof}
In the following part, we are going to see that assuming strong ${\sf CD}_{loc}(K(\cdot),\infty)$ implies a sharper $L^\infty$ density estimate. The reason being that optimal geodesics are unique in this case. The following analysis essentially follows from the one carried out in \cite{RajalaSturm12}, and we include a quick proof to reduce to their assumptions.
\begin{theorem}\label{Prop_NB}
Let \((\X,\sfd,\mm)\) be a complete metric measure space, and let $\varnothing \neq \Omega\subseteq  \X$ be an open set. Suppose that $\Omega$ satisfies strong \({\sf CD}_{loc}(K(\cdot),\infty)\). Fix $x \in \Omega$, and let $B:=B_{r(x)/4}(x)$. Then, $B$ is an essentially non-branching  locus, i.e.\ for every $\mu_0$, $\mu_1 \in \PP_2(B)$ which are absolutely continuous with respect to $\mm$, then any $\pi \in \OptGeo(\mu_0, \mu_1)$ is concentrated on a set of non-branching geodesics. In particular, there is a unique $\pi \in \OptGeo(\mu_0, \mu_1)$ and it is induced by a map.
\end{theorem}
\begin{proof}
The thesis follows directly from \cite[Remark 3.2]{RajalaSturm12} (which contains the most general form of the main result of the paper, where the global $\CD$ condition is dropped) and Corollary 1.4 therein (for the last conclusion). Indeed, we only have to check that $\Ent_\mm$ is $K(x)$-convex along $({\rm rest}_s^t)_\sharp(f\pi)$ ($x$ being the centre of the ball $B$), for every $\pi \in \OptGeo(\mu_0, \mu_1)$, $0 \le s < t \le 1$ and $f: {\rm Geo}(\X) \to [0, \infty)$ with $\int_{{\rm Geo}(\X)} f \,\d\pi=1$. This simply follows from the fact that $\supp((\e_t)_\sharp(f\pi)) \subset 2B$ for all $t \in [0,1]$ and by taking into account Lemma \ref{lem:restr under strong convexity}.
\end{proof}
Thanks to the above uniqueness result of optimal geodesic plans, the next corollary directly follows from Theorem \ref{thm:rajalaCDqKN}.
\begin{corollary}\label{cor:improved_rajala}
    Let \((\X,\sfd,\mm)\) be a complete metric measure space, and let $\varnothing \neq \Omega\subseteq \X$ be an open set satisfying strong \({\sf CD}_{loc}(K(\cdot),\infty)\). Let $x \in \Omega$, and let $\rho_0 ,\rho_1 \in L^{\infty}(\X)$ be boundedly supported probability densities satisfying
    \[
        \supp(\rho_0)\cup \supp(\rho_1) \subseteq \overline{B}_{r(x)/6}(x).
    \]
    Then, the unique $ \pi \in \OptGeo(\rho_0\mm,\rho_1\mm)$ with $(\e_t)_\sharp \pi \ll \mm$ satisfies
    \begin{equation}\label{Eq:rajalacompressionestimate stronc}
    \| \rho_t\|_{L^\infty} \le e^{\frac{K^-(x)}6\Lip(\pi)^2} \max\left\{ \|\rho_0\|_{L^\infty(\X)} , \|\rho_1\|_{L^\infty(\X)}\right\},\qquad \forall t \in [0,1],
    \end{equation}
    having written $(\e_t)_\sharp \pi\coloneqq  \rho_t\mm$.
\end{corollary}
We will also need the following variant of Theorem \ref{thm:rajalaCDqKN}, { where the uniform $L^\infty$ bound is only near $t=0$,
but the gain is the convexity inequality \eqref{eq:convexity Ent}.}
\begin{proposition}\label{prop:goodgeodesics}
Let \((\X,\sfd,\mm)\) be a complete metric measure space, and let $\varnothing \neq \Omega\subseteq  \X$ be an open set satisfying \({\sf CD}_{loc}(K(\cdot),\infty)\). Let $x \in \Omega$, and let $\rho_0 ,\rho_1 \in L^{\infty}(\X)$ be boundedly supported probability densities satisfying 
\[
   \supp(\rho_0)\cup \supp(\rho_1) \subseteq \overline{B}_{r(x)/6}(x).
\]
Then, there exists $ \pi \in \OptGeo(\rho_0\mm,\rho_1\mm)$ s.t. $\mu_t:=(\e_t)_\sharp \pi \ll \mm$ and, writing $\mu_t \coloneqq  \rho_t\mm$, we have that $\mu_t$ satisfies the convexity inequality \eqref{eq:convexity Ent} for all $t \in [0,1]$ and there exists $t_0>0$ such that
 \begin{equation}\label{eq:uniformbound}
  \sup_{t \in [0,t_0]}\|\rho_t\|_{L^\infty} < \infty.
 \end{equation}
\end{proposition}
\begin{proof}
The proof follows by the same lines of that in \cite[Theorem 4.2]{AGMR15} which is in turn similar to that of Theorem \ref{thm:rajalaCDqKN}, and it exploits again the fact that $\supp(\rho_0)\cup \supp(\rho_1) \subset \overline{B}_{r(x)/6}(x)$ and the very definition of ${\sf CD}_{loc}(K(\cdot),\infty)$. We omit the details.
\end{proof}
\section{Stability of infinitesimal Hilbertianity: non-negative curvature}\label{sec:nonnegative}
In this section we show that the local $L^\infty$-estimates along transport geodesic obtained in Theorem \ref{thm:rajalaCDqKN}, coupled with the assumption of local non-negative curvature, are already effective to deduce related stability results as already noted in \cite{NobiliRenziVitillaro25}. The proof in this simplified setting highlights the main steps of the strategy that will be carried out to show the main results.
\begin{proposition}\label{MainT K>=0}
Let $(\X_n, \sfd_n, \mm_n, x_n)$ be a sequence of pointed complete metric measure spaces that are pmG-converging to a pointed complete limit space $(\X_\infty, \sfd_\infty, \mm_\infty, x_\infty)$, and let $(\Z,\sfd)$ be a realization of the convergence. Let $\Omega_n\subseteq \X_n,\Omega_\infty \subseteq \X_\infty$ be non-empty open sets. Assume that $\Omega_n$ satisfies $\rcd_{loc}(K_n(\cdot), \infty)$ for all $1 \leq n < \infty$. Furthermore, denoting by $r_n(\cdot)$ the local radius of $\Omega_n$, suppose the following: for $x \in \Omega_\infty$ there are $y_n \in \Omega_n$ with $\sfd(y_n,x)\to 0$ satisfying $\liminf_n r_n(y_n) >0$  and $\liminf_n K_n(y_n) \geq 0$. Then $\Omega_\infty$ is infinitesimally Hilbertian.
\end{proposition}
\begin{proof}
    We divide the proof into two steps. First we show a lower semicontinuity result for the Cheeger energy, then we exploit the Proposition \ref{Gammalimsup} below to conclude the proof.
    
    \noindent{\sc Step 1}. We claim that, for any $f_\infty \in L^2(\X_\infty)$ supported in a ball $B_r(x)$ small enough, and for any $f_n \to f_\infty$ in $L^2$-weak, it holds
    \[
        {\rm Ch}(f_\infty) \leq \liminf_{n\uparrow\infty} {\rm Ch}(f_n).
    \]
    First, fixed $\eps>0$, notice that ${\rm Ch}(f_\infty)$ is the same as ${\rm Ch}^H(f_\infty)$ with $H\coloneqq \overline{B}_{r+\eps}(x)$, being $f_\infty$ supported in $B_r(x)$. Recall also by Theorem \ref{thm:Sobolevintegrated} that we can characterize ${\rm Ch^{1/2}}(f_\infty)$ as the minimal constant $C\ge 0$ such that
    \begin{equation}\label{eq:lagr_Sob K>=0}
    \int f_\infty(\gamma_1) - f_\infty(\gamma_0)\,\d \pi \le {\rm Comp}(\pi)^{1/2} \rmKe^{1/2}(\pi) C,
    \end{equation}
    holds for all test plans $\pi$. Again, it is enough to check the above for all test plans $\pi$ with $\gamma_t \in \overline{B}_{r+\eps}(x)$ for all $t\in[0,1]$ and $\pi$-a.e.\ $\gamma$, being $f_\infty$ supported in $B_r(x)$. We thus consider fixed any such test plan $\pi$ and, in order to conclude, it suffices to show that (\ref{eq:lagr_Sob K>=0}) holds for the choice $C=\liminf_n {\rm Ch^{1/2}}(f_n)$. Without loss of generality, we shall also suppose that $C=\lim_n {\rm Ch^{1/2}}(f_n)$ along a suitable subsequence. \\
    Let now $\rho_0 \mm_\infty:=(\e_0)_\sharp\pi$, $\rho_1 \mm_\infty:=(\e_1)_\sharp\pi$ and construct for $i=0,1$ sequences of densities $\rho_{i,n}$ s.t. $\rho_{i,n} \mm_n \in \PP(\X_n)$ and
    \[\rho_{i,n} \to \rho_i \mbox{ in } L^2\mbox{-strong,} \qquad\supp(\rho_{i,n}) \subseteq \overline{B}_{r+2\eps}(x) \cap \X_n, \qquad \limsup_n \|\rho_{i,n}\|_{L^\infty(\X_n)} \leq \|\rho_i\|_{L^\infty(\X_\infty)} \]
    (see \cite[Lemma 4.1]{NobiliRenziVitillaro25}). We may assume with no loss of generality that $x \in \supp(\mm_\infty)$ (otherwise the statement is trivial). Consider, by assumptions, a sequence $y_n \in \Omega_n$ with the properties that $y_n \to x$ (in the realization) and that  $\liminf_n r_n(y_n) >0,\liminf_n K_n(y_n) \geq 0$. Hence, for $n$ large enough we have
    \[\supp(\rho_{i,n}) \subseteq \overline{B}_{r+3\eps}(y_n) \cap \X_n.\] 
    Now, if $r< r_0(x)=\frac 1{6} \liminf_n r_n(y_n)$ then (choosing $\eps>0$ small enough) the supports of the measures $\rho_{i,n}$ satisfy the hypothesis in Theorem \ref{thm:rajalaCDqKN} for all $n$. Let then $\pi_n \in \OptGeo(\rho_{0,n}, \rho_{1,n})$ the plan given by Theorem \ref{thm:rajalaCDqKN} and observe that $\liminf_n K_n(y_n) \geq 0$ in \eqref{Eq:rajalacompressionestimate} yields $\limsup_n \Comp(\pi_n) \leq \Comp(\pi)$. Also, $\lim_n \rmKe(\pi_n)=\lim_n W_2^2(\rho_{0,n}, \rho_{1,n})=W_2^2(\rho_0, \rho_1) \leq \rmKe(\pi)$. Hence, using \eqref{eq:coupling} at first, we get
    \[
    \begin{aligned}
    \int f_\infty(\gamma_1)-f_\infty(\gamma_0)\,\d\pi  &=\lim_{n\uparrow\infty} \int f_{n}(\gamma_1)-f_{n}(\gamma_0)\,\d\pi_{n} \\
    &\le \limsup_{n\uparrow\infty}
    \Comp\left(\pi_{n} \right)^{1/2}\rmKe^{1/2}\left(\pi_{n}\right)\rmCh^{1/2}(f_{n})\\
    &\le \Comp(\pi)^{1/2}\rmKe^{1/2}(\pi) \limsup_{n\uparrow\infty} \rmCh^{1/2}(f_{n}),
    \end{aligned}
    \]
    thus showing the claim.

    \noindent{\sc Step 2}. Fix $f$, $g \in H^{1,2}(\X)$ with supports well-contained in $\Omega_\infty$. We claim that
    \begin{equation}\label{a.e.Par Kge0}
    2|\nabla f|^2+2|\nabla g|^2=|\nabla(f+g)|^2+|\nabla(f-g)|^2 \qquad \mm\mbox{-a.e. in }\Omega_\infty.
    \end{equation}
    Notice that, by separability and locality of weak gradients, the claim would follow if we can show that for every $x \in \Omega_\infty$ there is a ball $B\subset \Omega_\infty$ centred in $x$ where (\ref{a.e.Par Kge0}) holds $\mm$-a.e.\ on $B$. Therefore, thanks to Theorem \ref{thm:equivalent hilbert}, it is enough to prove that
    $$2{\rm Ch}(f)+2{\rm Ch}(g)={\rm Ch}(f+g)+{\rm Ch}(f-g)$$
    holds for each $f,g$ supported on $B=B_r(x)$, for some suitably small $r$. Choose $r=r(x)$ and construct recovery sequences $f_n,g_n$ converging in $L^2$-strong to $f,g$ as in Proposition \ref{Gammalimsup} below. For $n$ large enough, it holds that $f_n,g_n$ are supported in $\Omega_n$. Thus, we deduce
    \begin{equation*}
    \begin{split}
     2{\rm Ch}(f)+2{\rm Ch}(g)&=\lim_{n\to\infty} (2{\rm Ch}(f_n)+2{\rm Ch}(g_n))=\lim_{n\to\infty}({\rm Ch}(f_n+g_n)+{\rm Ch}(f_n-g_n)) \\
     & \geq {\rm Ch}(f+g)+{\rm Ch}(f-g).
    \end{split}
    \end{equation*}
    The other inequality simply follows by replacing $f,g$ with $(f+g)/2$, $(f-g)/2$. By arbitrariness of $B$, this shows \eqref{a.e.Par Kge0} and thus concludes the proof.
\end{proof}
\begin{proposition}\label{Gammalimsup}
 Let $(\X_n, \sfd_n, \mm_n, x_n)$ be a sequence of pointed complete metric measure spaces that are pmG-converging to a pointed complete limit space $(\X_\infty, \sfd_\infty, \mm_\infty, x_\infty)$. Let us fix $(\Z,\sfd)$ a realization of the convergence, and let $B \subseteq \Z$ be a ball. For all $f_\infty \in L^2(\X_\infty)$ with $\supp(f_\infty) \Subset B$, there exists a sequence $f_n \in L^2(\X_n) \cap \Lip(\X_n)$ so that $f_n$ converges in $L^2$-strong to $f_\infty$, with the property that $\supp(f_n) \subseteq B \cap \X_n$ and
$${\rm Ch}(f_\infty) \geq \limsup_{n\to \infty} {\rm Ch}(f_n).$$
Finally, the same conclusion holds for all $f_\infty \in L^2(\X_\infty)$ with no restriction on the support.
\end{proposition}
\begin{proof}
The last conclusion is known, and shown for instance in \cite[Theorem 4.4]{Gigli23_working} (see also the arguments in \cite{AmbrosioHonda17,GMS15}). Here, we handle the case $B\subset \Z$, being careful about the supports. Fix $\eps>0$ s.t. $\supp(f_\infty) \subseteq (1-\eps)B$.
First, by definition of the Cheeger energy, we can find a sequence $g^k \in L^2(\X_\infty) \cap \Lip(\X_\infty)$ s.t. $g^k \to f_\infty$ in $L^2(\X_\infty)$ and $\int (\lip_a g^k)^2 \d\mm_\infty \to {\rm Ch}(f_\infty)$. We can also assume that $\supp(g^k) \subset (1-\frac \eps 2) B$, since one can easily check that the sequence $\eta g^k$, where $\eta \in \Lip(\X_\infty),0\le\eta\le 1$ is a suitable cut-off function, provides a better competitor. Now, exploiting the main result in \cite{DiMarinoGigliPratelli} we extend $g^k$ to a function $\tilde{g}^k \in \Lip(\Z)$ such that
\begin{equation}\label{ExtLip}
 \lip_a(g^k)=\lip_a(\tilde{g}^k) \qquad \mbox{on }\X_\infty. 
\end{equation}
Notice that, if we take $\chi \in \Lip(\Z; [0,1])$ s.t.\ $\chi \equiv 1$ in $(1-\frac \eps 2)B$, $\supp(\chi) \subset B$, then obviously $\chi \tilde{g}^k=\tilde{g}^k=g^k$ on $\X_\infty$, and furthermore, thanks to the Leibniz rule $\lip_a(\chi\tilde{g}^k) \leq \| \chi\|_{L^\infty}\lip_a(\tilde{g}^k)+\Lip(\chi)|\tilde{g}^k|$, we deduce
$$\lip_a(\chi \tilde{g}^k)=\begin{cases}
\lip_a(\tilde{g}^k) & \mbox{on }  \X_\infty \cap (1-\frac \eps 2) B , \\
0 & \mbox{on } \X_\infty \cap \left( (1-\frac \eps 2) B\right)^c.
\end{cases}$$
Hence, (\ref{ExtLip}) continues to hold for $\chi \tilde{g}^k$. That is, we can assume that $\supp(\tilde{g}^k) \subset B$. Now, set $f_{k,n}:=\tilde{g}^k \mid_{\X_n} \in \Lip_{\rm bs}(\X_n)$; it is immediate that, for any $k \in \N$, $f_{k,n} \to g^k$ in $L^2$-strong. Furthermore, using the upper semicontinuity of $\lip_a (\tilde{g}^k)$ on $\Z$, we obtain
\begin{equation*}\begin{split}
\int_{\X_\infty} \left(\lip_a (g^k)\right)^2 \d\mm_\infty&=\int_\Z \left(\lip_a (\tilde{g}^k)\right)^2 \d\mm_\infty \ge \limsup_{n\to\infty} \int_\Z \left(\lip_a (\tilde{g}^k)\right)^2 \d\mm_n \\
&\ge \limsup_{n\to\infty}\int_{\X_n} \left(\lip_a (f_{k,n})\right)^2 \d\mm_n \ge \limsup_n {\rm Ch}(f_{k,n}).\end{split}\end{equation*}
We can then conclude with a diagonalization argument, i.e. setting $f_n:=f_{k(n), n}$ for a suitable sequence $k(n) \to \infty$. This is justified as the $L^2$-strong convergence is metrizable (cf.\ \cite{GMS15,AmbrosioHonda17}). 
\end{proof}
The rest of the manuscript is devoted to dealing with the general case with possibly negative local curvature lower bounds. We will implement the same strategy as in \cite{NobiliRenziVitillaro25}, but the local setting presents a key challenge in deducing local essential nonbranching properties. We will overcome these difficulties by first studying localized properties of the heat flow. Before passing to this analysis, we show an application of our methods in the presence of nonnegative curvature.
\subsection{Application: Euclidean weak tangents}
In this part, we shall exploit the scaling properties of Ricci curvature lower bounds along a blow-up sequence, and show that the weak tangents of spaces satisfy the (global) ${\sf RCD}$ condition. This easily implies the existence of Euclidean weak tangents thanks to \cite{GigliMondinoRajala15}.
    \begin{definition}[Weak tangents]\label{Def:weakTan}
        Let $\Xdm$ be a metric measure space and let ${x}\in\supp(\mm)$, a pointed metric measure space $(Y,\rho,\mu,y)$ is a weak tangent to $\Xdm$ at ${x}$ if there exists a sequence of radii $s_n\downarrow0$ such that 
        \begin{align*}
            (\X,\sfd_{s_n},\mm_{{x}}^{s_n},{x})\xrightarrow{pmG}(Y,\rho,\mu,y),\qquad \text{as }n\to\infty,
        \end{align*}
        where 
        \begin{align*}
            \sfd_{s}=\frac{\sfd}{s}\quad\textit{and}\quad\mm_{{x}}^{s}=\left(\int_{B_s({x})}1-\frac{\sfd(\cdot,{x})}{s}\,\d\mm \right)^{-1}\mm.
        \end{align*}
        The collection of weak tangents at the point ${x}$ is denoted by ${\rm Tan}(\X,\sfd,\mm,{x})$.
    \end{definition}
    We are ready to show next our application.
\begin{proof}[Proof of Theorem \ref{thm:weak tangents intro}]
    Let $(Y,\rho,\mu,y) \in {\rm Tan}(\X,\sfd,\mm, x)$ be a weak tangent arising from a pmG-convergent sequence $(\X,\sfd_{s_n},\mm_{{x}}^{s_n},{x})$, as in Definition \ref{Def:weakTan}. It is  straightforward to check $\Omega$ satisfies ${\sf CD }_{loc}(s_n^2\, K(\cdot),N(\cdot))$, when regarded as an open subset of the metric measure space $(\X,\sfd_{s_n},\mm_{x}^{s_i},x)$ with local radius $r_n(x)\coloneqq \frac{r(x)}{s_n}$, for all $n\in\N$. Since, by the scaling, it holds that $r_n(x)\to+\infty$, then by Theorem \ref{thm:stability dimensional CDloc} and Proposition \ref{MainT K>=0} (applied with $\Omega_\infty=Y$) we conclude that any weak tangent satisfies ${\sf RCD}(0,N({x}))$ globally. This settles the first conclusion.

    The last conclusion, being local, will follow if we can show that for all $x\in\Omega$, there are Euclidean weak tangents at $\mm$-a.e.\ $x'\in B_{\frac{r(x)}{10}}(x)$.
    By \cite[Theorem 3.15]{GMS15}, and employing in particular the so-called intrinsic approach (\cite[Definition 3.8]{GMS15}), it follows that, for all $x'\in B_{s}(x)$, the weak tangent enjoys the following localization property: 
    $$ \mathrm{Tan}(\X, \sfd, \mm, x')=\mathrm{Tan}\left(B_s(x), \sfd, \mm\res B_s(x),x'\right).$$ 
    By Remark \ref{rmk:doubling small ball}, we know that $\left(B_{\frac{r(x)}{10}}(x), \sfd, \mm\res B_{\frac{r(x)}{10}}(x) \right)$ is doubling. Therefore, for any $s_n\downarrow0$, it holds that $\left(B_{\frac{r(x)}{10}}(x), \sfd_{s_n}, (\mm\res B_{\frac{r(x)}{10}}(x))_{x'}^{s_n},x'\right)$
    is a family of uniformly doubling pointed metric measure spaces. Then, by \cite[Proposition 2.2]{GigliMondinoRajala15} there exists, up to taking a  subsequence, a weak tangent $(Y,\rho,\mu,y)\in \mathrm{Tan}\left(B_{\frac{r(x)}{10}}(x), \sfd, \mm\res B_{\frac{r(x)}{10}}(x),x'\right)$ and, by the previous discussion, this satisfies the ${\sf RCD}(0,N(x')) $ condition.  Note that after the restriction to the ball $B_{\frac{r(x)}{10}}(x)$, thanks to the uniform doubling property, it is possible to work with the stronger notion of pmGH-convergence given by \cite[Definition 2.1]{GigliMondinoRajala15}, hence, we can apply the results from  \cite{GigliMondinoRajala15}. In particular, by \cite[Theorem 1.1]{GigliMondinoRajala15}, there exists a point $y' \in Y$ such that $(\mathbb{R}^n, \sfd_{\rm eu}, \mathcal{L}^n, 0) \in \mathrm{Tan}(Y, \rho, \mu, y')$ for some $n\le N(x')$.
    Finally, thanks to the metric version of Preiss' theorem \cite[Theorem 3.2]{GigliMondinoRajala15} (which states that tangents of tangents are tangents), we can conclude that, for $\mm$-a.e. $x'\in B_{\frac{r(x)}{10}}(x)$,
    \begin{align*}
        (\mathbb{R}^n, \sfd_{\rm eu}, \mathcal{L}^n, 0) \in \mathrm{Tan}(Y, \rho, \mu, y')\subseteq \mathrm{Tan}\left(B_{\frac{r(x)}{10}}(x), \sfd, \mm\res B_{\frac{r(x)}{10}}(x), x'\right) =\mathrm{Tan}(\X, \sfd, \mm, x').
    \end{align*}
    This concludes the proof.
\end{proof}
\section{Properties of the heat kernel on John domains}\label{sec:heat}

The goal of this section is to deduce both qualitative and quantitative properties of the kernel associated with the heat flow localized on suitable domains (namely, John domains) of metric measure spaces locally satisfying the doubling condition and Poincaré inequalities. The main result of the section will be Theorem \ref{th:heatonOmega}. Since we will invoke results from the literature on general Dirichlet forms, we start with some basic preliminaries about them (though we will only be interested in the form induced by a quadratic Cheeger energy).
\subsection{General Dirichlet forms}
We only consider the classical case of \emph{symmetric} Dirichlet forms in a $\sigma$-finite measure space $(\X, \mm)$, 
see for example \cite{FOT94} for a standard reference. { For simplicity, we omit from the notation the $\sigma$-algebra where $\mm$ is
defined, which will always be the Borel $\sigma$-algebra}.
\begin{definition}[Dirichlet form]
A \emph{(symmetric) Dirichlet form on $L^2(\X, \mm)$} is a pair $(\mathcal{E}, \mathcal{F})$, where $\mathcal{F}$ is a dense linear subspace of $L^2(\X, \mm)$ and $\mathcal{E}$ is a bilinear symmetric positive semidefinite form $\mathcal{E}: \mathcal{F} \times \mathcal{F} \to \R$ satisfying the following properties:
\begin{itemize}
    \item (Closedness) $\mathcal{F}$ is a Hilbert space endowed with the inner product $\mathcal{E}_1(f,g):=\mathcal{E}(f,g)+~\int fg \, \d\mm$;
    \item (Markovian property) for all $f \in \mathcal{F}$, $\eta: \R \to \R$ 1-Lipschitz with $\eta(0)=0$, setting $g:=\eta\circ f$, we have 
    $g\in \mathcal{F}$ and $\mathcal{E}(g,g) \le \mathcal{E}(f,f)$.
\end{itemize}
\end{definition}
\begin{definition}[Heat semigroup]
Let $(\X, \mm)$ be a $\sigma$-finite measure space. Given a Dirichlet form $(\mathcal{E}, \mathcal{F})$, we call \emph{heat semigroup} induced by $\mathcal{E}$ the family of operators $\mathbf{h}_t: L^2(\X, \mm) \to \mathcal{F}$, $t>0$, s.t. for every $f_0 \in L^2(\X, \mm)$ the family $f_t:=\mathbf{h}_t(f_0)$ represents the $L^2$-gradient flow starting from $f_0$ of the functional
$$
E(f):=\begin{cases}
    \frac12 \mathcal{E}(f,f), & \text{if } f\in \mathcal{F}, \\
    + \infty, &\text{if } f \in L^2(\X, \mm) \setminus \mathcal{F}.
\end{cases}$$
\end{definition}
Notice that the gradient flow of $E$ is well-defined as the closedness assumption guarantees that $E$ is lower semicontinuous (besides being convex). Equivalently (see e.g. \cite[Section 1.3]{FOT94}), $\mathbf{h}_t$ could be defined as the semigroup generated by the unique (self-adjoint negative semidefinite) linear operator $\Delta_{\mathcal{E}}: \mathcal{D}(\Delta_{\mathcal{E}}) \subset \mathcal{F}\to L^2(\X, \mm)$ s.t. $\mathcal{E}(f,g)=-\int \Delta_{\mathcal{E}}f \cdot g \,\d\mm$. This gives in particular that $f \mapsto \mathbf{h}_t(f)$ is linear (and thus we can use the notation $\mathbf{h}_tf$). As a consequence of the Markovian property, if $f \ge 0$ then $\mathbf{h}_t f \ge 0$ for all $t>0$.
\begin{definition}\label{def: Diri_regu_loca}
Let $(\X, \tau)$ be a locally compact, separable and Hausdorff space, and let $\mm$ be a positive Radon measure on $\X$ s.t. $\supp(\mm)=\X$. We say that a Dirichlet form $(\mathcal{E}, \mathcal{F})$ on $L^2(\X, \mm)$ is
\begin{itemize}
     \item \emph{regular} if $\mathcal{F} \cap C_c(\X)$ is dense both in $(\mathcal{F}, \sqrt{\mathcal{E}_1})$ and in $C_c(\X)$, endowed with the sup norm;
    \item \emph{strongly local} if $\mathcal{E}(f,g)=0$ for all $f,g \in \mathcal{F}$ s.t. $f$ is constant on a neighbourhood of $\supp(g)$.
\end{itemize}
\end{definition}
The Dirichlet form associated to a quadratic Cheeger energy is always strongly local by locality of weak gradients (see Theorem \ref{Calculus-weakg}). The regularity property is a bit more delicate, but it is at least satisfied if the ambient space is compact by the density in energy of Lipschitz functions (see e.g. \cite[Theorem 4.3]{AmbrosioGigliSavare11-3}).
\begin{remark}\label{rem:ht1=1}\rm
 If $\mathcal{E}$ is strongly local and $1 \in \mathcal{F}$ (in particular $\mm(\X) < \infty$), then it is easily seen that $\mathbf{h}_t 1=1$ for all $t>0$ (i.e. $\mathbf{h}_t$ is mass-preserving).  \fr 
\end{remark}
Let $(\mathcal{E}, \mathcal{F})$ be a symmetric regular Dirichlet form. Then, { under the assumptions of Definition~\ref{def: Diri_regu_loca}}, we can write
$$\mathcal{E}(f,g)=\int \d\Gamma(f,g),$$
where $\Gamma$ is measure-valued \emph{carré du champ}, i.e. a positive semidefinite symmetric bilinear form on $\mathcal{F}$ with values in the signed Radon measures on $\X$, the so-called energy measure (see e.g. \cite[Section 3.2]{FOT94}). The map $f \mapsto \Gamma(f,f)$ can be canonically extended to 
$$\mathcal{F}_{loc}:=\{f \in L^2_{loc}(\X, \mm): \forall K \subseteq \X \mbox{ compact } \exists\, g \in \mathcal{F} \mbox{ s.t. } f=g \, \,\mm\mbox{-a.e. in } K\}.$$
In the special case of quadratic Cheeger energies $\Gamma(f,g)$ can be represented by $L^1$ functions, with $\Gamma(f,f)=|\nabla f|^2$ for all $f \in H^{1,2}(\X)$. 
\begin{definition}\label{def:int-dist}
Let $(\X, \tau)$ be a locally compact and separable Hausdorff space, and let $\mm$ be a positive Radon measure on $\X$ s.t. $\supp(\mm)=\X$. Let $(\mathcal{E}, \mathcal{F})$ be a symmetric regular Dirichlet form. For all $x,y \in \X$, we define the \emph{intrinsic (pseudo)distance} induced by $\mathcal{E}$ by
$$\sfd_{\mathcal{E}}(x,y):=\sup \{f(x)-f(y): f \in \mathcal{F}_{loc} \cap C(\X), \, \Gamma(f,f) \le \mm \mbox{ on }\X\}.$$
\end{definition}
Notice that, even though $\sfd_\mathcal{E}$ is lower semicontinuous in $\X\times\X$, it might be degenerate (i.e. $\sfd_{\mathcal{E}}(x,y)=0$ for some $x \neq y$) or take the value $+\infty$. We will denote $B^{\mathcal{E}}_r(x):=\{y \in \X: \sfd_{\mathcal{E}}(x,y) < r \}$.
\begin{definition}
Let $(\X, \tau)$ be a locally compact, separable and Hausdorff topological space, and let $\varnothing \neq U \subset \X$ be open. We say that a symmetric regular Dirichlet form $(\mathcal{E}, \mathcal{F})$ is \emph{strongly regular} on $U$ if $\sfd_{\mathcal{E}}$ is a (finite-valued) distance on $U$ and it induces on $U$ the original topology $\tau $.
\end{definition}
\subsection{Heat kernel: definitions and properties}
We now recall some results about the existence of a continuous kernel for the heat semigroup of an abstract Dirichlet form, and pointwise upper Gaussian bounds on the kernel itself. We recall that our standing assumptions on the metric measure structure $\Xdm$ is that $(\X,\sfd)$ is a locally complete  separable metric space and $\mm$ is a nonnegative measure satisfying \eqref{eq:growth_cond}, in particular it is locally finite. According to the literature of Dirichlet forms, it is also standard to work with $(\X,\sfd)$ that is locally compact, and this is often implied by the fact that $\mm$ is locally doubling (see \cite{Sturm95II}). In our setting, we shall verify these assumptions by working on John domains of a PI space, which guarantee useful properties on the heat semigroup.
\begin{definition}[Heat kernel]\label{def:heat kernel}
Given a $\sigma$-finite measure space $(\X, \mm)$, and a symmetric Dirichlet form $(\mathcal{E}, \mathcal{F})$ on $L^2(\X, \mm)$, we say that a measurable function $p:(0, \infty) \times \X \times \X \to \R_+$ is a \emph{heat kernel} for $\mathcal{E}$ if, denoted by $(\mathbf{h}_t)_{t>0}$ the $L^2$ heat semigroup induced by $\mathcal{E}$, there exist a $\mm$-negligible subset $N \subseteq \X$ s.t. for all $s,t>0$ and $x,y \in X \setminus N$ we have
\begin{align}
 &\mathbf{h}_tf(x)=\int p_t(x,z)f(z) \, \d\mm(z) \qquad \forall f \in L^2(\X); \\
 &p_t(x,y)=p_t(y,x); \\
 &p(s+t,x,y)=\int p(t,x,z)p(s,z,y) \d\mm(z), \label{eq:semigroup_prop}
 \end{align}
where $p_t(x,y):=p(t,x,y)$.
\end{definition}
A natural question is whether the heat kernel of a Dirichlet form satisfies Gaussian decay estimates as in the Euclidean case. In the setting of general Dirichlet forms, a key contribution was given by the papers \cite{Sturm95II,Sturm96III}. In \cite{Sturm96III} it is shown that, if the doubling condition and a Poincaré inequality hold globally on $(\X, \sfd_\mathcal{E}, \mm)$, and under the assumption that the form $\mathcal{E}$ is strongly local and strongly regular on the whole $\X$, then $p$ admits a Hölder continuous representative, and it is possible to prove upper and lower Gaussian like bounds for $p_t(x,y)$, i.e. estimates of the form
\[\frac{1}{t^\beta}\exp{\left(-c_1\frac{\sfd_{\mathcal{E}}^2(x,y)}t+C_1t\right)}\lesssim p_t(x,y) \lesssim \frac{1}{t^\beta}\exp{\left(-c_2\frac{\sfd_{\mathcal{E}}^2(x,y)}t+C_2t\right)},\]
for constants $c_1, c_2, C_1, C_2>0, \beta >1$. However, when working with open domains $\Omega$ of a metric measure space (as it will be in our case), matching this framework would require a really strong assumption which makes $\overline{\Omega}$ a global PI space itself (for instance the fact that $\Omega$ is a uniform domain, see \cite{BjornSh07}).

Luckily, (uniform) Gaussian upper bounds require less structure on the space than double-sided bounds and can be achieved with a weaker condition on $(\X, \sfd_\mathcal{E}, \mm)$ than being globally a PI space. This was already clear from the paper \cite{Sturm95II}, where it was not required a Poincaré inequality, but only a Sobolev inequality (on all balls $\overline{B}^{\mathcal{E}}_r(x) \subseteq \X$). More recently, in \cite{Chen_et_al.21} it has been shown that the presence of a global Nash-type inequality, which can be easily deduced from a global Sobolev inequality, is by itself sufficient to infer (and actually characterize) the validity of upper Gaussian bounds. Hence, we have the following (see \cite[Theorem 1.2 and Corollary 1.3]{Chen_et_al.21} and also \cite[Theorem 2.4]{Sturm95II}).
\begin{theorem}[Heat kernel upper bounds]\label{th:upper_gauss}
 Let $(\X, \sfd)$ be a locally compact, separable metric space, let $\mm$ be a positive Radon measure with full support on $\X$ and let $(\mathcal{E}, \mathcal{F})$ be a symmetric, strongly local and regular Dirichlet form on $L^2(\X, \mm)$, with associated heat semigroup $\mathbf{h}_t$. Let $\delta \ge 0$ and $\beta>0$, then the following are equivalent.   
 \begin{enumerate}
     \item The following Nash-type inequality holds:
     \begin{equation}\label{eq:Nash_type}
     \|f\|_{L^2}^{2+2/\beta} \leq C(\mathcal{E}(f,f)+\delta\|f\|_{L^2}^2),\end{equation}
     for some $C>0$ and all $f \in \mathcal{F}$ with $\|f\|_{L^1} \leq 1$. 
     \item There exist a heat kernel $p$ for the form $\mathcal{E}$ and a $\mm$-negligible set $N \subseteq \X$ such that for any $\eps \in (0,1)$ there is a constant 
     $C_\eps >0$ satisfying 
     \begin{equation}\label{eq:upper_gauss}
     p_t(x,y) \le \frac{C_\eps}{t^\beta} \exp\left(-\frac{\sfd_\mathcal{E}^2(x,y)}{4(1+\eps)t}+\delta t\right),
     \qquad\forall t>0,\,\,\forall x,\,y\in\X\setminus N.
     \end{equation} 
 \end{enumerate}
\end{theorem}
\begin{corollary}\label{cor:upper_gauss}
In the setting of Theorem \ref{th:upper_gauss}, assume that $\mm(\X)<\infty$ and that the following Sobolev inequality holds for some $Q>2$ and all $f \in \mathcal{F}:$  
\begin{equation}\label{eq:Sobolev_in}
\|f-(f)_\X\|_{L^{2^*}(\X)}^2 \leq C\cdot \mathcal{E}(f,f),
\end{equation}
where $C>0$, $(f)_\X:=\frac{1}{\mm(\X)}\int f \d\mm$ and $2^*:=\frac{2Q}{Q-2}$. Then the upper Gaussian estimate \eqref{eq:upper_gauss} holds for all $t>0$ and { all} $x, y \in \X \setminus N$ (with $\beta=\frac Q2$ and $\delta=\frac 1C\mm(\X)^{-\frac 2Q} $), for some $\mm$-negligible set $N$.
\end{corollary}
\begin{proof}
We deduce the Nash-type inequality \eqref{eq:Nash_type} just by combining \eqref{eq:Sobolev_in} with the interpolation inequality $\|f\|_{L^2} \leq \|f\|_{L^{2^*}}^{\frac{Q}{Q+2}}\|f\|_{L^1}^{\frac{2}{Q+2}}$ and the inequality $\|f\|_{L^{2^*}} \le \|f-(f)_\X\|_{L^{2^*}}+\|f\|_{L^2}\mm(\X)^{-\frac 1Q}$. 
\end{proof}
We next directly invoke the results of \cite{Sturm96III} to show the following local estimate about the regularity of the heat kernel.
\begin{theorem}\label{th:cont-kernel}
 Let $(\X, \tau)$ be a locally compact and separable Hausdorff space, let $\mm$ be a positive Radon measure with full support on $\X$ and let $(\mathcal{E}, \mathcal{F})$ be a symmetric, strongly local and regular Dirichlet form on $L^2(\X, \mm)$ which admits a heat kernel $p$. Furthermore, equip $\X$ with the distance $\sfd_{\mathcal{E}}$. Let $U \subseteq\X$ be an open set such that for every $x \in U$ there exist an open neighbourhood $V_x \ni x$ and $r_x>0$ such that, if we set $B^{\mathcal{E}}(x):=B_{r_x}^{\mathcal{E}}(x)$, then $B^{\mathcal{E}}(x) \subset V_x$ and
\begin{enumerate}[(i)]
\item $\mathcal{E}$ is strongly regular on $V_x$; \label{kernel-(i)}
\item $\mm$ is doubling on $B^{\mathcal{E}}(x)$;\label{kernel-(ii)}
\item the following generalized weak $(2,2)$-Poincaré inequality holds on $B^{\mathcal{E}}(x)$:  \label{kernel-(iii)}
 \begin{equation}\label{eq:gen-Poincare}
        \int_{B^{\mathcal{E}}_r(x_0)}|f-(f)_{B^{\mathcal{E}}_r(x_0)}|^2 \d\mm
        \leq C_P^2 r^2 \int_{B^{\mathcal{E}}_{\sigma r}(x_0)} \d\Gamma(f,f), \qquad \forall f \in \mathcal{F},
    \end{equation}
    for any $x_0\in B^{\mathcal{E}}(x)$ and $r>0$ such that $\{ y \in \X \colon \sfd_{\mathcal E}(x_0,y)\le \sigma r\} \subset B^{\mathcal{E}}(x)$, where $C_P=C_P(x)>0$ and $\sigma=\sigma(x)>1$.
\end{enumerate}
Then the heat kernel $p$ admits a { pointwise} representative which is locally Hölder continuous on $(0,\infty) \times~U \times U$
{ with respect to $\sfd_{\mathcal{E}}$}.
\end{theorem}
\begin{proof}
The key estimate is given by \cite[Proposition 3.1]{Sturm96III}, which (note that, up to decreasing the radius, we can assume that $\overline{B}^\mathcal{E}(x)$ is compact, recall that $\sfd_{\mathcal{E}}$ is lower semicontinuous) yields that every $u:(0,\infty) \times \X \to \R$ of the form $u(t, \cdot)=\mathbf{h}_tf$ for some $f \in L^2(\X)$ admits a representative which satisfies the following: for all $x \in U$ there exist $C_H(x)>0$ and $0<\alpha(x)<1$ s.t. for all $t_0>0$
\begin{equation}\label{eq:Holder}
|u(t_1,y_1)-u(t_2,y_2)| \le C_H(x) \, \sup_{(t_0-r_x^2, t_0+r_x^2) \times B^{\mathcal{E}}(x)} |u| \, (|t_1-t_2|^{1/2}+\sfd_{\mathcal{E}}(y_1,y_2))^{\alpha(x)},
\end{equation}
for all $\displaystyle (t_1,y_1), (t_2,y_2) \in \left(t_0-\frac{r_x^2}8, t_0+\frac{r_x^2}8\right)\times B^\mathcal{E}_{r_x/2}(x)$. The estimate above follows from the fact that the doubling and Poincaré assumptions allow to prove a local Harnack inequality for $u$, see \cite[Theorem 3.5]{Sturm96III}. The basic idea is now to apply \eqref{eq:Holder} with $u=p(\cdot, y, \cdot)$ and $u=p(\cdot, \cdot, z)$, with $y$, $z$ close to $x_1$, $x_2 \in U$ respectively, and to use the Harnack inequality to control the term $ \sup_{(t_0-r_x^2, t_0+r_x^2) \times B^{\mathcal{E}}(x)} |u|$, after we have chosen a canonical representative for $p$ in $(0,\infty) \times U \times U$ (which requires again the Harnack inequality). For the details, we refer to \cite[Proposition 6.3]{Lierl18} or to \cite[Section 4.3.3]{BMGK12}.
\end{proof}
\begin{remark}\rm
Note that the setting of Theorem \ref{th:cont-kernel} differs from that in \cite{Sturm96III} only for two details: the fact that the radius $2r$ is replaced by $\sigma r$ in \eqref{eq:gen-Poincare} (as it is quite standard in the theory of PI spaces, see e.g. \cite{HK00}), which only affects the value of the local constants, and the fact that $\mathcal{E}$ is required to be strongly regular on $V_x$ only, which makes no difference due to the local nature of the results \cite[Proposition 3.1 and Theorem 3.5]{Sturm96III}.\fr 
\end{remark}
Now, in the special case where $\mathcal{E}$ is induced by a quadratic Cheeger energy (under suitable doubling and Poincaré assumptions), we are essentially able to replace the intrinsic distance $\sfd_{\mathcal{E}}$ with the original distance $\sfd$ of the space in the statements of Theorem \ref{th:upper_gauss} and Theorem \ref{th:cont-kernel}.
\begin{lemma}[Local Lipschitz equivalence of $\sfd$ and $\sfd_{\mathcal E}$]\label{lem:d-dE}
Let $(\X, \sfd, \mm)$ be an infinitesimally Hilbertian and compact metric measure space s.t. $\supp(\mm)=\X$, and let $\mathcal{E}$ be the Dirichlet form associated to the Cheeger energy on $\X$. Then $ \sfd_{\mathcal{E}} \ge \sfd$ on $\X \times \X$. \\
Moreover, let $B \subseteq\X$ be an open ball s.t. $\mm$ is doubling on $B$ and the weak $(1,2)$-Poincaré inequality holds on $B$. Then there exists a constant $C>0$ s.t. $\sfd_{\mathcal{E}} \le C\sfd$ on $\frac19 B \times \frac19 B$, i.e. $\d$ and $\sfd_{\mathcal{E}}$ are Lipschitz-equivalent on $\frac19 B$.
\end{lemma}
\begin{proof}
Since $\Gamma(f,f)=|\nabla f|^2$, the inequality $\sfd_{\mathcal{E}} \ge \d$ simply follows from testing Definition \ref{def:int-dist} with $f=\sfd(\cdot, y)$. The converse inequality (on $\frac19 B$) is instead a consequence of the following Lusin-Lipschitz estimate. There exists $C>0$ such that, for all $f \in W^{1,2}(B)$
\begin{equation}\label{eq:Lusin-Lip}
|f(x)-f(y)| \le \frac{C}{2}\sfd(x,y)\left( \sqrt{M_{4\sfd(x,y)} |\nabla f|^2(x)}+\sqrt{M_{4\sfd(x,y)} |\nabla f|^2(y)} \right),
\end{equation}
for all $x,y \in (\frac19 B) \setminus N$ (for some $\mm$-negligible set $N$), where $M_Rg(x):=\sup_{0<r<R} \fint_{B_r(x)} |g| \, \d\mm$. Indeed, if \eqref{eq:Lusin-Lip} holds, then for every $f \in W^{1,2}_{loc}(\X) \cap C(\X)$ s.t. $|\nabla f| \le 1$ $\mm$-a.e. we have
$$|f(x)-f(y)| \le C\sfd(x,y),$$
for $x,\,y \in (\frac19 B)\setminus N$, and then for all $x,y \in \frac19 B$ by continuity. The validity of \eqref{eq:Lusin-Lip} follows by the same lines of the proof of \cite[Theorem 3.2]{HK00}. Namely, one fixes $x, y$ which are Lebesgue points for $f$ (recall the doubling assumption) and then, by applying the Poincaré inequality on balls of the form $B_i(x)=B_{2^{-i}\sfd(x,y)}(x)$, $i \in \N\cup \{ 0\}$, one infers
$$
 |f(x)-f_{B_0(x)}| \le C_PC_D\sum_{i=0}^\infty 2^{-i}\sfd(x,y)\left(\fint_{2B_i(x)} |\nabla f|^2 \d\mm\right)^{1/2} \le 2C_PC_D \sfd(x,y)(M_{2\sfd(x,y)} |\nabla f|^2(x))^{1/2},
$$
and the analogous estimate for $|f(y)-f_{B_0(y)}|$, with $C_P>0$ as in \eqref{eq:weak_P} and $C_D>0$ as in \eqref{eq:locally_doub}. Moreover (notice that $4\overline{B}_0(x) \subset B$ for all $x,y \in \frac19 B$),
\begin{equation*}
\begin{split}
|f_{B_0(x)}-f_{B_0(y)}| &\le |f_{B_0(x)}-f_{2B_0(x)}|+|f_{B_0(y)}-f_{2B_0(x)}| \le 4C_PC_D\sfd(x,y)\left(\fint_{4B_0(x)} |\nabla f|^2 \d\mm\right)^{1/2} \\
&\le 4C_PC_D\sfd(x,y)(M_{4\sfd(x,y)} |\nabla f|^2(x))^{1/2},
\end{split}
\end{equation*}
which completes the proof of \eqref{eq:Lusin-Lip}.
\end{proof}
\begin{proposition}\label{prop:cont-kernel-Cheeger}
Let $(\X, \sfd, \mm)$ be an infinitesimally Hilbertian and compact metric measure space s.t. $\supp(\mm)=\X$. Let $\mathcal{E}$ be the Dirichlet form associated to the Cheeger energy on $\X$, let $\mathbf{h}_t$ be the associated $L^2$-heat flow, and assume that it admits a heat kernel $p$. Let $U \subseteq \X$ be an open set s.t. for every $x \in U$ there exists $r_x>0$ such that $\mm$ is doubling on $B_{r_x}(x)$ and the weak $(2,2)$-Poincaré inequality holds on $B_{r_x}(x)$. Then the heat kernel $p$ admits a representative which is locally Hölder continuous on $(0,\infty) \times U \times U$.    
\end{proposition}
\begin{proof}
We check the assumptions of Theorem \ref{th:cont-kernel}. First, $\mathcal{E}$ is regular as $\X$ is compact. Thanks to Lemma \ref{lem:d-dE}, we have $\sfd_{\mathcal{E}} \ge \d$, and for every $x \in U$ we can find a radius $r_x>0$ such that, if we set $B(x):=B_{r_x}(x)$, then $\overline{B}(x)$ is compact, $\mm$ satisfies the doubling property and the weak $(2,2)$-Poincaré inequality on $B(x)$ and $\sfd_{\mathcal{E}} \le C(x)\sfd$ on $B(x) \times B(x)$ for some $C(x)>0$. By combining these properties the validity of \eqref{kernel-(ii)} and \eqref{kernel-(iii)} in the ball $B^\mathcal{E}_{r_x}(x) \subseteq B(x)$ follows easily. Finally, the equivalence of $\sfd$ and $\sfd_{\mathcal{E}}$ on $B(x)$ also proves \eqref{kernel-(i)} with $V_x=B(x)$.
\end{proof}
\subsection{Upper Gaussian estimates on John domains}
The goal is to apply Corollary \ref{cor:upper_gauss} and Proposition \ref{prop:cont-kernel-Cheeger} on balls arising from local curvature dimension conditions. To this aim, we shall restrict the metric measure structure $(\X, \sfd, \mm)$ to closed balls and check the validity of a global Sobolev inequality on the same balls. To this aim, a crucial property to point out is that, metric balls in geodesic spaces have sufficiently nice connectivity properties. Namely, they are \emph{John domains} according to the following notion.
\begin{definition}[John domain]
Let $(\X, \sfd)$ be a metric space and let $\varnothing\neq \Omega\subseteq \X$ be open.
 We say that $\Omega$ is a \emph{John domain} if there exist a point $x_0 \in \Omega$ and a constant $c>0$ s.t. for every $x \in \Omega$ there exists a curve $\gamma:[0,L] \to \Omega$ (parametrized by arclength) s.t. $\gamma_0=x$, $\gamma(L)=x_0$ and
\[\sfd(\gamma_t, \Omega^c) \geq ct \qquad \forall t \in [0,L],\]
where the length $L$ depends on $x$.
\end{definition}
\begin{remark}\label{rem:balls_john}\rm
If $(\X, \sfd)$ is geodesic, then any ball $\Omega=B_r(y_0) \subset \X$ is trivially a John domain with $x_0=y_0$. Indeed, if we pick a geodesic $\gamma$ from $x \in \Omega$ to the center $y_0$ (parametrized by arclength) then
\[\sfd(\gamma_t, \Omega^c) \ge r-\sfd(\gamma_t, y_0)=r-(\sfd(x, y_0)-t) > t.\]
Note that it is actually enough to have geodesics from $y_0$ to any point in $B_r(y_0)$.
Actually in the locally compact case it suffices that $(\X, \sfd)$ is a length space, see \cite[Corollary 9.5]{HK00} (then $c>0$ is a universal constant valid for all balls). \fr  
\end{remark}
We report from \cite[Theorem 9.7]{HK00} a useful result for our goals.
\begin{theorem}[Sobolev inequality on John domains]\label{th:Sob_on_john}
Let $(\X, \sfd)$ be a metric space and let $\varnothing \neq \Omega\subseteq \X$ be a bounded John domain. Let $\mm$ be a Borel measure on $\X$ which satisfies the doubling property
\[
    \mm(B_{2r}(x)) \le 2^Q \mm(B_r(x)),
\]
for some $Q>0$,  all $x \in \Omega$ and $r \le 5\diam(\Omega)$. Assume that the pair $(u,g)$ satisfies the weak $(1,p)$-Poincaré inequality on $\Omega$ (with constant $C_P$), as well as the pair $(v_a^b, g\chi_{\{a<v \leq b \}})$ for all $0<a<b<\infty$, $c \in \R$ and $\eps \in \{-1,1\}$, where $v:=\eps(u-c)$ and $v_a^b:=\max\{0 ,\min\{(v-a),(b-a)\}\}$. Then, if $1 \leq p<Q$
\begin{equation}\label{eq:Sob_u_g}
\|u-u_\Omega\|_{L^{p^*}(\Omega)} \leq C\, \diam(\Omega)\,\|g\|_{L^p(\Omega)},
\end{equation}
where $p^*=\frac{Q p}{Q-p}$ and $C=C(Q, C_P)$.
\end{theorem}
\begin{theorem}[Local regularity of the heat kernel on John domains]\label{th:heatonOmega}
Let $\Omega$ be a bounded John domain in a metric space $(\X, \sfd)$. Let $\mm$ be a Borel measure on $\X$ s.t. $\mm(\partial \Omega)=0$ and $\supp(\mm \res \Omega)=\overline{\Omega}$, which satisfies the doubling property
\[
    \mm(B_{2r}(x)) \le C_D\mm(B_r(x)),
\]
for some $C_D>0$ and all $x \in \Omega$ and $r \le 5\diam(\Omega)$. Assume also that the weak $(2,2)$-Poincaré inequality holds in $\Omega$. Assume further that $\overline{\Omega}$ is infinitesimally Hilbertian, let $\mathcal{E}$ be the Dirichlet form associated to the Cheeger energy on $(\overline{\Omega}, \sfd, \mm \res{\overline{\Omega}})$ and let $\mathbf{h}_t$ be the associated $L^2$-heat flow. Then, there exists a function $p: (0, \infty) \times \Omega \times \Omega \to \R_+$ s.t.:
\begin{enumerate}
    \item $p$ is locally Hölder continuous in $(0,\infty) \times \Omega \times \Omega$;
    \item $p$ is the integral kernel of $\mathbf{h}_t$ (as in Definition \ref{def:heat kernel});
    \item we have the Gaussian estimate
    $$p_t(x,y) \leq \frac{C}{t^\beta}\exp{\left(-c\frac{\d^2(x,y)}t+Ct\right)},$$
    for some $\beta>0$, $0 < c,C<\infty$ and all $(t, x,y) \in (0, \infty) \times \Omega \times \Omega$.
\end{enumerate}
\end{theorem}

\begin{proof}
By assumption, for every function $u \in H^{1,2}(\overline{\Omega})$, the couple $(u, |\nabla u|)$ satisfies the weak $(1,2)$-Poincaré inequality on $\Omega$ (as well as all the truncations $(v_a^b, |\nabla u|\chi_{\{a<v \leq b \}})$ in the statement of Theorem \ref{th:Sob_on_john}, by locality of the weak gradient). Then, by Theorem \ref{th:Sob_on_john}, the couple $(u, |\nabla u|)$ satisfies the Sobolev inequality \eqref{eq:Sob_u_g}. Hence, we can apply Corollary \ref{cor:upper_gauss} with $\X=\overline{\Omega}$ ($\overline{\Omega}$ can be shown to be compact with the very same arguments of Lemma \ref{lem:loc_compact}) to deduce that $\mathbf{h}_t$ is representable by a kernel $p$ satisfying, for example (recall that $\sfd_\mathcal{E} \ge \d$),
 \[p_t(x,y) \le \frac{C}{t^\beta} \exp\left(-\frac{\d^2(x,y)}{5t}+C t\right),\]
 for all $t>0$ and $x,y \in \Omega \setminus N$, for some constant $C>0$ and some $\mm$-negligible set $N$, with $\beta=\frac{\log_2(C_D)}{2}>0$. To conclude, we apply Proposition \ref{prop:cont-kernel-Cheeger} (with { $\X=\overline{\Omega}$} and $U=\Omega$) to infer that $p$ admits a representative which is (locally Hölder) continuous on $\Omega$, exploiting that for any $x \in \Omega$ doubling and Poincaré hold on the neighbourhood $B_{r_x}(x)$ with $r_x=\sfd(x, \Omega^c)$.
\end{proof}

The regularity of the heat kernel guaranteed by Theorem \ref{th:heatonOmega} allows us (for any $\Omega$ as in the statement) to define a dual heat flow $\mathbf{h}_t^*$ on $\overline{\Omega}$ for measures concentrated on $\Omega$ (which include absolutely continuous measures). Indeed, the formula
\begin{equation}\label{eq:dualflow} \langle \mathbf{h}_t^* \mu, \phi \rangle:=\langle \mu, \mathbf{h}_t \phi\rangle, \qquad \phi \in C_b(\Omega),\end{equation}
defines an operator $\mathbf{h}_t^*: \PP(\Omega) \to \PP(\Omega)$ since $\mathbf{h}_t \phi=\int_\Omega p_t(\cdot,y)\phi(y) \d\mm(y)$ is still an element of $C_b(\Omega)$ (even for $\phi$ bounded Borel, i.e. $\mathbf{h}_t$ satisfies the \emph{strong Feller property} on $\Omega$). The fact that $\mathbf{h}_t^*\mu$ has unitary mass follows from $\mathbf{h}_t 1=1$ (which holds pointwise in $\Omega$, since $\mathbf{h}_t 1$ is continuous), recall Remark \ref{rem:ht1=1}. The following result provides the standard representation formula and additional properties for $\mathbf{h}_t^*\mu$.

\begin{corollary}[Representation of the dual heat flow on interior measures]\label{cor:heatflow_onmeas}
Let $\mu \in \PP(\Omega)$, with $\Omega\subset \X$ and $(\X,\sfd,\mm)$ as in Theorem \ref{th:heatonOmega}, and let $\mathbf{h}_t^*$ be the dual heat flow on $\PP(\Omega)$ defined by \eqref{eq:dualflow}. Then $\mathbf{h}_t^*\mu=f_t\mm$, for a function $f_t \in C_b(\Omega)$ given by
\begin{equation}\label{eq:heatmeas}
f_t(x)=\int p_t(x,y)  \d \mu(y).
\end{equation}
Furthermore, $f_t \mm$ is weakly convergent to $\mu$ as $t \to 0$.
\end{corollary}

\begin{proof}
The formula (\ref{eq:heatmeas}) is well-posed and defines a continuous bounded function (on the open domain $\Omega$) by continuity and boundedness of $p_t$. Moreover, for every $\phi \in C_b(\Omega)$
$$\langle 
\mathbf{h}_t^* \mu, \phi \rangle=\int \phi(x)\left(\int p_t(x,y)\d\mu(y)\right)\d\mm(x).$$
Now, recalling \eqref{eq:dualflow}, in order to check the weak convergence of $f_t\mm$ to $\mu$ it suffices to show that $\mathbf{h}_t \phi \to \phi$ pointwise (i.e. the statement for $\mu=\delta_x$, $x \in \Omega$). This follows by the Gaussian estimates on $p_t$, as for any $x \in \Omega$ and $\eps>0$ one infers
$$\int_{\Omega} p_t(x,y)|\phi(y)-\phi(x)| \, \d\mm(y) \leq \frac C{t^\beta}e^{-c\frac{\eps^2}t} e^{Ct}2\|\phi\|_{L^\infty}+\sup_{\sfd(x,y) \leq \eps} |\phi(y)-\phi(x)|,$$
which gives the claimed convergence by sending $t \to 0$ and then $\eps \to 0$.
\end{proof}
\section{From $\rcd_{loc}$ to strong $\CD_{loc}$}
In this section, the setting will be an open subset $\Omega$ of a metric measure space $(\X, \sfd, \mm)$ which satisfies the ${\sf RCD}_{loc}(K(\cdot),\infty)$ property. We shall also assume, without loss of generality since restricting to the support of the ambient measure has no impact on curvature dimension conditions, that $\supp(\mm \res \Omega)=\overline{\Omega}$. We will denote by $r(x)$ and $K(x)$ the parameters given by Definition \ref{def:CDloc open set}. Our main goal will be to show that $\Omega$ also satisfies the strong ${\sf CD}_{loc}(K(\cdot),\infty)$ property. To achieve this result, we will eventually follow the argument of \cite[Theorem 6.1]{AGMR15}: thus, our strategy will consist in proving a local version of the ${\rm EVI}$ inequality along the heat flow. Before starting, we note  that some important adaptations will be required: we will actually obtain a version of ${\rm EVI}$ modified by an error term. To control this term, we will have to assume that the reference measure $\mm$ is locally doubling on $\Omega$ (so that we can exploit the results of Section \ref{sec:heat}). The main result of this section will be Theorem \ref{th:RCD_implies_strongCD}.

\subsection{EVI with error term}
The goal of this subsection is to prove Theorem \ref{th:EVI Intro}. We shall first prove the following technical result and later use results from Appendix \ref{appendix}, collected there as they are independent from curvature dimension conditions.
%
%
%
%
%
%
\begin{theorem}[Derivative of the entropy]\label{th:EVI_sporca}
Let $\Xdm$ be a complete metric measure space, let $\varnothing \neq \Omega\subseteq  \X$ be open, and let us assume that $\Omega$ satisfies ${\sf RCD}_{loc}(K(\cdot),\infty)$. Let $B_{r}(x) \subseteq\Omega$ be a ball such that $r < \frac{r(x)}6$ and $\frac{\d}{\d s} \mm(B_s(x))$ exists at $s =r$. Let $\mathcal{E}$ be the Dirichlet form associated to the Cheeger energy on $L^2(\overline{B}_r(x))$. Let $\eta=f\mm \in \PP(B_r(x))$, with $f \in L^2\cap L^\infty(B_r(x))$, $\mathcal{E}(f,f) < \infty$ and let also $\sigma \in \PP(\overline{B}_r(x))$ such that $\supp (\sigma) $ is compact. 
Then there exists a Lipschitz Kantorovich potential $\varphi$  for $(\eta,\sigma)$ on $\overline{B}_{2r}(x)$  such that
$$\Ent_\mm (\sigma)-\Ent_\mm(\eta)-\frac{K(x)}{2}W_2^2(\sigma,\eta)\geq -\mathcal{E}(f, \varphi)-C\lim_{s\uparrow r}\|f\|_{L^\infty(B_r(x) \setminus B_s(x))},$$
where $C=C(r)=2r\frac{\d}{\d r} \mm(B_r(x))$.
\end{theorem}

\begin{proof} By Remark \ref{rem:restr_infHilb}, and by the assumptions on $\overline \Omega$, we have that for any open ball $B \subseteq \Omega$ with $\mm(\partial B)=0$ it holds that $(\overline{B}, \sfd, \mm \res{\overline{B}})$ is an infinitesimally Hilbertian space. Consequently, we can consider the Dirichlet form $\mathcal E(f,f) = \rmCh^{\overline B}(f)$ associated to a quadratic Cheeger energy defined for every $f \in L^2(\overline B)$. 

For any $n\in\N$, let $\chi_n$ be a Lipschitz function such that $\supp(\chi_n) \subset B_r(x)$, $0 \le \chi_n \le 1$, $\chi_n \equiv 1$ on $B_{(1-\frac{1}{n})r}(x)$ and $|\nabla\chi_n| \le \frac{n+1}{r}$.
    Set, for any $n\in\N$ and $\delta>0$,
    \[
        f_n:=\frac{\chi_n f}{\int_{B_r(x)}\chi_n f\,\d\mm}\quad\textrm{and}\quad f_{n, \delta}:=\mathbf{1}_{B_{r}(x)}\frac{ f_n +\delta}{1+\delta\mm(\overline{B}_r(x))}.
    \]
    Note that both $f_n \mm$ and $f_{n,\delta}\mm$ are probability measures supported on $\overline{B}_{r}(x)$. To simplify the notation, from now on we set $\delta':= \frac{\delta}{1+\delta\mm(B_r(x))}$ (notice that $f_{n, \delta} \equiv \delta'$ on $B_r(x) \setminus \supp(\chi_n)$).\\
    Consider $\sigma \in\PP(\overline{B}_r(x))$, and let $\pi\in \OptGeo(f_{n, \delta} \mm,\sigma)$ be the geodesic given by Proposition \ref{prop:goodgeodesics}, which then satisfies the displacement convexity of the entropy \eqref{eq:convexity Ent} and the density bound \eqref{eq:uniformbound}
    for some $t_0>0$. Note that the curves in its support will be supported in $\overline{B}_{2r}(x)$.
    Hereafter, we shall assume without loss of generality that $\mm(\partial B_{2r}(x))=0$; if this is not the case,  the construction can be still carried for a.e.\ $R$ such that $R>2r$.
    Let $\varphi_{n,\delta}\in\Lip(\overline{B}_{2r}(x))$ be a Kantorovich potential associated with $((\e_0)_\sharp\pi, (\e_1)_\sharp\pi)$ as in 
    Proposition \ref{prop:goodKant}, so that $\Lip(\varphi_{n,\delta})\leq 3r$.
   
    Since the properties given by Proposition \ref{prop:goodKant} are invariant under additions of constants, we may assume that
    \begin{equation}\label{eq:stime_Kant}
    |\varphi_{n, \delta}(y)| \le 6r^2\qquad \forall y \in \overline{B}_{2r}(x).
    \end{equation}
 Let $\rho_t \mm:=\mu_t:=(\e_t)_\sharp\pi$ for all $t \in [0,1]$. We set
$$
    \tilde{\rho}_t=
    \rho_t +\delta'\mathbf{1}_{\overline{B}_{2r}(x) \setminus\overline{B}_r (x)}, \quad \tilde{\mu}_t:=\tilde{\rho}_t \mm.
$$
In particular, we denote
\[\tilde{f}_{n,\delta}:=\tilde{\rho}_0=f_{n, \delta}+\delta'\mathbf{1}_{\overline{B}_{2r}(x) \setminus\overline{B}_r (x)}.\]
   We have
    \[
        \Ent_\mm (\sigma)-\Ent_\mm(\mu_0)-\frac{K(x)}{2}W_2^2(\sigma,\mu_0)\geq
        \frac{\Ent_\mm(\mu_t)-\Ent_\mm(\mu_0)}{t}=\frac{\Ent_\mm(\mu_t)-\Ent_\mm(f_{n, \delta} \mm)}{t}.
    \]
    We will now show that
    \begin{equation}\label{Eq:erroredelta1}
        \liminf_{t\to 0} \frac{\Ent_\mm(\mu_t)-\Ent_\mm(f_{n, \delta} \mm)}{t} \ge
         \liminf_{t\to 0}\frac{\Ent_\mm(\tilde{\mu}_t)-\Ent_\mm(\tilde{f}_{n,\delta} \mm)}{t}
         +k(\delta),
    \end{equation}
    where $\lim_{\delta\to 0} k(\delta)=0$. By construction, since $\pi$ is concentrated on geodesics of length smaller than $2r$, $\mu_t$ is concentrated on $B_{r+2tr}(x)$, hence
   \begin{align*}
&\liminf_{t\to0} \frac{\Ent_\mm(\mu_t)-\Ent_\mm(f_{n, \delta} \mm)}{t}-\frac{\Ent_\mm(\tilde{\mu}_t)-\Ent_\mm(\tilde{f}_{n,\delta} \mm)}{t}\\
&=
\liminf_{t\to0} 
    \frac{1}{t}\left(
    \int_{B_{r+2tr}(x)\setminus B_{r}(x)} \rho_t \log(\rho_t)-(\rho_t +\delta') \log(\rho_t+\delta')+
   \delta'\log (\delta') \d\mm  \right)
\\
&\geq
-2\omega(\delta')\lim_{t\to0} \frac{\mm(B_{r+2tr}(x))-\mm(B_{r}(x))}{t}=-
    4\omega(\delta')r \frac{\d}{\d s} \mm(B_s(x))\mid_{s=r},
\end{align*}
where $\omega$ is the modulus of uniform continuity of the function $z \mapsto z \log z$ in the interval $[0, M]$, with $M:=\sup_{t \in [0,t_0]}\|\rho_t\|_{L^\infty}+1 < \infty$.
Now, by convexity
\begin{align*}
    \frac{\Ent_\mm(\tilde{\mu}_t)-\Ent_\mm(\tilde{f}_{n,\delta} \mm)}{t}\geq
    \int \log (\tilde{f}_{n,\delta})\frac{\tilde{\rho}_t-\tilde{f}_{n,\delta}}{t}\d\mm&=\int \log (\tilde{f}_{n,\delta})\frac{\rho_t-f_{n,\delta}}{t}\d\mm\\
   & =\int \frac{\log(\tilde{f}_{n,\delta}\circ \e_t)-\log(\tilde{f}_{n,\delta}\circ \e_0)}{t}\d\pi.   
\end{align*}
Define the functions $ F_t, G_t : C([0, 1],\X)\to \R$ as
$$F_t =\frac{\log \,(\tilde{f}_{n,\delta}\circ \e_0)-\log \,(\tilde{f}_{n,\delta}\circ \e_t)}{\sfd(\gamma_t,\gamma_0)}\quad \mbox{and} \quad G_t=\frac{\varphi_{n,\delta}\circ \e_0 - \varphi_{n,\delta} \circ \e_t}{\sfd(\gamma_t,\gamma_0)},$$
with $F_t$, $G_t$ set to $0$ when $\gamma_t=\gamma_0$.
By the second part of Lemma \ref{lem5.2AGMR} (which we apply on $\overline{B}_{2r}(x)$), we have that 
\begin{equation}  \label{Eq:G_tandvarphi}
\lim_{t \to 0} \,G_t =\lim_{t\to 0}\frac{\sfd(\gamma_t,\gamma_0)}{t}= |\nabla\varphi_{n,\delta}|(\gamma_0)\quad\mbox{in} \,\,L^{2}( C ([0, 1], \X),\pi).
\end{equation}

We now apply the first part of Lemma \ref{lem5.2AGMR} to $h= \log(\tilde{f}_{n,\delta})$; note that $|\nabla h| \in L^2(\overline{B}_{2r}(x))$ as $\inf \tilde{f}_{n,\delta}=\delta'>~0$ (here we used the Leibniz rule and the chain rule for weak gradients from Theorem \ref{Calculus-weakg}).
Then, for some $t_0 >0$, we have 
\begin{equation}\label{Eq:boundF_t^2}
    \sup_{t\leq t_0} \int |F_t|^2 \d\pi <\infty.
\end{equation}
Using \eqref{Eq:G_tandvarphi} and \eqref{Eq:boundF_t^2} jointly we find
\begin{align*}
  \liminf_{t\to 0}  \frac{\Ent_\mm(\tilde{\mu}_t)-\Ent_\mm(\tilde{f}_{n,\delta} \mm)}{t}\geq
  \liminf_{t\to 0}\int \,-F_t (\gamma) \frac{\sfd(\gamma_0 ,\gamma_t )}{t}\d\pi
  =-\limsup_{t\to0}\int F_t G_t \d\pi.
\end{align*}
Applying again Lemma \ref{lem5.2AGMR} with $h=\varphi_{n,\delta}+\varepsilon \log(\tilde{f}_{n, \delta})$ for $\varepsilon>0 $ we get 
\begin{equation}\label{Eq:incrementFtGt}
    \int |\nabla(\varphi_{n,\delta}+\eps\log(\tilde{f}_{n, \delta}))| ^2 (\gamma_0) \,\d\pi \geq 
 \limsup_{t\to 0} \int|G_t+\varepsilon F_t|^2 \,\d\pi.
 \end{equation}
 By subtracting \eqref{Eq:G_tandvarphi} to  \eqref{Eq:incrementFtGt}, after a few computations and sending $\varepsilon$ to 0, while accounting for \eqref{Eq:boundF_t^2}, we find
 \[
    \limsup_{t\to0}\int F_t G_t \d\pi
    \leq 
    \liminf_{\varepsilon\to 0}\int \frac{|\nabla(\varphi_{n,\delta}+\eps\log (\tilde{f}_{n, \delta}))| ^2 - |\nabla\varphi_{n,\delta}| ^2}{2\varepsilon}f_{n, \delta}\d\mm 
 \]
Notice furthermore that, by locality (since $\tilde{f}_{n, \delta}\mid_{\supp ( f_{n, \delta})}=f_{n, \delta}$), we get
\begin{align*}
  \int \frac{|\nabla(\varphi_{n,\delta}+\eps\log (\tilde{f}_{n, \delta}))| ^2 - |\nabla\varphi_{n,\delta}| ^2}{2\varepsilon}f_{n, \delta}\d\mm &=
  \int \frac{|\nabla(\varphi_{n,\delta}+\eps\log(f_{n, \delta}))| ^2 - |\nabla\varphi_{n,\delta}| ^2}{2\varepsilon}f_{n, \delta}\d\mm\\
  &=\mathcal{E}_{f_{n,\delta}\mm}(\log(f_{n,\delta}),\varphi_{n,\delta} ),
\end{align*}
having lastly used \eqref{eq:weak gradi pesati}. All in all, we have proved that
\[
      \Ent_\mm (\sigma)-\Ent_\mm(f_{n,\delta} \mm)-\frac{K(x)}{2}W_2^2(\sigma,f_{n,\delta} \mm)
      \geq
      -\mathcal{E}_{f_{n,\delta}\mm}(\log(f_{n,\delta}),\varphi_{n,\delta} )+k(\delta).
      \]
      Now, we can apply the chain rule given by Theorem \ref{Teo:Dirichletchain}, exploiting that (on 
      $\overline{B}_{r}(x)$) $\varphi_{n,\delta}$ is Lipschitz and $\rmCh(\sqrt{f_{n,\delta}}) < \infty$ (using $f_{n,\delta} \ge \delta'>0$). Therefore
      \[ \mathcal{E}_{f_{n,\delta}\mm}(\log(f_{n,\delta}),\varphi_{n,\delta} )=\mathcal{E}(f_{n, \delta}\,,\varphi_{n,\delta} ).\]
      We have then proved that 
         \begin{equation}\label{Eq:beforepassingtothelimitdeltaandn}
        \Ent_\mm (\sigma)-\Ent_\mm(f_{n,\delta}\mm)-\frac{K(x)}{2}W_2^2(\sigma,f_{n, \delta}\mm)
        \geq     -\mathcal{E}(f_{n, \delta},\varphi_{n,\delta})+k(\delta).
    \end{equation}
    
    \noindent\textsc{Sending $\delta \to 0$}. 
    We note that, by strong locality 
    $$\mathcal{E}(f_{n, \delta},\varphi_{n,\delta})=\frac{1}{1+\delta\mm(\overline{B}_r(x))}\mathcal{E}(f_{n},\varphi_{n,\delta}).$$
   By Lemma \ref{lemma:Kantorovichpotentialscompactness} we can find a sequence $\delta_m \downarrow 0$ and a $3r$-Lipschitz Kantorovich potential $\varphi_n$ relative to $f_n \mm$ and $\sigma$ in $\overline{B}_{2r}(x)$ (still of the form \eqref{eq:choice_kant}) such that $\varphi_{\delta_m , n}\to \varphi_n$ pointwise. Note that, since by \eqref{eq:stime_Kant} $(\varphi_{\delta_m,n})_m$ is bounded in $H^{1,2}(\overline{B}_{2r}(x))$, this implies by reflexivity that $\varphi_{\delta_m , n}\weakto \varphi_n$ in $H^{1,2}(\overline{B}_{2r}(x))$ (cf.\ Remark \ref{rem:restr_infHilb}). Thus, by  sending $m\to\infty$ in \eqref{Eq:beforepassingtothelimitdeltaandn} with $\delta=\delta_m$, we conclude that
   \begin{equation}\label{Eq:afterdeltato0}
        \Ent_\mm (\sigma)-\Ent_\mm(f_{n}\mm)-\frac{K(x)}{2}W_2^2(\sigma,f_{n}\mm)   
        \geq     -\mathcal{E}(f_{n},\varphi_{n}).
    \end{equation}
 
 \noindent\textsc{Sending ${n}\to\mathbf{\infty}$.} Using again Lemma \ref{lemma:Kantorovichpotentialscompactness} we can find a sequence $n_m \uparrow \infty$ and a $3r$-Lipschitz Kantorovich potential $\varphi$ relative to $f \mm$ and $\sigma$ such that $\varphi_{ n_m}\to \varphi$ pointwise, and hence weakly in $H^{1,2}(\overline{B}_{2r}(x))$. We obtain $$ \lim_{m\to\infty} \mathcal{E}(f,\varphi_{n_m})= \mathcal{E}(f, \varphi). $$
 We note that $|\nabla \varphi_n| \le 2r$ on $B_r(x)$ by \eqref{eq:choice_kant})
 \begin{equation*}
 \begin{split}
     &|\mathcal{E}(f_n,\varphi_{n})-\mathcal{E}(f,\varphi_{_n})|\leq
     \left|\int \Gamma(f_n-f, \varphi_n) \d\mm\right|\leq \int |\nabla(f_n-f)| |\nabla\varphi_n|\d\mm\\
     &\le 2r \left( \int|1-c_n\chi_n||\nabla f| \d\mm+ c_n\int f|\nabla\chi_n| \d\mm\right) \\
     &
   \le 2r \left( \int|1-c_n\chi_n||\nabla f| \d\mm+c_n \frac{n+1}{n}\|f\|_{L^\infty (B_r(x)\setminus B_{(1-\frac{1}{n})r}(x))} \frac{\mm(B_r(x))-\mm(B_{(1-\frac 1n)r}(x))}{r/n}\right)
 \end{split}
\end{equation*}
where $c_n =\frac{1}{\int \chi_n f \d\mm}\to 1$ as $n\to\infty$.
All in all, we obtain, passing to the limit, that
\begin{equation}\label{Eq:stimaerrorelimiteinn}
   \limsup_{n\to\infty}  |\mathcal{E}(f_n\,,\varphi_{n})-\mathcal{E}(f\,,\varphi_{n})|\leq C\lim_{s\uparrow r}\|f\|_{L^\infty(B_r(x) \setminus B_s(x))}.
\end{equation}
We are finally able to conclude the proof by sending $m\to \infty$ in \eqref{Eq:afterdeltato0} (with $n=n_m$) and obtaining
\begin{equation*}
    \begin{split}
        \Ent_\mm (\sigma)-\Ent_\mm(f\mm)-\frac{K(x)}{2}W_2^2(\sigma,f\mm)& =\lim_{m\to\infty} \Ent_\mm (\sigma)-\Ent_\mm(f_{n_m}\mm)-\frac{K(x)}{2}W_2^2(\sigma,f_{n_m}\mm)\\
        &\geq \lim_{m\to\infty} -\mathcal{E}(f, \varphi_{n_m})- |\mathcal{E}(f_{n_m},\varphi_{n_m})-\mathcal{E}(f,\varphi_{n_m})|\\
        & \geq -\mathcal{E}(f, \varphi)-C\lim_{s\uparrow r}\|f\|_{L^\infty(B_r(x) \setminus B_s(x))}.
    \end{split}
\end{equation*}
\end{proof}
\begin{proof}[Proof of Theorem \ref{th:EVI Intro}]
    By combining Theorem \ref{th:EVI_sporca} with Proposition \ref{prop:der_Wass} we obtain  that 
    \begin{equation}\label{Eq:1.2inproof}
         \Ent_\mm (\sigma)-\Ent_\mm(\mu_t)-\frac{K(x)}{2}W_2^2(\mu_t ,\sigma) \ge \frac 12 \frac\d{\d t}W_2^2(\mu_t, \sigma)-C\lim_{r'\uparrow r}\|f_t\|_{L^\infty(B_r(x) \setminus B_{r'}(x))}
    \end{equation}
   for a.e.\ $t>0$, and any $\sigma$ compactly supported in $\overline{B}_r(x)$.
   To drop the assumption on $\sigma$ we multiply both sides of \eqref{Eq:1.2inproof} by $e^{K(x)t}$ and integrate to deduce, after manipulations
    \begin{align}\label{Eq:1.2proofintegralform}
            \frac{e^{K(x)t}}{2}  W_2^2(\mu_t,\sigma)-\frac{e^{K(x)s}}{2}  W_2^2(\mu_s,\sigma)  &\le  \left( {\rm Ent}_\mm(\sigma)  - {\rm Ent}_\mm(\mu_t)\right)\int_s^t e^{K(x)p}\,\d p \\
            &\qquad +C \int_s^t e^{K(x)p}\lim_{r' \uparrow r} \|f_p\|_{L^\infty(B_r(x) \setminus B_{r'}(x))}\,\d p, \nonumber
    \end{align}
    having also used that the function $t\mapsto {\rm Ent}_\mm(\mu_t)$ is nonincreasing by Proposition \ref{prop:dis_flussi} . Since \eqref{Eq:1.2proofintegralform} holds for any $\sigma$ compactly supported, it is possible to infer by density that it also holds for any $\sigma \in \PP(\overline{B}_r(x))$. A differentiation provides the conclusion.
\end{proof} 
\subsection{Local ${\rm EVI}$ implies strong ${\sf CD}_{loc}$ under locally doubling}%
From now on, we additionally assume that $\mm$ is locally doubling on an open set $\Omega$, and we denote by $r_D(x)$ the radius given by Definition \ref{def:loc_doubling}.
\begin{lemma}\label{lem:balls_satisfy}
Let $\Xdm$ be a complete metric measure space, let $\varnothing \neq \Omega\subseteq  \X$ be open, let us assume that $\Omega$ satisfies ${\sf RCD}_{loc}(K(\cdot),\infty)$ and that $\mm$ is locally doubling on $\Omega$. Then, for every $x \in \Omega$, any ball $B=B_r(x)$ with $r < \min\{\frac{r_D(x)}{21}, \frac{r(x)}6\}$ and $\mm(\partial B)=0$ satisfies the hypotheses of Theorem \ref{th:heatonOmega}. In particular, Corollary \ref{cor:heatflow_onmeas} holds on the domain $B$. Moreover, $\overline{B}$ is compact.
 \end{lemma}

 \begin{proof}
 First, recall that in our setting balls are John domains (see Remark \ref{rem:balls_john} and Proposition \ref{Prop:existgeodesicsCDdoubling}). The condition $r <  r_D(x)/21$ guarantees that $\overline{B}_{2\rho}(y) \subset B_{r_D(x)}(x)$ for every $y \in B$ and $\rho \le 5\diam(B)=10r$, so that the required doubling property holds. Notice that $2r < r_D(x)$ also implies the compactness of $\overline{B}$ by Lemma  \ref{lem:loc_compact}. Finally, since $r < \frac{r(x)}6$, Proposition \ref{prop:poincare} gives that the weak $(2,2)$-Poincaré inequality is also satisfied. 
 \end{proof}
Thanks to Lemma \ref{lem:balls_satisfy}, for sufficiently small balls $B$ around $x \in \Omega$ we can extend the heat flow $\mathbf{h}_t$ localized on $\overline{B}$ to every initial datum of the form $\mu_0 \in \mathscr{P}(B)$. The next lemma gives some properties of this flow.

\begin{lemma}\label{lem:der_Wass}
Let $\Xdm$ be a complete metric measure space, let $\varnothing \neq \Omega\subseteq  \X$ be open, and let us assume that $\Omega$ satisfies ${\sf RCD}_{loc}(K(\cdot),\infty)$ and that $\mm$ is locally doubling on $\Omega$. Fix $x \in \Omega$ and consider $B_r(x)$ with $r <  \min\{\frac{r_D(x)}{21}, \frac{r(x)}6\}$ and $\mm(\partial B_r(x))=0$. Let $\mathbf{h}_t$ be the heat semigroup on $(\overline{B}_r (x), \sfd, \mm \res{\overline{B}_{r}(x)})$, let $\mu_0 \in~\mathscr{P}(B_{r}(x))$, and let $\mu_t:=\mathbf{h}_t^*\mu_0=f_t \mm$. Then, for any $t > t_0 >0$ we have $(\mu_s)\in { AC}^2([t_0,t]; \mathscr{P}_2(\X))$ and
\begin{equation}\label{eq_dis_flussi2}
\Ent_\mm(\mu_{t_0}) \geq \Ent_\mm(\mu_t)+2\int_{t_0}^t \rmCh(\sqrt{f_s}) \,\d s+\frac 12\int_{t_0}^t |\dot{\mu}_s|^2 \, \d s.
\end{equation}
In particular, $t \mapsto \Ent_\mm(\mu_t)$ is non-increasing in $(0, \infty)$. Finally, if $\sigma \in \PP(\overline{B}_r (x))$, then for a.e.\ $t>0$ we have that any 
Lipschitz Kantorovich potential $\varphi_t:\overline{B}_{2r}(x) \to\R$ for $(\mu_t, \sigma)$ satisfies
$$\frac 12 \frac\d{\d t}W_2^2(\mu_t, \sigma)=-\mathcal{E}(f_t, \varphi_t),$$
where $\mathcal{E}$ is the Dirichlet form associated to the Cheeger energy on $L^2(\overline{B}_r(x))$.
\end{lemma}
\begin{proof}
 Apply Proposition \ref{prop:dis_flussi} and Proposition \ref{prop:der_Wass} using the semigroup property $f_t=~\mathbf{h}_{t-t_0}f_{t_0}$ for all $t_0>0$ with $f_{t_0} \in C_b(B)$ (see Corollary \ref{cor:heatflow_onmeas}).
\end{proof} 
Now, combining Lemma \ref{lem:der_Wass} with Theorem \ref{th:EVI_sporca} in the case $f=f_t=\mathbf{h}_t^*\mu_0$ we finally deduce the following ${\rm EVI}$ inequality with an error term:
$$\Ent_\mm (\sigma)-\Ent_\mm(\mu_t)-\frac{K(x)}{2}W_2^2(\mu_t,\sigma) \ge \frac 12 \frac\d{\d t}W_2^2(\mu_t, \sigma)-C\lim_{s\uparrow r}\|f_t\|_{L^\infty(B_r(x) \setminus B_s(x))}.$$
We are left with the estimate of the error $\|f_t\|_{L^\infty(B_r(x) \setminus B_s(x))}$ for $s<r$. Our aim is to show that, if we pick an initial datum $\mu_0$ whose support is well-contained in $B_r(x)$, then the error term is infinitesimal as $t \to 0$. The pointwise estimates on the heat kernel will be fundamental.
\begin{proposition}\label{prop:flux anulus vanish}
Let $\Xdm$ be a complete metric measure space, let $\varnothing \neq \Omega\subseteq \X$ be open, and let us assume that $\Omega$ satisfies ${\sf RCD}_{loc}(K(\cdot),\infty)$ and that $\mm$ is locally doubling on $\Omega$. Let $ B_{r}(x)$ be such that $r < \min \{\frac{r_D(x)}{21}, \frac{r(x)}6 \}$ and $\mm(\partial B_r(x))=0$. Then, for any
$\mu_0\in\PP(\overline{B}_{r'}(x))$ for some $r' < r$, for $f_t:=\mathbf{h}_t^*\mu_0$ one has
\begin{equation}
 \lim_{t \to 0}\|f_t\|_{L^\infty(B_r(x) \setminus B_s(x))}=0,\qquad \forall s \in (r',r).
\end{equation}
\end{proposition}
\begin{proof}
As $f_t(y)=\int p_t(y,z) \d\mu_0(z)$, a simple application of Theorem \ref{th:heatonOmega} (recall Lemma \ref{lem:balls_satisfy}) gives
\[
    f_t(y) \le \frac{C}{t^\beta}\exp \left(-c\frac{(s -r')^2}{t}+Ct \right),
\]
for all $y \in B_r(x) \setminus B_s(x)$, thus concluding the proof.
\end{proof}
We will now show how the local Riemannian curvature dimension condition implies the {strong} convexity of its defining entropy inequality. The proof follows the same strategy employed in \cite{DaneriSavare08} (see Proposition 3.1 and Theorem 3.2 therein).
\begin{theorem}\label{th:RCD_implies_strongCD}
    Let $\Xdm$ be a complete metric measure space, let $\varnothing \neq \Omega\subseteq \X$ be open, and let us assume that $\Omega$ satisfies ${\sf RCD}_{loc}(K(\cdot),\infty)$ and that $\mm$ is locally doubling on $\Omega$. Fix $x \in \Omega$, and let $\tilde{r}(x):=~\min \{\frac{r_D(x)}{21}, \frac{r(x)}6 \}$. Then equation \eqref{eq:convexity Ent} holds for all $\mu_0,\mu_1 \in \PP_2(\X)$  with $\supp(\mu_i)\subseteq B_{\frac{\tilde{r}(x)}{2}} (x)$ for $i=0,1$ and all \(\pi\in{\rm OptGeo}(\mu_0,\mu_1)\). In particular, $\Omega$ satisfies \emph{strong} ${\sf CD}_{loc}(K(\cdot),\infty)$ with local radius $\tilde{r}(\cdot)$.
\end{theorem}
\begin{proof}
    Clearly, it is enough to show the statement for boundedy supported probabilities $\mu_0,\mu_1\ll \mm$. Given such $\mu_0$, $\mu_1$, fix $r' < \frac{\tilde{r}(x)}{2}$ s.t. $\supp(\mu_i) \subset \overline{B}_{r'}(x)$ for $i=0,1$, and fix $r$ s.t. $2r' < r <  \tilde{r}(x)$ and $\frac{\d}{\d s} \mm(B_s(x))$ exists at $s=r$. Let then $r_0:=\frac{2r'+r}{2}$.

    Set $K:=K(x)$ for brevity and let $\sigma \in \PP(\overline{B}_{r}(x))$ be of the form $\sigma=g\mm$, $g \in L^\infty(\X)$, and combine Theorem \ref{th:EVI_sporca} with Proposition \ref{prop:der_Wass} to obtain
    \[
        \Ent_\mm (\sigma)-\Ent_\mm({\bf h}_t^*(\nu))-\frac{K}{2}W_2^2({\bf h}_t^*(\nu),\sigma) \ge \frac 12 \frac\d{\d t}W_2^2({\bf h}_t^*(\nu), \sigma)-C\|f_t\|_{L^\infty(B_{r}(x) \setminus B_{r_0}(x))},
    \]
    where $\nu \in \PP(\overline{B}_{2r'}(x))$ is arbitrary, $t>0$ and ${\bf h}_t^*(\nu) = f_t\mm$ denotes the heat flow of $\nu$ in $\overline{B}_r(x)$. \\
    We now multiply both sides by $e^{Kt}$ and integrate to deduce, after manipulations, the following
    \[
        \frac{e^{Kt}}{2}W_2^2({\bf h}_t^*(\nu),\sigma) - \frac12W_2^2(\nu,\sigma) \le  {\sf I}_K(t) \left( {\rm Ent}_\mm(\sigma) - {\rm Ent}_\mm({\bf h}_t^*(\nu)) \right) + C \int_0^t e^{Ks}\|f_s\|_{L^\infty(B_r(x) \setminus B_{r_0}(x))}\,\d s,
    \]
    where ${\sf I}_K(t):=\int_0^t e^{Ks} \d s$, having also used that $t\mapsto {\rm Ent}_\mm(\mu_t)$ is nonincreasing (cf.\ \eqref{eq_dis_flussi2}). Notice that the above  holds in fact for any $\sigma$ supported in $B_{r'}(x)$, by approximation, as every term is weakly continuous in $\sigma$. \\
    By using the upper Gaussian estimate (cf.\ Proposition \ref{prop:flux anulus vanish}), we deduce 
    \[
        \int_0^t e^{Ks}\|f_s\|_{L^\infty(B_r(x) \setminus B_{r_0}(x))}\,\d s \le {\sf I}_K(t)\omega(t),
    \]
    for some $\omega(t) \to 0$ as $t\downarrow 0$. 

    We can thus apply the above estimate twice, once with $\nu = ({\sf e}_s)_\sharp \pi,\sigma=\mu_0$, and then with $\nu = ({\sf e}_s)_\sharp \pi,\sigma=\mu_1$ to deduce
    \begin{align*}
        &{\sf I}_K(t)\frac{K}{2} s(1-s)W_2^2(\mu_0,\mu_1)=\frac{e^{Kt}-1}{2} s(1-s)W_2^2(\mu_0,\mu_1) \\ 
        &\le\frac{e^{Kt}}{2}s(1-s)\left( W_2({\bf h}_t^*(\nu),\mu_0) + W_2({\bf h}_t^*(\nu),\mu_1) \right)^2  - \frac{1}{2}\left(   s(1-s)^2 W_2^2(\mu_0,\mu_1)+s^2 (1-s) W_2^2(\mu_0,\mu_1) \right) \\ 
        &\le\frac{e^{Kt}}{2}\left(s W_2^2({\bf h}_t^*(\nu),\mu_1) + (1-s) W_2^2({\bf h}_t^*(\nu),\mu_0) \right) - \left( \frac{s}2 W_2^2(\nu,\mu_1) +  \frac{1-s}2 W_2^2(\nu,\mu_0)\right)  \\ 
        &\le {\sf I}_K(t) \left(s {\rm Ent}_\mm(\mu_1) +(1-s) {\rm Ent}_\mm(\mu_0) - {\rm Ent}_\mm({\bf h}_t^*(\nu)) \right) +  {\sf I}_K(t)\omega(t).
    \end{align*}
    Finally, dividing by ${\sf I}_K(t)$ and sending $t\downarrow 0$, we achieve
    \[
       \frac{K}{2} s(1-s)W_2^2(\mu_0,\mu_1) \le s {\rm Ent}_\mm(\mu_1) +(1-s) {\rm Ent}_\mm(\mu_0) -\liminf_{t\downarrow 0}{\rm Ent}_\mm({\bf h}_t^*(\nu)),
    \]
    which, after exploiting the weak lower semicontinuity of the relative entropy along with Corollary \ref{cor:heatflow_onmeas} (which guarantees that $\mathbf{h}_t^*(\nu) \weakto \nu$), gives the desired conclusion after rearranging the inequality.
\end{proof}
\subsection{Local essential non-branching}
We obtain the following improvement that will be crucial for the scope of our work.
\begin{corollary}\label{final_cor}
Under the hypotheses of Theorem \ref{th:RCD_implies_strongCD}, for every $\mu_0$, $\mu_1 \in \PP_2(\Omega)$ which are absolutely continuous with respect to $\mm$ and satisfy $\supp(\mu_i) \subseteq B_{\frac{\tilde{r}(x)}{4}}(x)$ for $i=0,1$ and some $x \in \Omega$, there is a unique $\pi \in \OptGeo(\mu_0, \mu_1)$, which is concentrated on a set of non-branching geodesics and is induced by a map.
\end{corollary}
\begin{proof}
Combine Theorem \ref{th:RCD_implies_strongCD} and Theorem \ref{Prop_NB}.    
\end{proof}
\section{Stability of infinitesimal Hilbertianity: the general case}
In this part, we prove our main results Theorem \ref{Main intro} and Theorem \ref{thm:mosco local intro}, as a byproduct of a more general analysis for infinite dimensional converging spaces.

\begin{theorem}\label{MainT}
Let $(\X_n, \sfd_n, \mm_n, x_n)$ be a sequence of pointed complete metric measure spaces that are pmG-converging to a pointed complete limit space $(\X_\infty, \sfd_\infty, \mm_\infty, x_\infty)$, and let $(\Z,\sfd)$ be a realization of the convergence. Let $\Omega_n\subseteq \X_n,\Omega_\infty \subseteq \X_\infty$ be non-empty open sets. Assume that $\Omega_n$ is infinitesimally Hilbertian and satisfies strong ${\sf CD}_{loc}(K_n(\cdot),\infty)$ for all $1 \leq n < \infty$. Furthermore, denoting by $r_n(\cdot)$ the local radius of $\Omega_n$, suppose the following: for all $x \in \Omega_\infty$ there are $y_n \in \Omega_n$ with $\sfd(y_n,x)\to 0$ satisfying $\liminf_n r_n(y_n) >0$ and $\liminf_n K_n(y_n) > -\infty$. Then $\Omega_\infty$ is infinitesimally Hilbertian.
\end{theorem}
Notice that the above theorem implies our first main result Theorem \ref{Main intro}.
\begin{proof}[Proof of Theorem \ref{Main intro}]
Since $\Omega_n$ satisfies ${\sf RCD}_{loc}(K(\cdot),N(\cdot))$, with $N(\cdot)$ finite, then $\mm$ is locally doubling on $\Omega_n$ (see Corollary \ref{cor:dim_CDloc}). Therefore, by Theorem \ref{th:RCD_implies_strongCD}, it holds that $\Omega_n$ satisfies strong ${\sf CD}_{loc}(K(\cdot),\infty)$, which implies that $\Omega_\infty$ is also infinitesimally Hilbertian by Theorem \ref{MainT}. Finally, the stability of the dimensional ${\sf CD}_{loc}$ condition follows by Theorem \ref{thm:stability dimensional CDloc}.
\end{proof} 
The proof of Theorem \ref{MainT} is postponed to the end of this section. We now carry out the main lower semicontinuity analysis for the Cheeger energy.
\begin{theorem}\label{PolyProp}
Let $(\X_n, \sfd_n, \mm_n, x_n)$ be a sequence of pointed complete metric measure spaces that are pmG-converging to a pointed complete limit space $(\X_\infty, \sfd_\infty, \mm_\infty, x_\infty)$, and let $(\Z,\sfd)$ be a realization of the convergence. Let $\Omega_n\subseteq \X_n,\Omega_\infty \subseteq \X_\infty$ be non-empty open sets. Assume that $\Omega_n$ satisfies strong ${\sf CD}_{loc}(K_n(\cdot),\infty)$ for all $1 \leq n < \infty$. Furthermore, denoting by $r_n(\cdot)$ the local radius of $\Omega_n$, suppose there is $x \in \Omega_\infty$ and $y_n \in \Omega_n$ with $\sfd(y_n,x)\to 0$ satisfying $\liminf_n r_n(y_n) >0$ and $\liminf_n K_n(y_n) > -\infty$. Let $f_\infty \in L^2(\X_\infty)$ with $\supp(f_\infty)\Subset B_r(x)$ with $r<r_0(x):=\frac 16\cdot  \liminf_n r_n(y_n)$. Then, for every sequence $f_n \in L^2(\mm_n)$ converging $L^2$-weak to $f_\infty \in L^2(\mm_\infty)$, we have
$${\rm Ch}(f_\infty) \leq \liminf_{n\to\infty} {\rm Ch}(f_n).$$
Finally, if $r_0(x)=+\infty$, the same conclusion holds for all $f_\infty \in L^2(\X_\infty)$ with no restriction on the support.
\end{theorem}
\begin{proof}
We follow the steps of the proof of \cite[Theorem 5.1]{NobiliRenziVitillaro25}. First, fixed $\eps>0$, notice that ${\rm Ch}(f_\infty)$ is the same as ${\rm Ch}^H(f_\infty)$ with $H\coloneqq \overline{B}_{r+\eps}(x)$, being $f_\infty$ supported in $B_r(x)$. Recall also by Theorem \ref{thm:Sobolevintegrated} that we can characterize ${\rm Ch^{1/2}}(f_\infty)$ as the minimal constant $C\ge 0$ such that
\begin{equation}\label{eq:lagr_Sob}
\int f_\infty(\gamma_1) - f_\infty(\gamma_0)\,\d \pi \le {\rm Comp}(\pi)^{1/2} \rmKe^{1/2}(\pi) C,
\end{equation}
holds for all 2-test plans $\pi$. Again, it is enough to check the above for all test plans $\pi$ with $\gamma_t \in \overline{B}_{r+\eps}(x)$ for all $t\in[0,1]$ and $\pi$-a.e.\ $\gamma$, being $f_\infty$ supported in $B_r(x)$. We thus consider fixed any such test plan $\pi$ and, in order to conclude, it suffices to show that (\ref{eq:lagr_Sob}) holds for the choice $C=\liminf_n {\rm Ch^{1/2}}(f_n)$. Without loss of generality, we shall also suppose that $C=\lim_n {\rm Ch^{1/2}}(f_n)$ along a suitable subsequence.

Fixed $M \in \N$, we build a sequence of approximating test plans $\pi_n^M \in  \PP\big(C([0,1],\X_n)$ as follows. Denoted $t_i:=i/M$ and $\rho_i:=(\e_{t_i})_\sharp\pi \in L^\infty(\mm_\infty)$ for $i=0, \ldots, M$, we construct sequences of densities $\tilde{\rho}_{i,n}$ s.t. $\tilde{\rho}_{i,n} \mm_n \in \PP(\X_n)$ and
\[\tilde{\rho}_{i,n} \to \rho_i \mbox{ in } L^2\mbox{-strong,} \qquad\supp(\tilde{\rho}_{i,n}) \subseteq \overline{B}_{r+2\eps}(x) \cap \X_n, \qquad \limsup_n \|\tilde{\rho}_{i,n}\|_{L^\infty(\mm_n)} \leq \|\rho_i\|_{L^\infty(\mm_\infty)} \]
for some fixed $\eps>0$ (see \cite[Lemma 4.1]{NobiliRenziVitillaro25}). Assuming with no loss of generality that $x \in \supp(\mm_\infty)$ (otherwise the statement is trivial), we can find a sequence $y_n \in \Omega_n$, $y_n \to x$ as in the hypotheses of Theorem \ref{MainT}. In particular, it holds that $K_n(y_n) \ge K$ for some $K>-\infty$. Hence, for $n$ large enough we have
\[\supp(\tilde{\rho}_{i,n}) \subseteq \overline{B}_{r+3\eps}(y_n) \cap \X_n.\]
Let $\tilde{\eta}_{i,n} \in \OptGeo(\tilde{\rho}_{i,n}\mm_n,\tilde{\rho}_{i+1,n}\mm_n)$. We would like to construct $\pi_n^M$ as a polygonal geodesic by gluing the plans $\tilde{\eta}_{i,n}$, using Theorem \ref{thm:rajalaCDqKN} to control ${\rm Comp}(\pi_n^M)$. However, in order to get a good estimate, some technical refinements are needed. Namely, denoted
\[\Gamma_{i,n}\coloneqq \Big\{ \gamma \in C([0,1],\X_n) \colon  \sfd^2(\gamma_0,\gamma_1) \le W_2(\tilde \rho_{i,n}\mm_n,\tilde\rho_{i+1,n}\mm_n)^{\frac 12} \Big\},\]
we apply \cite[Proposition 4.4]{NobiliRenziVitillaro25} in order to replace $\tilde{\rho}_{i,n}, \tilde{\eta}_{i,n}$ with $\rho_{i,n}, \eta_{i,n}$ as in the statement of that result. In particular, $\supp(\rho_{i,n}) \subset \supp(\tilde{\rho}_{i,n})$,
\begin{equation}\label{Gamma=0}
\eta_{i,n}(\Gamma_{i,n}^c)=0 \qquad \forall i=0, \ldots, M-1, 
\end{equation}
and there exists $\delta_{n,M}$ with $\limsup_M \limsup_n \delta_{n,M}=0$ s.t. $\|\rho_{i,n}- \tilde{\rho}_{i,n}\|_{L^1(\mm_n)} \leq \delta_{n,M}$ and $\|\rho_{i,n}\|_{L^\infty} \le (1+\delta_{n,M}) \| \tilde{\rho}_{i,n}\|_{L^\infty}$. The technical hypothesis $\sum_{i=0}^{M-1} \tilde{\eta}_{i,n}(\Gamma_{i,n}^c) \leq \frac 12$ is satisfied for $M \geq M_0$ and $n \geq n_0(M)$, see \cite[Eq. (4.12)]{NobiliRenziVitillaro25}. \\
Now, if $r < r_0(x)=\frac 16 \liminf_n r_n(y_n)$ then (choosing $\eps>0$ small enough) the supports of the measures $\rho_{i,n}$ satisfy the hypotheses in Theorem \ref{thm:rajalaCDqKN} for all $n$ large enough. By Corollary \ref{cor:improved_rajala} we know that the plan $\eta_{i,n}$ is the unique element of $\OptGeo(\rho_{i,n}\mm_n,\rho_{i+1,n}\mm_n)$, and so (\ref{Eq:rajalacompressionestimate}) holds for $\eta_{i,n}$, with the constant $D$ being exactly $\Lip(\eta_{i,n})$ (recall \eqref{Eq:rajalacompressionestimate stronc}). Hence
$${\rm Comp}(\eta_{i,n}) \leq e^{\frac{K^-}6 \Lip(\eta_{i,n})^2} \max\left\{\| \rho_{i,n}\|_{L^\infty(\mm_n)}, \|\rho_{i+1,n}\|_{L^\infty(\mm_n)})\right\},$$
which, using (\ref{Gamma=0}), yields
\[{\rm Comp}(\eta_{i,n}) \le e^{\frac{K^-}{6} W_2(\tilde{\rho}_{i,n}\mm_n,\tilde{\rho}_{i+1,n}\mm_n )^{\frac 12}}\max_{j=0,\dots,M}\|\rho_{j,n}\|_{L^\infty(\mm_n)}.\]
Then, with a gluing argument (see, e.g., \cite[Lemma 2.1.1]{G11}), we find $\pi_n^M \in \PP(C[0,1],\X_n))$ with the property
$$
   \eta_{i,n}=\left({\sf rest}_{t_{i}}^{t_{i+1}} \right)_\sharp \pi^M_{n},\qquad \forall i=0,1,...,M-1.
$$
Arguing as in \cite[Proposition 4.5]{NobiliRenziVitillaro25}, we can show that
$$\limsup_{M\to\infty}\limsup_{n \to\infty} \Comp(\pi^M_{n}) \le \Comp(\pi),\qquad \limsup_{M\to\infty}\limsup_{n \to \infty}\rmKe(\pi_{n}^M) \le \rmKe(\pi).$$
Furthermore, setting $\xi_n = \tilde{\rho}_{0,n},\zeta_n= \tilde{\rho}_{M,n}$ (notice that the latter does not depend on $M$), it holds
\begin{align*}
    &\xi_n \to \frac{\d (\e_0)_\sharp \pi }{\d \mm_{\infty}}\quad \text{in $L^2$-strong},& & \zeta_n \to \frac{\d (\e_1)_\sharp \pi }{\d \mm_{\infty}} \quad \text{in $L^2$-strong},\\
    &\limsup_{M\to\infty}\limsup_{n\to\infty}\left\| \xi_n - \frac{\d (\e_0)_\sharp \pi^{M}_{n}}{\d \mm_n}\right\|_{L^2(\mm_n)} =0,&
    &\limsup_{M\to\infty}\limsup_{n\to\infty}\left\| \zeta_n - \frac{\d (\e_1)_\sharp \pi^{M}_{n}}{\d \mm_n} \right\|_{L^2(\mm_n)}=0.
\end{align*}
Then, as in the proof of Theorem 5.1 in \cite{NobiliRenziVitillaro25}, we find sequences $M_k\to\infty$, and $n_k\ge n_0(M_k)\to\infty$ so that
\begin{itemize}
    \item[{\rm (a)}] $\rmKe(\pi_{n_k}^{M_k}) \le \rmKe(\pi) +\frac{2}{k}$ for all $k\in\N$; 
    \item[{\rm (b)}] $\Comp(\pi_{n_k}^{M_k}) \le \Comp(\pi) +\frac{2}{k}$ for all $k\in\N$;
    \item[{\rm (c)}]it holds as $k\to\infty $ that
    \[
       \frac{\d (\e_0)_\sharp \pi_{n_k}^{M_k} }{\d \mm_{n_k}} \to \frac{\d (\e_0)_\sharp \pi }{\d \mm_{\infty}},\qquad \frac{\d (\e_
       1)_\sharp \pi_{n_k}^{M_k} }{\d \mm_{n_k}}\to \frac{\d (\e_1)_\sharp \pi }{\d \mm_{\infty}}, \qquad \text{in $L^2$-strong}.
    \]
\end{itemize}
We are now ready to conclude the proof. For all $k\in\N$, using again Theorem \ref{thm:Sobolevintegrated}, we have
\[
\int f_{n_k}(\gamma_1)-f_{n_k}(\gamma_0)\,\d\pi_{n_k}^{M_k} \le \Comp\left(\pi_{n_k}^{M_k} \right)^{1/2}\rmKe^{1/2}\left(\pi_{n_k}^{M_k}\right)\rmCh^{1/2}(f_{n_k}).
\]
By property (c), and using \eqref{eq:coupling}, we have 
\begin{equation}
\int f_{n_k}(\gamma_1)-f_{n_k}(\gamma_0)\,\d\pi_{n_k}^{M_k} = 
\int f_{n_k}\,\d (\e_1)_\sharp \pi^{M_k}_{n_k} - \int f_{n_k}\,\d (\e_0)_\sharp \pi_{n_k}^{M_k} \to \int f_\infty(\gamma_1)-f_\infty(\gamma_0)\,\d\pi.
\label{eq:endpoint coupling}
\end{equation}
All in all, using the properties (a),(b), we have
\[
\begin{aligned}
\int f_\infty(\gamma_1)-f_\infty(\gamma_0)\,\d\pi  &=\lim_{k\to\infty} \int f_{n_k}(\gamma_1)-f_{n_k}(\gamma_0)\,\d\pi_{n_k}^{M_k} \\
&\le \limsup_{k\to\infty}
\Comp\left(\pi_{n_k}^{M_k} \right)^{1/2}\rmKe^{1/2}\left(\pi_{n_k}^{M_k}\right)\rmCh^{1/2}(f_{n_k})\\
&\le \Comp(\pi)^{1/2}\rmKe^{1/2}_2(\pi) \limsup_{k\to\infty} \rmCh^{1/2}(f_{n_k}).
\end{aligned}
\]
Recalling that $\limsup_k \rmCh^{1/2}(f_{n_k}) = \lim_n \rmCh^{1/2}(f_n)$, the  first conclusion is proven. 

Finally, if $r(x)=+\infty$, we argue by approximation as follows. Let $(\eta_k) \subset \Lip_{bs}(\Z)$ with $\supp(\eta_k)\subset B_{2k}(x), \eta_k\equiv 1$ on $B_k(x)$, $0\le \eta_k\le 1$ and $\Lip(\eta_k)\le 1/k$. If $f_n$ converges $L^2$-weak to $f_\infty$, then for each $k\in \N$, we have that $\eta_k f_n$ converges $L^2$-weak to $\eta_k f_\infty$. Then
\[
    \rmCh^{1/2}(f_\infty)\le \liminf_{k\to \infty}\liminf_{n\to\infty}\rmCh^{1/2}(\eta_k f_n)  \le \lim_{n\to\infty}\rmCh^{1/2}(f_n) + \liminf_{k\to \infty}\liminf_{n\to\infty}\frac{\|f_n\|_{L^2(\X_n)}}{k} = \liminf_{n\to\infty}\rmCh^{1/2}(f_n),
\]
having used the lower semicontinuity of the Cheeger energy, then a Leibniz rule argument followed by the triangular inequality, and lastly that $\liminf_{n\to\infty}\|f_n\|_{L^2(\X_n)}<\infty$. This concludes the proof.
\end{proof}
\begin{proof}[Proof of Theorem \ref{MainT}]
The proof is analogous to that of Step 2 in Proposition \ref{MainT K>=0}, here relying on the combination of Proposition \ref{Gammalimsup} (requiring no curvature assumptions) and Theorem \ref{PolyProp}.
\end{proof}
We next show our second main result of this work.
\begin{proof}[Proof of Theorem \ref{thm:mosco local intro}]
    The proof is directly implied by the combination of Proposition \ref{Gammalimsup} and Theorem \ref{PolyProp} for a suitable numerical constant, taking also into account Theorem \ref{th:RCD_implies_strongCD} and Corollary \ref{cor:dim_CDloc}.
\end{proof}
\begin{remark}\label{rmk:locally complete results}
    \rm
    Theorem \ref{Main intro} follows from the stability of the parallelogram identity and of the local curvature dimension condition on \emph{complete} spaces, the former relying on the local Mosco convergence property established in Theorem \ref{thm:mosco local intro}. However, the proofs of these results remain valid if $\Xdm$ is only assumed \emph{locally complete}, taking into account the following considerations. Indeed, the assumption $\liminf_n r_n(y_n)>0$ implies $\liminf_n r_{n,C}(y_n)>0$, where $r_{n,C}(y_n)$ denotes the local radius of completeness of $\Omega_n$. Thus, choosing
    \[
        r_\infty(x) \coloneqq \min\{  c\cdot \liminf_n r_n(y_n), r_C(x)\},
    \]
    all arguments can be localized to sufficiently small complete balls along the pmG-converging sequence. More precisely:
    \begin{itemize}
        \item The optimal transport arguments proving the stability of the curvature-dimension conditions (Theorems \ref{thm:stability CDloc} and Theorem \ref{thm:stability dimensional CDloc}) are local. 
         
        \item Theorem \ref{thm:mosco local intro} remains valid for sufficiently small balls with complete closure. Completeness is only used in the proof of the liminf inequality via Theorem \ref{thm:Sobolevintegrated} in conjunction with local optimal transport based arguments exploiting Corollary \ref{cor:improved_rajala} and Corollary \ref{final_cor} (possible by Theorem \ref{th:EVI Intro}). If the supports of $f_n$ and $f$ are contained in such a small ball, the proof carries over taking into account the following comment.
        
        \item Theorem \ref{th:EVI Intro} is also formulated for complete metric measure spaces, but again this is guaranteed for $\overline{B}$ small enough.

        \item The stability of infinitesimal Hilbertianity is also local and relies crucially on Theorem \ref{thm:mosco local intro} and Theorem \ref{thm:equivalent hilbert}. The former has been discussed, while the latter is only invoked with $H^{1,2}(\overline{B})$ Hilbert for $(\overline{B},\sfd,\mm\res{\overline B})$ complete metric measure space, which is the case for sufficiently small $B$.
       
       \fr
    \end{itemize}
\end{remark}

\appendix
\section{Technical calculus results}\label{appendix}
In this appendix we show two known technical results, with a proof tailored for bounded spaces. 

First, we show a version of the Metric Brenier's Theorem, whose proof is simpler thanks to the boundedness assumption (cf.\ \cite{AmbrosioGigliSavare11-2}).
\begin{theorem}\label{thm:metric brenier}
Let $(\X,\sfd,\mm)$ be a bounded complete metric measure space, let $\mu  \in \PP(\X)$ with $\mu=\rho \mm$ for some $\rho\in L^\infty(\X,\mm)$, and let $\nu\in \PP(\X)$. Let $\pi\in \OptGeo(\mu,\nu)$ and let $\varphi\in\Lip(\X)$ be a Kantorovich potential for $((\e_0)_\sharp\pi ,(\e_1)_\sharp\pi)$.
Assume that for some $0< s_0\leq 1$ we have that
\begin{equation}
    \sup_{s\in[0,s_0]}\|\rho_s\|_{L^\infty}<\infty,\qquad \text{where }(\e_s)_\sharp \pi = \rho_s\mm.
\end{equation}
Then, it holds
    \begin{equation}
        \sfd(\gamma_0,\gamma_1)=|\nabla\varphi|(\gamma_0)=|D^+\varphi|(\gamma_0) \qquad\pi\text{-a.e.\ } \gamma.
    \end{equation}
\end{theorem}
\begin{proof}
It will be enough to show that 
\begin{equation*}
    \int\rho|\nabla\varphi|^2  \d\mm=\int|\nabla\varphi|^2 (\gamma_0) \,\d\pi(\gamma)\geq \int \sfd^2(\gamma_0,\gamma_1) \,\d\pi(\gamma),
\end{equation*}
since, by Proposition \ref{Prop:slopeKantoPotentials}, we know that
\begin{equation}\label{Eq:inequalitykantpotdistslopeweakgrad}
    |\nabla\varphi|(\gamma_0)\leq|D^+ \varphi|(\gamma_0)\leq \sfd (\gamma_0,\gamma_1)\qquad \pi\text{-a.e.\ }\gamma.
\end{equation}
Using that $\varphi=(\varphi^c)^c$, then for $\pi$-a.e.\ $\gamma\in\Geo(\X)$, and for all $t \in (0,1]$ we get 
 \begin{align*}
        &\varphi(\gamma_0)-\varphi(\gamma_t)\geq \left( \frac{\sfd^2(\gamma_0, \gamma_1)}{2}  -\varphi^c(\gamma_1) \right) -\left( \frac{\sfd^2(\gamma_t, \gamma_1)}{2}  -\varphi^c(\gamma_1) \right)= \frac{2t - t^2}{2} \sfd^2(\gamma_0, \gamma_1) .       
    \end{align*}
Since $|\dot{\gamma}_t|=\sfd(\gamma_0,\gamma_1)$ for each  $\gamma \in \Geo(\X)$, then for $t\leq s_0$,  we infer by Jensen's inequality that
\begin{align*}
    |\varphi(\gamma_0) - \varphi(\gamma_t)|^2 \leq \left( \int_0^t |\nabla\varphi|(\gamma_s) \sfd(\gamma_1, \gamma_0) \, \d s \right)^2 \le t \sfd^2(\gamma_1, \gamma_0) \int_0^t |\nabla\varphi|^2(\gamma_s) \, \d s.
\end{align*}
All in all, dividing by $\sfd^2(\gamma_t, \gamma_0)=t^2 \sfd^2(\gamma_1, \gamma_0)$ and integrating with respect to $\pi$, we find 
\begin{align*}
    &\fint_0 ^t \int|\nabla\varphi|^2(\gamma_s)\,\d \pi\d s=\frac{1}{t}\int\int_0 ^t |\nabla\varphi|^2(\gamma_s)\,\d s\d\pi \geq\int\frac{|\varphi(\gamma_0) - \varphi(\gamma_t)|^2}{\sfd^2(\gamma_t, \gamma_0)}\,\d\pi\geq \\ 
    &\geq \frac{(2-t)^2}{4} \int \sfd^2(\gamma_1, \gamma_0) \,\d\pi.
\end{align*}
In order to conclude, it is sufficient to show that
\begin{align}\label{Eq:continuityDKmetricbrenier}
    \lim_{s\to0} \int|\nabla\varphi|^2\,\d((\e_s)_\sharp\pi)= \lim_{s\to0}\int\rho_s|\nabla\varphi|^2\,\d\mm=\int\rho|\nabla\varphi|^2\,\d\mm.
\end{align}
Indeed, for any $f\in L^1 (\X,\mm)$, $t<s_0$   and $\varepsilon>0$, we can choose $h\in\Lip(\X)$ such that $\|f-h\|_{L^1(\X)}\leq\varepsilon$, then
\begin{align*}
   & \left|\int\rho_t f-\rho f \,\d\mm\right|\leq 2\varepsilon\left(\sup_{s\in[0,s_0]}\|\rho_s\|_{L^\infty}\right)+\left|\int\rho_t h-\rho h \,\d\mm\right|\\ &\leq
   2\varepsilon\left(\sup_{s\in[0,s_0]}\|\rho_s\|_{L^\infty}\right)+\left|\int h(\gamma_0)-h(\gamma_t)\,\d\pi\right|
   \leq 2\varepsilon\left(\sup_{s\in[0,s_0]}\|\rho_s\|_{L^\infty}\right)+t\Lip(h)W_2(\mu, \nu).
\end{align*}
Taking the limit as $t \to 0$, since $\varepsilon$ is arbitrary, yields
\begin{equation*}
    \lim_{t\to 0}\left|\int\rho_t f-\rho f \,\d\mm\right|=0,\qquad \forall\,f\in L^1(\X).
\end{equation*}
Equation \eqref{Eq:continuityDKmetricbrenier} follows immediately by choosing $f=|\nabla\varphi|^2$.
\end{proof}
The next result revisits \cite[Lemma 5.2]{AGMR15} also for bounded metric spaces. 
\begin{lemma}\label{lem5.2AGMR}
Let $(\X, \sfd, \mm)$ be a bounded complete metric measure space. Let $\mu=f\mm,\,\sigma=g\mm\in~\PP(\X)$ with $f,g\in~L^\infty(\X)$. Assume that there exists $\pi\in\OptGeo(\mu,\sigma)$ such that  $\mu_t =  \rho_t\mm=(\e_t)_\sharp \pi$ satisfies the convexity inequality \eqref{eq:convexity Ent} and there exists $t_0>0$ such that
 $$C_*:=\sup_{t \in [0,t_0]}\|\rho_t\|_{L^\infty} < \infty.$$
 Then, we have:
\begin{itemize}
\item For $h \in H^{1,2}(\X)$, the following holds
\begin{equation}\label{Eq:ineqincremental}
\limsup_{t \downarrow 0} \int\left| \frac{h(\gamma_t)-h(\gamma_0)}{\sfd(\gamma_t, \gamma_0)}\right|^2 \d \pi(\gamma) \leq \int |\nabla h|^2(\gamma_0) \,
\d\pi(\gamma).
\end{equation} 
\item  For all Kantorovich potentials $\varphi$ relative to $(\mu,\sigma)$ with $|\nabla \varphi|$ bounded
one has
\begin{equation}\label{derKant} 
\lim_{t\downarrow 0} \frac{\varphi(\gamma_0)-\varphi(\gamma_t)}{\sfd(\gamma_0, \gamma_t)}= \lim_{t\downarrow 0} \frac{\sfd(\gamma_0,\gamma_t)}{t}=|\nabla\varphi|(\gamma_0) \quad \text{in } L^2(C([0,1];\X), \pi).
\end{equation}
\end{itemize} 
\end{lemma}
\begin{proof} We follow the proof of \cite[Lemma 5.2]{AGMR15} with some simplifications due to the fact that $\X$ is bounded.

By \eqref{eq:weak gradient} and applying Cauchy-Schwarz inequality, we find that, for any $t \in (0,t_0)$  and $\pi$-a.e. geodesic $\gamma$, it holds
\begin{equation}
\left| \frac{h(\gamma_t)-h(\gamma_0)}{\sfd(\gamma_t,\gamma_0)} \right|^2 \leq 
\frac{\left(\int_0^t |\nabla h|(\gamma_s)|\dot{\gamma}_s| \d s \right)^2 }
{\sfd ^2(\gamma_t,\gamma_0)} \leq \frac 1 t \int_0^t |\nabla h|^2(\gamma_s) \d s\,.
\end{equation}
 Therefore, applying twice Fubini's theorem, we get
\begin{equation*}
\int \left| \frac{h(\gamma_t)-h(\gamma_0)}{\sfd(\gamma_t,\gamma_0)} \right|^2 \d \pi(\gamma) \leq
 \int \left( \frac 1 t \int_0^t |\nabla h|^2(\gamma_s) \d s \right) \d \pi(\gamma) 
 = \int_\X \left(\frac{1}{t} \int_0^t \rho_s \d s \right) |\nabla h|^2 \,\d \mm.
\end{equation*}
We also note that 
\begin{equation*}
\lim_{t\downarrow 0} \int_\X \left(\frac{1}{t} \int_0^t \rho_s \d s \right) |\nabla h|^2\, \d \mm= \int_\X |\nabla h|^2 f\,\d \mm,
\end{equation*}
 since
\begin{equation}\label{eq:sigmat}
\frac{1}{t} \int_0^t \rho_s \d s \to f, \quad \text{in duality with } L^1(\X).
\end{equation}
This follows from the weak convergence $\rho_t \mm \to f\mm$ and  the uniform $L^\infty$ bound 
$$\frac{1}{t} \int_0^t \rho_s \d s \leq C_*,\qquad \forall t\le t_0.$$
This concludes the proof of the first statement.

Let us now turn to the second statement. By the metric Brenier Theorem  \ref{thm:metric brenier}, we infer the existence of
a Borel function $L$ satisfying
$L(\gamma_0):=\sfd(\gamma_0,\gamma_1)$ for $\pi$-a.e. $\gamma \in \Geo(\X)$ and, in addition,
\begin{equation}\label{eq:metrbren}
|\nabla\varphi|(x)=|D^+ \varphi| (x)= L(x) \quad \text{$\mm$-a.e.\ $x$.}
\end{equation}
It follows that for $\pi$-a.e. $\gamma \in\Geo(\X)$
$$|\nabla\varphi|(\gamma_0)=\sfd(\gamma_0,\gamma_1)=\frac{\sfd(\gamma_0,\gamma_t)}{t} \quad
\text{$\forall t\in (0,1)$.}$$ 
It remains to show \eqref{derKant}.
By optimality,  for $\pi$-a.e. $\gamma$,
$$\varphi(\gamma_0)+\varphi^c(\gamma_1)=\frac{\sfd^2(\gamma_0,\gamma_1)}{2}, \qquad \varphi(\gamma_t)
+\varphi^c(\gamma_1)\leq \frac{\sfd^2(\gamma_t,\gamma_1)}{2}  ,$$
which gives
$$\varphi(\gamma_0)-\varphi(\gamma_t)\geq \frac{1-(1-t)^2 }{2}\sfd^2(\gamma_0,\gamma_1)
=\frac{2t-t^2}{2} \sfd^2(\gamma_0,\gamma_1).$$
Dividing both sides by $\sfd(\gamma_t,\gamma_0)=t\sfd(\gamma_1,\gamma_0)$, we find that for $\pi$-a.e. $\gamma$ it holds
\begin{equation}\label{liminfKant}
\liminf_{t \downarrow 0} \frac{\varphi(\gamma_0)-\varphi(\gamma_t)}{\sfd(\gamma_0,\gamma_t)}\geq
 \sfd(\gamma_0,\gamma_1)=|\nabla\varphi| (\gamma_0).
\end{equation}
On the other hand, by definition of ascending slope 
\begin{equation}\label{limsupKant}
\limsup_{t\downarrow 0} \frac{\varphi(\gamma_0)-\varphi(\gamma_t)}{\sfd(\gamma_0,\gamma_t)}\leq 
|D^+ \varphi|(\gamma_0).
\end{equation}
So, combining \eqref{eq:metrbren} and \eqref{liminfKant} with \eqref{limsupKant} we get 
\begin{equation}\label{eq:ConvKanta.e.}
\lim_{t \downarrow 0} \frac{\varphi(\gamma_0)-\varphi(\gamma_t)}{\sfd(\gamma_0,\gamma_t)} = |\nabla\varphi|(\gamma_0) \quad 
\text{$\pi$-a.e. $\gamma$.}
\end{equation}
Next we will show that
\begin{equation}\label{eq:weakConvKant}
\frac{\varphi(\gamma_0)-\varphi(\gamma_t)}{\sfd(\gamma_0,\gamma_t)} \rightharpoonup |\nabla\varphi|\circ\e_0
\qquad \text{weakly in } L^2(C([0,1];\X), \pi).
\end{equation}
Since by assumption $|\nabla \varphi|$ is bounded, by  \eqref{Eq:ineqincremental}  we have
\begin{equation}\label{eq:part1present}
\limsup_{t\downarrow 0}\int \left |\frac{\varphi(\gamma_0)-\varphi(\gamma_t)}{\sfd(\gamma_0,\gamma_t)}\right |^2\,\d\pi\leq
\int |\nabla\varphi|^2(\gamma_0)\,\d\pi. 
\end{equation}
If $\Psi$ is a limit point in the weak $L^2$ topology of the left-hand side of \eqref{eq:part1present} as $t\downarrow 0$, by Mazur's lemma there exists a sequence of convex combinations of 
these difference quotients strongly converging in $L^2(C([0,1];\X)$ to $\Psi$. Since a further subsequence converges $\pi$-a.e., from 
\eqref{eq:ConvKanta.e.} we obtain that $\Psi=|\nabla\varphi|(\gamma_0)$. By weak compactness, since the only possible weak limit is $|\nabla\varphi|(\gamma_0)$, we obtain \eqref{eq:weakConvKant}.
Finally, we conclude by observing that \eqref{eq:part1present} ensures the convergence of the $L^2(C([0,1];\X), \pi)$ norms. Combined with \eqref{eq:weakConvKant}, this yields the desired strong convergence.
\end{proof}
\medskip 
\noindent\textbf{Acknowledgements}. F.N., F.R. and F.V. are members of INDAM-GNAMPA. F.N acknowledges partial support by the European Union (ERC ConFine, 101078057), the INdAM-GNAMPA Project ``Analisi e Gamma-convergenza per alcuni funzionali non locali'' CUP E53C25002010001\#, and the MIUR Excellence Department Project awarded to the Department of Mathematics, University of Pisa, CUP I57G22000700001.
L.A., F.R. and F.V. acknowledge support by the Balzan Prize 2019. Part of this work was carried out when F.V. was visitor at TUM. F.V. wishes to thank G. Friesecke and the Department of Mathematics of TUM for the warm hospitality and the stimulating atmosphere, and the MIUR-PRIN 202244A7YL project ``Gradient Flows and Non-Smooth Geometric Structures with Applications to Optimization and Machine Learning" for financially supporting his stay in Munich.

\end{document}